\documentclass[12pt]{article}
\usepackage{bm,amssymb,amsmath, amsthm}

\usepackage[section]{placeins}
\usepackage{graphicx}
\usepackage{subfig}
\usepackage{caption}

\usepackage{color}

\begin{document}

\newtheorem{lemma}{Lemma}
\newtheorem{theorem}{Theorem}
\newtheorem{corollary}{Corollary}
\newtheorem{definition}{Definition}
\newtheorem{remark}{Remark}

\title{Two-dimensional grain boundary networks: stochastic particle models
and kinetic limits }
\author{Joseph Klobusicky, Govind Menon, and Robert L. Pego}
\date{}
\maketitle

\date{}
 \maketitle

\begin{abstract}
We study kinetic theories for isotropic, two-dimensional grain boundary networks
which evolve by curvature flow.  The number densities $f_s(x,t)$ for $s$-sided
grains, $s =1,2,\ldots$,  of area $x$ at time $t$, are modeled by kinetic
equations of the form $\partial_t f_s + v_s \partial_x f_s =j_s$. The velocity
$v_s$ is given by the Mullins-von Neumann 
rule and the flux $j_s$ is determined by the topological transitions caused
by the vanishing of grains and their edges. 
The foundations of such kinetic models are examined through simpler particle
models for the evolution of grain size, as well as purely topological models
for the evolution of trivalent maps. These models are used to characterize
the parameter space for the flux $j_s$. Several kinetic models in the literature,
as well as a new kinetic model, are simulated and compared with direct numerical
simulations of mean curvature flow on a network. Existence and uniqueness
of mild solutions to  the kinetic equations with continuous initial data
is established. 

\end{abstract}

\section{Introduction } 
\label{sec:intro}
  
\subsection{Two dimensional grain boundary networks}
We propose a new class of kinetic and stochastic models to describe the statistics
of an evolving cellular network. We focus on the evolution of an isotropic,
two dimensional grain boundary network consisting of smooth arcs such that:

(i) the normal velocity of each arc is proportional to its curvature ({\em
curvature flow}); (ii) edges (typically) meet at trivalent junctions at
an angle of $2\pi/3$ (the {\em Herring boundary condition})~\cite{Herring}.
 Such a network of arcs decomposes the plane into a disjoint collection of
grains, each of which have the topology of polygons. Condition   (ii) expresses
the equilibrium of line tensions at a junction.

An important aspect of grain boundary evolution is the celebrated von Neumann-Mullins
relation~\cite{von1952discussion,Mullins}: the area $a(t)$ of a grain with
$s$ sides (a topological $s$-gon) changes linearly in time
\begin{equation}
\label{eq:von-neumann}
\frac{da}{dt} = c (s-6), 
\end{equation}
where $c$ is a material constant depending on surface tension and grain mobility.
Thus, the geometry of each grain does not affect its growth, all that matters
is the topology. In this setting, the statistics of a network with many grains
are naturally described by a set of {\em number densities\/} $f_s(t,a)$ that
count the number of   $s$-sided  cells per unit area that have area $a$ at
time $t$. 

We  derive kinetic equations that describe the evolution of $f_s(t,a)$ from
a simpler stochastic particle system that includes a deterministic drift
(as in equation~\ref{eq:von-neumann}) along with stochastic `switching'
rules  between populations based on the geometry of grain boundary networks.
This particle system is an instance of a Piecewise Deterministic Markov Process
(PDMP).  Several similar kinetic theories have been proposed in the literature
(as discussed in Section~\ref{sec:kinmodels}) on the basis of ad hoc rules,
or comparison with experiments. However, there appears to have been no prior
attempt to characterize the set of all possible kinetic models that may be
derived from similar foundations (the von Neumann-Mullins rule and assumptions
on the topological changes that arise when grains or grain boundaries vanish);
nor does there appear to have been a prior attempt to compare the predictions
of kinetic models with direct numerical simulations of grain boundary evolution.

\subsection{Outline}
from the geometry of grain boundary networks, a well-posedness analysis of
the limiting kinetic equtions, and simulations which compare several previous
kinetic models.
This article is organized into three inter-related, but loosely dependent
parts:
\begin{enumerate}
\item A general framework and well-posedness analysis for formal kinetic
limits of a  class of piecewise deterministic Markov  processes related to
grain boundary coarsening
 (Section  \ref{sec:model}, Appendices \ref{app:pdmpexplain} and \ref{sec:wp}).
\item Derivation of new stochastic and kinetic models for grain boundary
coarsening from topological rules
(Sections \ref{sec:tops}-\ref{sec:tunableparam}).\item Simulation of  stochastic
particle models that correspond to both new and existing kinetic models,
with  comparison to direct numerical simulations of a level set method (Section
\ref{sec:results}). 
\end{enumerate}       
In technical terms, the first part of this paper is most closely tied to
fluid-limits in queuing theory and the theory of piecewise-deterministic
Markov processes. It can be viewed as a demonstration of the utility of these
methods for cellular networks. The stochastic process studied here, an $M$-species
PDMP model, is a general model for particles that drift on the positive
real line and mutate between  several species.  An interesting feature of
this model is the set of mutation times, which are caused by particles reaching
the origin.
From the perspective of maps on surfaces, this corresponds to a face or edge
collapsing to
a point, and immediately changing its topology to satisfy the Herring conditions.
Each mutation, from predetermined model parameters,  induces randomness on
particle positions, and thus
makes mutation times random.  Understanding the behavior of the system as
a whole largely depends on describing the cumulative  number of mutations,
a `natural clock' for the system. The interplay between mutation
and empirical particle densities was rigorously studied in \cite{klobusicky2017concentration}
with a minimal example, where
authors JK and GM viewed the problem as an instance of a diminishing urn,
and provided exponential concentration inequalities for the convergence of
the particle model to its hydrodynamic limit.

While the first part focuses on the evolution of area statistics, with topological
restrictions arising only in the  description of boundary fluxes, the second
part of this paper is devoted to a study of the `topological skeleton' of
grain boundary evolution. The analysis of annihilation and creation of grains
is examined using the theory of maps on compact surfaces. These ideas are
used to define a Markov chain on the space of trivalent maps ({\em trivalent
map evolution\/}). We do not explore such Markov chains in detail. Instead,
we use trivalent map evolution to systematically derive parameters for the
stochastic particle system model.

In the third part, we  provide a numerical comparison between stochastic
particle simulations and direct numerical simulations of grain boundary networks
using a level set method~\cite{Elsey1}. The stochastic particle models  described
in Section \ref{sec:model} are general enough to include assumptions from
previous  kinetic models  \cite{Fradkov1,Fradkov2,flyvbjerg1993model,marder1987soap}.
We also  consider a new assumption  which takes the rate of topological changes
due to edge deletion events to be proportional to the total grain number.
Our models also allow us to incorporate information about first-neighbor
correlations in networks.

\section{The stochastic particle system and kinetic equations} 
\label{sec:model}

In Section \ref{subsec:finitemodel}, we will  describe a class of particle
processes
which are amenable for modelling the coarsening, growth, and  mutation found
in cellular coarsening. These processes are   examples of  {\em piecewise
deterministic Markov processes\/}
(PDMPs). Roughly, the theory of PDMPs augments the structure of jump
Markov processes to include random jumps triggered by deterministic drift.
It is shown in
Appendix~\ref{app:pdmpexplain} that the particle system defined informally
below generates a well-defined evolution which is a strong Markov process.
In Section \ref{subsec:kineqns}, we present  kinetic limits for PDMP models
and state a well-posedness theorem for this limit, whose proof is provided
in Appendix \ref{sec:wp}.

\subsection{The finite particle  model } \label{subsec:finitemodel}

\label{subsec:model} 
We consider a system of $N(t)$ particles at time $t$ distributed amongst
$M$ species. Each particle is of the form $(s,x)$ where $s\in\{ 1, \ldots,
M\}$ indexes the species and $x \in \mathbb{R}_+$ denotes the size of the particle.
The total number of particles in each species is denoted $N_s(t)$, thus 
$N(t)=\sum_{s=1}^M N_s(t)$. The letter $N$ (without the argument $t$) is
always used to mean $N(0)$, and is a measure of the size of the system. The
state of the system is denoted 
\begin{equation}
\label{eq:state}
(\mathbf{s},\mathbf{x}) = (s_1, \ldots, s_{N(t)}; x_1, \ldots, x_{N(t)}).
\end{equation}  
The evolution of the system consists of a deterministic flow interspersed
with stochastic jumps. We describe these in turn.   

The deterministic flow is motivated by the von Neumann-Mullins rule. We divide
the species into three distinct groups $S_-$, $S_0$ and $S_+$ of size $M_-$,$M_0$
and $M_+$ respectively, with $M= M_-+M_0+M_+$. It is convenient to label
these species in order:
\begin{align}
\label{eq:species-order}
S_-&=\{1, \ldots, M_-\},\\  S_0 &= \{ M_-+1, \ldots, M_-+M_0\},\\  S_+&=\{M_-+M_0+1,
\ldots, M\}. 
\end{align}
For each species $s \in S_-$, we assume given a constant velocity  $v_s <0$,
so that a particle of size $x$ at $t=0$ has size $x+v_st$ at time $t>0$.
The  {\em exit time\/} for the particle $(s,x) $ is the time at which the
size of the particle vanishes, namely
\begin{equation}
\label{eq:first-hit}
T_s(x) =- \frac{x}{v_s}, \quad s \in S_-.
\end{equation}
We assume that the species $s \in S_0$ do not drift. That is, $v_s = 0$ for
$s \in S_0$. Finally, we assume $v_s >0$ for $s \in S_+$. The exit time for
all particles of species $S_0$ and $S_+$ is $+\infty$.

Randomness is introduced into the system in the following way. As $t$ increases,
each particle $(s,x)$ in the system drifts  deterministically  $(s,x) \mapsto
\left(s,\varphi_s(x,t)\right)$  where $\varphi_s$ is the flow map defined
by $x \mapsto \varphi_s(x,t) = x+v_st$ for  $s  = 1, \dots, M$. Particles
evolve until one of the following  {\em critical events} occur:

\begin{enumerate}                                     
\item[(B)] 
{\em Boundary event:\/} A particle hits the origin, i.e.  $\varphi_s(x,t)=0$
for some $(s,x)$ with $s \in S_-$ and $t=T_s(x)$.

\item[(I)]
{\em Interior event:} An independent Poisson clock with rate $\beta(t)>0$
attached to each particle rings. 
\end{enumerate} 
To fix ideas, we illustrate these definitions in the context of grain boundary
networks. Here the state of the system is a collection of $N(t)$ grains,
each belonging to one of $M$ topological classes; $s$ denotes the number
of sides of a grain and $x$ denotes its area. The velocity field $v_s$ governing
the evolution of an $s$-gon is  given by the von Neumann-Mullins rule~\ref{eq:von-neumann}.
   The critical events correspond respectively to: (B) the removal of a grain
from the network when its size shrinks to zero; and (I) a random interchange
of grains of different topology  when an edge  vanishes.    
    
Though the size of each particle evolves deterministically, each boundary
and interior event gives rise to a random mutation of particles of different
species.  We model each mutation with a {\em mutation matrix\/}. There are
$M_-+1$ such matrices: $M_-$ matrices corresponding to the $M_-$ possible
boundary events at each species $l \in S_-$,  and one matrix for interior
events. We find it necessary to include mutations in such generality to account
for the topology of cellular networks -- the topological changes arising
from the vanishing of  $3$, $4$ and $5$ sided grains in grain boundary networks
is {\em not\/} the same. Aside from some notational complexity, such generality
does not affect the analysis. In Section \ref{sec:tops}, we explain how to
choose the mutation rules based on the geometry of planar grain boundary
networks.

Consider the boundary event when a single particle of species $l$ hits the
origin.
The corresponding mutation is determined by a positive integer $K^{(l)}$,
an  $M \times K^{(l)}$  mutation matrix $R^{(l)}$ taking values in $1,\ldots,
M$, and a fixed $M$-vector $w^{(l)}$ with positive entries. We choose $K^{(l)}$
particles and mutate them as follows. First, $K^{(l)}$ iid integers $S_1,
\ldots, S_{K^{(l)}}$ that index species are chosen with  probability proportional
to the weights $w^{(l)}$ and the total population of each species:
\begin{equation}   
        \label{eq:flip-probs}
\mathbb{P} \left(S =\sigma\right) = \frac{w^{(l)}_\sigma N_\sigma(t)}{\sum_{n=1}^M
w^{(l)}_n  N_n(t)},   \quad \sigma = 1, \dots, M.
\end{equation}
Second, for each random species $S_j$, a random size $X_j$ is chosen with
equal
probability $1/N_{S_j}(t)$ amongst the sizes of all the particles of species
$S_j$. Finally, these random particles are mutated as follows:
\begin{equation}
\label{eq:mutation}     
\left(S_j, X_j\right) \longmapsto \left(R^{(l)}_{S_j,j},X_j\right), \quad
 j = 1, \ldots, K^{(l)}. 
\end{equation}
Thus, a particle of species $S_j$ with size $X_j$ is lost, and a particle
of species $R^{(l)}_{S_j,j}$   with the size $X_j$ is created in the mutation.
A particle of species $l$ with  size 0  is also lost, so that the total number
of the system decrements by one. Note that selection probabilities and mutations
may vary for each of the $K^{(l)}$ mutating particles. See Figure \ref{pdmppic}
for an example with four species. In the degenerate event that the sizes
of $p$  species, $p>1$, hit the origin simultaneously, we repeat the process
above $p$ times, ordering the boundary events at species $l_1, \ldots, l_p$
in the sequence $l_1 \leq l_2 \leq \ldots \leq l_p$ to be definite. Such
`collisions'  occur with zero probability in the kinetic limit.

The process of mutation at an interior event is similar. No particle vanishes,
but particles are mutated according to a fixed positive integer $K^{(0)}$,
a mutation matrix $R^{(0)}$ and weight $w_s^{(0)}$ as above. The integers
$S_1, \ldots, S_{K^{(0)}}$ are chosen with probability  
\begin{equation}   
        \label{eq:flip-probs-int}
\mathbb{P} \left(S =\sigma\right) = \frac{w^{(0)}_\sigma N_\sigma(t)}{\sum_{n=1}^M
w^{(0)}_n N_n(t)},   \quad \sigma = 1, \dots, M,
\end{equation}
and the particles mutated as follows
\begin{equation}
\label{eq:mutation-int} 
\left(S_j, X_j\right) \longmapsto \left(R^{(0)}_{S_j,j},X_j\right),
\quad  j = 1, \ldots, K^{(0)}. 
\end{equation}

In the context of grain boundary coarsening, the necessity of introducing
 randomness for selecting grains results from a mean field assumption.  Specifically,
 our  models will track individual grain areas and number of sides, but 
not information about which grains neighbor each other. To determine  
mutated grains at a critical event, we select  randomly according to equations
(\ref{eq:flip-probs})-(\ref{eq:mutation-int}), which impose that grains with
identical  topologies have equal chances of being selected to mutate, regardless
of their areas. 

In order to keep track of the flux in and out of a species  we define the
(constant) matrices with integer entries
\begin{equation}  
\label{eq:Jdef}
J^{(l)}_{s, \sigma  } = \sum_{j= 1}^{K^{(l)}}  \mathbf{1}_{\{R^{(l)}_{sj}=\sigma\}},
\quad 
J^{(0)}_{s, \sigma }   =   \sum_{j= 1}^{K^{(0)}} \mathbf{1}_{\{R^{(0)}_{sj}
= \sigma\}}.
\end{equation}
Here $s$ and $\sigma$ index species in $\{1,\ldots, M\}$, $l$ indexes a species
in $S_-$, and the entry $J^{(l)}_{s,\sigma}$ counts the total number of mutations
from species $s$ to species $\sigma$ when a particle of species $l$ hits
the origin. Similarly,  $J^{(0)}_{s,\sigma}$ enumerates the total number
of mutations from species $s$ to species $\sigma$ at an interior event. We
assume that there are no trivial mutations from a species to itself, i.e.,
\begin{equation}  
\label{eq:Jdeftrivial}
J^{(l)}_{\sigma \sigma}  = 0, 
\quad 
J^{(0)}_{\sigma \sigma} =   0.
\end{equation}
Summing over all species we obtain the identities
\begin{equation}
\label{eq:Jequality}
\sum_{\sigma=1}^M J^{(l)}_{s,\sigma} = \sum_{j=1}^{K^{(l)}} \sum_{\sigma=1}^M
\mathbf{1}_{R^{(l)}_{sj}=\sigma} = \sum_{j=1}^{K^{(l)}} 1 = K^{(l)}, \quad \sum_{\sigma=1}^M
J^{(0)}_{s,\sigma} =K^{(0)}.
\end{equation}
However, we do not assume detailed balance of mutations between species.
That is, in general,
\begin{equation}
\label{eq:nodetail}
\sum_{s=1}^M J^{(l)}_{s,\sigma}  \neq K^{(l)}, \quad \sum_{s=1}^M J^{(0)}_{s,\sigma}
 \neq K^{(0)}.
\end{equation}

\begin{figure}
\begin{centering}
\includegraphics[width=.5\textwidth]{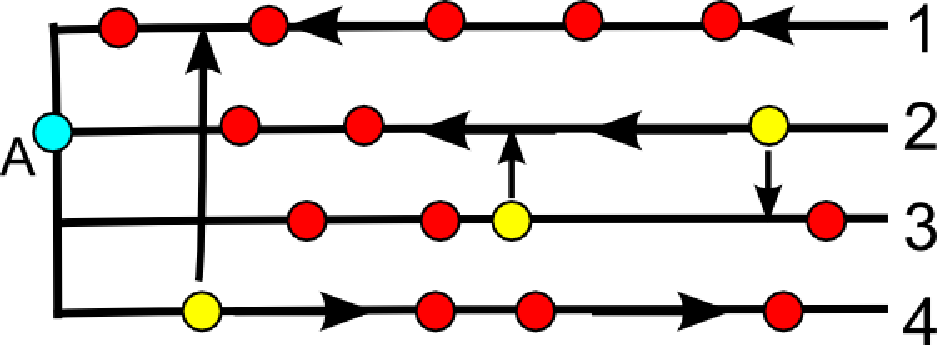}
\caption{\textbf{A PDMP with four species}.  Particles travel on four separate
copies of $\mathbb{R}_+$.  Velocity directions are represented by horizontal arrows,
with species 3 having zero velocity. A boundary event occurs when a particle
(labelled by ``A") hits the origin.  Three particles  are then randomly selected
($K^{(2)} = 3$), and reassigned  to different species by predetermined reassignments
(given by vertical arrows). In this example, $R_{41}^{(2)} = 1, R_{32}^{(2)}
= 2$, and $R_{23}^{(2)} = 3$. }
\label{pdmppic}
\end{centering}
\end{figure}

\subsection{Kinetic limits of finite particle model}  \label{subsec:kineqns}
Several  kinetic equations arise as $N\to \infty$ limits of Markovian particle
models. We now apply this approach to the particle models of Section 2.1

For each state $(\mathbf{s}(t),\mathbf{x}(t))$ and species $\sigma \in \{1, \ldots, M\}$
we define an empirical measure                  
\begin{equation}
        \label{eq:empirical}
\mu_\sigma^N(t)= \frac{1}{N}\sum_{i=1}^{N(t)} \mathbf{1}_{\{s_i =\sigma\}} \delta_{x_i}.
\end{equation}   
The empirical measures are normalized by the fixed initial number $N=N(0)$,
not $N(t)$. Thus, $\sum_{\sigma=1}^N\mu_\sigma^N(t)$ is in general  not
a probability measure for $t >0$. In what follows,
we will not be interested in a rigorous demonstration of the convergence
of empirical measures to solutions of kinetic equations
(see \cite{klobthesis} for a compactness argument for the convergence of
empirical measures to kinetic equations).  Rather, we begin   with the assumption
that for each species $\sigma$ the weak limit $\mu_\sigma(t)
= \lim_{N\to \infty} \mu_\sigma^N(t)$ is deterministic and has a number
density 
\begin{equation}
\mu_\sigma(t)(dx) = f_\sigma(x,t)\, dx. 
\end{equation}
With these densities, we give a formal argument for kinetic equations. In
Appendix \ref{sec:wp}, we show well-posedness for these kinetic models, along
with  some expected conservation properties  
 for special instances related to grain-boundary coarsening. 

In the continuum limit, we can define the total numbers of particles in $\mu_\sigma$,

\begin{equation} 
        \label{eq:numbers-not1}
F_\sigma(t)= \int_0^\infty f_\sigma(x,t)\, dx, \quad  F(t) = \sum_{\sigma=1}^M
F_\sigma(t),
\end{equation} 
and the weighted fractions
\begin{equation}
        \label{eq:numbers-not2}
W_\sigma^{(l)}(t)= \frac{w^{(l)}_\sigma}{\sum_{n=1}^M w^{(l)}_n F_n(t)},  \quad
\gamma(t)= \frac{F(t)}{\sum_{n=1}^M w_n ^{(0)} F_n(t)}.    
\end{equation}
Then for each species $\sigma \in \{1, \ldots, M\}$, the formal kinetic equations
for the number density $f_\sigma$ are
\begin{eqnarray} 
\lefteqn{\partial_t f_\sigma(x,t)+ v_\sigma \partial_{x}f_\sigma (x,t)= j_\sigma
:= j_\sigma^+(x,t) -j_\sigma^-(x,t),} \label{eq:kinetic}
\\  
\label{eq:jplus}
j_\sigma^+(x,t) & = & \sum_{s = 1}^{M} \left( \sum_{l = 1}^{M_-} \dot{L_l}
 J^{(l)}_{s,\sigma} W_s^{(l)}(t)  
+\beta (t)\gamma(t)  J^{(0)}_{s,\sigma} w^{(0)}_s \right) f_s(x,t), \\
\label{eq:jminus}
j_\sigma^-(x,t) & = & \left( \sum_{l = 1}^{M_-} \dot{L_l} K^{(l)} W_\sigma^{(l)}(t)
+  \beta (t)\gamma(t) K^{(0)}  w^{(0)}_\sigma \right) f_\sigma(x,t), \\
\label{eq:fluxdef}
\dot{L}_l &=& - f_l(0,t)v_l, \quad l = 1, \ldots, M_-.
 \end{eqnarray}
While perhaps cumbersome at first sight, equation \ref{eq:kinetic} is easily
understood as a formal hydrodynamic limit of the $N$ particle PDMP described
in the previous section.
The index $\sigma$ denotes a fixed species under consideration. The left-hand
side of \ref{eq:kinetic} describes the advection of the number density $f_\sigma$
under the constant velocity $v_\sigma$. The right-hand side describes the
growth and loss of species $\sigma$ due to fluxes $j_\sigma^{\pm}$ into and
out of species $\sigma$. The fluxes  in equations \ref{eq:jplus} and \ref{eq:jminus}
arise from interior and boundary events.
In these equations, the index $l$ enumerates all possible boundary events,
and the index $s$ enumerates all the species that could mutate to species
$\sigma$. A boundary event for  species $l \in S_-$ gives rise to both birth
and death terms in proportion to the rate $\dot{L_l}= -f_l(0,t)v_l(0)$ and
the weights $W_s^{(l)}(t)f_s(x,t)$. The  weights $J^{(l)}_{s,\sigma}$ and
$J^{(0)}_{s,\sigma}$ defined in equation~\ref{eq:Jdef} arise as we sum
over all mutations that lead to the creation of particles of species $\sigma$
of size $x$ when a particle of species $l$ hits the origin. Similarly, such
particles may be lost when they are mutated. This occurs in proportion to
the weight $W_\sigma^{(l)}(t)$. The terms multiplied by the rate $\beta (t)\gamma(t)$
account for interior events.
  
The boundary values $f_l(0,t)$, for the outgoing species $l \in S_-$ play
a subtle role in the kinetic equation since they determine the rate of boundary
events. In order to obtain well-posedness of the kinetic equations, we will
assume that the  number densities are continuous on $[0,\infty)$ so that
there is no ambiguity in defining their boundary values. In contrast, the
boundary value of the species $l \in S_0,S_+$ do not affect the flux and
we impose the boundary conditions
\begin{equation}
\label{eq:kinetic-BC}
f_l(0,t) =0, \quad l= M_-+1, \ldots, M.
\end{equation}  

\subsection{Well-posedness}
The kinetic equations \ref{eq:kinetic} admit mild solutions on a maximal
interval of existence. In order to define mild solutions, we integrate \ref{eq:kinetic}
along characteristics for each species $\sigma$ to obtain
\begin{equation}
\label{eq:uq1}
f_\sigma(x,t) = f_\sigma(x-v_\sigma t, 0) + \int_0^t j_\sigma\left(x-v_\sigma(t-\tau),
\tau\right) \, d\tau.
\end{equation}
Here we assume that $x \geq 0$, $t >0$. Thus, formula \ref{eq:uq1} is well
defined for all species with $v_\sigma \leq 0$, i.e. for $\sigma \in S_-,S_0$.
For the species with $v_\sigma>0$ we must use the boundary condition \ref{eq:kinetic-BC}
and a priori the integral in time is defined only over the time domain $\tau
\in [x/v_\sigma,t]$. However, for convenience, we extend the formula \ref{eq:uq1}
to include the domain $\tau \in [0,t]$ by setting  $f_\sigma(x,\tau) = j_\sigma(x,\tau)
= 0$ when $x \leq 0$. It is then clear that \ref{eq:uq1} agrees with the
solution obtained from the method of characteristics and the boundary condition
\ref{eq:kinetic-BC}.

Let $X$ denote the space of continuous and integrable functions $f=(f_1,
\ldots, f_M) :[0,\infty) \to \mathbb{R}^M$ equipped with the norm
\begin{equation}
\label{eq:def-norm}
\|f\| :=  \| f\|_{L^1}  + \| f \|_{L^\infty}, \quad
\|f\|_{L^1} := \sum_{\sigma=1}^M \| f_\sigma \|_{L^1}, \quad \|f\|_{L^\infty}
:= \sum_{\sigma=1}^M \| f_\sigma \|_{L^\infty}. 
\end{equation}
It is easy to check that $X$ is a Banach space. We also denote 
\begin{equation}
\label{eq:def-F}
F_\sigma = \int_0^\infty f_\sigma(x) \, dx, \quad F = \sum_{\sigma=1}^M F_\sigma.
\end{equation}
We say that $f \in X$ is {\em positive\/} if $f_\sigma(x) \geq 0$ for each
 $\sigma$ and each $x \geq 0$. When $f$ is positive, $F =  \|f  \|_{L^1}$.

\begin{definition}
\label{def:class-soln}  
Assume $T >0$ is given.  A map  $f\in C([0,T];X)$ is a mild solution to \ref{eq:kinetic}
if \ref{eq:uq1} holds for $x \in [0,\infty)$ and $t \geq 0$. We say that
$f$ is a positive  mild solution if $f(t)$ is positive for each $t \in [0,T]$.
\end{definition}

\begin{theorem}
\label{thm:wp}  
Assume given positive $f_0 \in X$, and $\sum_{n=1}^M w^{(l)}_n
F_n>0$ for  $l = 0, \dots, M_-$. Also assume $\beta(t) \equiv \beta$ is constant.
There exists a (possibly infinite) time $T_*>0$ and a unique map $f \in 
C([0,T_*); X)$ with $f(0)=f_0$ such that $f$ is a positive, mild solution
to \ref{eq:kinetic} on each interval $[0,T]$ with $0 < T < T_*$. 
Further, $\lim_{t \to T_*}\sum_{n=1}^M w^{(l)}_n F_n(t) =0$ if $T_* < \infty$.
\end{theorem}
The proof of Theorem~\ref{thm:wp} is presented in Appendix~\ref{sec:wp}.
Also shown for kinetic equations related to grain coarsening  is a proof
of  global existence  and  conservation of total area and zero polyhedral
defect. Note that the time $T_*$ may be finite as shown in Figure \ref{fig:counterpdmp}.

\section{Topological evolution of trivalent maps}
\label{sec:tops}

In this section we characterize the admissible topological changes during
grain boundary evolution on a compact surface $S$. In Section \ref{sec:mspecgrain},
these changes will be incorporated into our PDMP model of grain boundary
coarsening through the mutation matrix $R^{(l)}$. 

Recall that a map is a graph drawn on a surface. More precisely, a map is
a graph embedded in  a surface in such a manner that (i) all vertices are
distinct; (ii) none of the edges intersect except at vertices; (iii) each
face obtained by cutting the surface along its edges is homeomorphic to an
open disk~\cite[Ch.1.3]{Lando}. The topology of a grain boundary network
on  $S$ is that of a trivalent map on $S$. We denote the set of maps by $\mathfrak{M}(S)$
and the set of trivalent  maps by $\mathfrak{T}(S)$. 

The topology of a grain boundary network stays constant when the evolution
is smooth, but jumps when a grain or grain boundary vanishes. Thus, the topology
of the network is a map $M: [0,\infty) \to \mathfrak{T}(S)$ that is piecewise
constant. We adopt the convention that $M$ is right continuous. At a jump
at time $t$, $M(t)$ is obtained from $M(t_-)$ by a surgery consisting of
the contraction of a face or edge of $M(t_-)$ to a vertex $v_*$ followed
by the attachment of a new graph $T$ at $v_*$. The topology of $T$ must be
consistent with the face or edge of $M(t_-)$ that was contracted, as well
as the irreversibility condition that energy cannot increase. We show below
that these restrictions imply that $T$ must be a planar rooted tree.

The remainder of this section is organized as follows. We first establish
the topological restrictions on attachment in Theorem~\ref{topthm} below.
This is followed by the definition of a Markov chain on $\mathfrak{T}(S)$ that
may be used to model the `topological skeleton' of grain boundary evolution.
Finally, we apply these topological restrictions to obtain the mutation matrices
for grain boundary evolution, thus connecting the size-based kinetic theory
of Section~\ref{sec:model} to topological changes.

\subsection{The topology of attachment}
\label{subsec:attachment}
We will restrict our analysis of detachment and attachment to the situation
where at most one face shrinks to zero size or at most one edge shrinks to
zero length. When a $k$-sided face, $k \leq 5$, of  a trivalent map $M \in
\mathfrak{T}(S)$ vanishes we obtain a map with a $k$-valent vertex. In a similar
manner, if a single edge in a trivalent map shrinks to zero length, we obtain
a $4$-vertex. 

It is convenient to introduce some notation for this process. Given $M \in
\mathfrak{T}(S)$, let 
$\mathcal{O}(M)$ denote the set of maps that are obtained from $M$  by contracting
a $k$-sided face ($k \leq 5$) or an edge to a single vertex. Thus, every
vertex in $N \in \mathcal{O}(M)$ is trivalent, except for possibly one distinguished
vertex, denoted $v_*$, which is $k$-valent, $1 \leq k \leq 5$. Let $e_1,
\ldots, e_k$ denote the edges in $N$ that are incident to $v_*$.

The process of attachment is a mapping from $\mathcal{O}(M)$ to $\mathfrak{T}(S)$
that involves `blowing-up' the vertex $v_*$ into a map $T$ and attaching
$T$ to $N$ by the edges $e_1, \ldots, e_k$. The possible choices for $T$
are constrained by the following criterion:
\begin{enumerate}
\item[(i)]  $T$ is a map into the closed unit disk $B$ with $k$ vertices
on the boundary $\partial B$. We denote these vertices $\{ p_1,\ldots, p_k\}$
in cyclic order on $\partial B$.
\item[(ii)] All vertices of $T$ in the interior of $B$ are trivalent. 
\item[(iii)] All interior faces of $T$ (i.e. faces whose edges lie in the
interior of $B$) have at least seven sides.
\end{enumerate}
Condition (i) is the boundary condition necessary so that $T$ can be attached
to $N$ to obtain a trivalent map on $S$. The vertices $\{p_j\}_{j=1}^k$ are
temporary placeholders that are deleted when we attach $T$ to $N$; their
role is to ensure that $T$ can be connected to $N$ by the edges $e_1, \ldots,
e_k$. The other two conditions are consequences of irreversibility. Condition
(ii) says that vertices that are not trivalent cannot persist in grain boundary
evolution. Condition (iii) says that it is impossible to nucleate grains
with six or fewer sides since this would contradict the Mullins-von Neumann
relation. These criteria sharply limit the possible choices of $T$.

\begin{theorem} 
\label{topthm} 
Suppose $k \leq 5$ and $T \in \mathfrak{M}(B)$ satisfies the conditions listed
above. Then $T$ is a tree.
\end{theorem}
\begin{corollary}
\label{topcor}  
When $k=2$, $3$, $4$ or $5$ respectively, there are $1$, $1$, $2$ or $5$
topologically distinct choices for $T$ as shown in Figure~\ref{crittypes}.

\end{corollary}
\begin{figure}
\centering
\includegraphics[width=.5\linewidth]{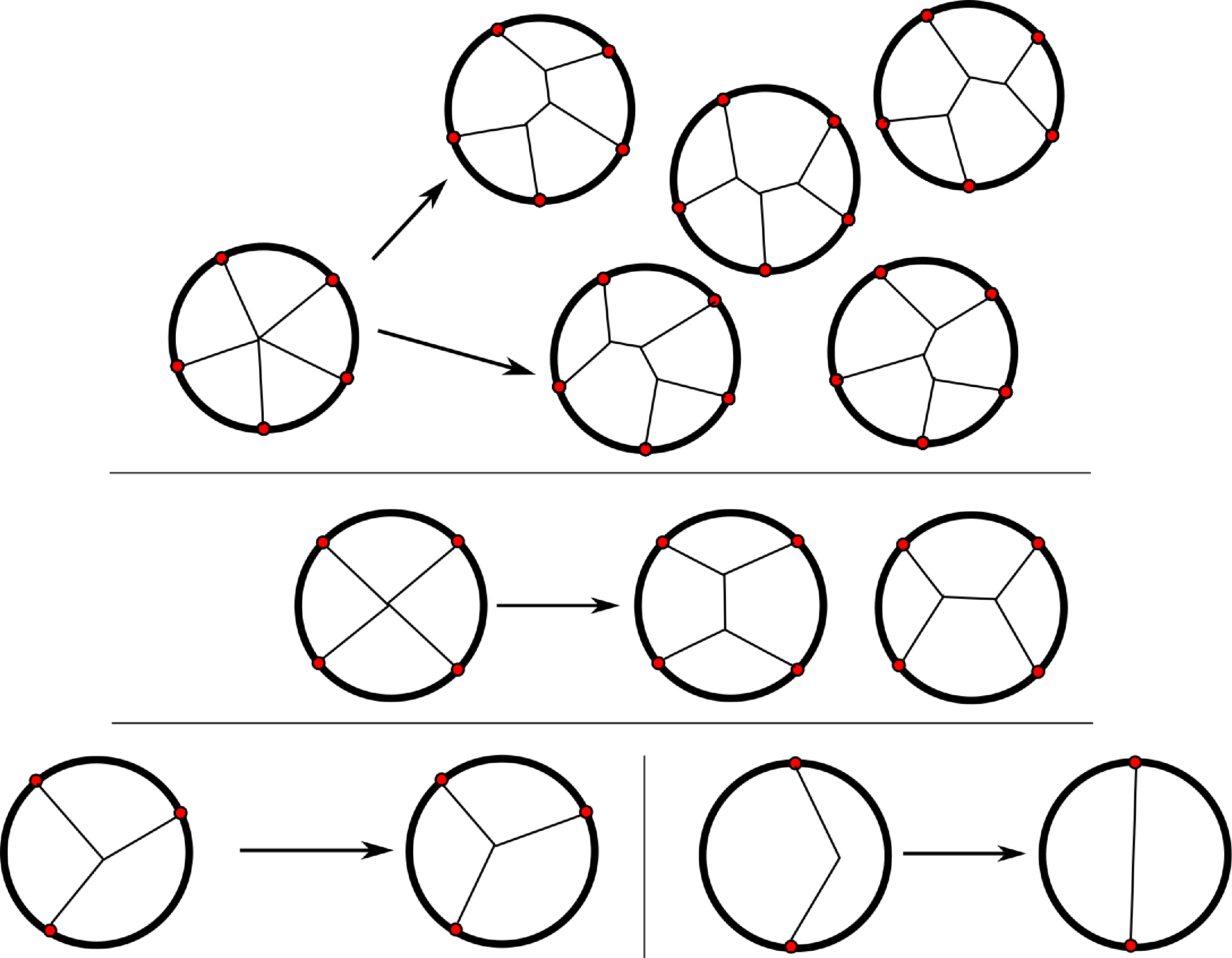}
\caption {\textbf{Continuation through a $k$-degree vertex. }\label{crittypes}
All trivalent trees embedded in the unit disk with $k$ labeled vertices on
the boundary for  
$k=2,  3, 4,$ and $5$. This set is in bijective correspondence with planar
rooted trivalent trees with $k$ leaves.  } 
\end{figure}

\begin{proof}[Proof of Theorem~\ref{topthm}]
We may trivially extend $T$ to a map into $S^2$. Suppose that  $T$ has $V$
vertices, $E$ edges, and $F$ faces. The number $F \geq 1$ since $T$ always
has  an exterior face in $S^2$.  Let $V_{\mathrm{int}} = V-k$ denoted the
vertices of $T$ which lie in the interior of $D$. Similarly, let $E_{\mathrm{int}}=E-k$
denote the number of edges of $T$ that are not adjacent to the $k$ vertices
on the boundary of $D$. Finally, let $F_{\mathrm{int}}$ denote the faces
of $T$ in the interior of $D$. By condition (iii) all faces $f \in F_{\mathrm{int}}$
must contain have than six sides. We will show that no such faces can exist.
This relies on an identity relating the number of edges and vertices in $T$.

All vertices $v \in V_{\mathrm{int}}$ are trivalent and all $k$ boundary
vertices have degree one.  We sum over the degrees of vertices in $T$ to
obtain
\begin{equation} \label{vertest}
2E = \sum_{v \in V} d(v) = 3V_{\mathrm{int}}+k \Rightarrow V_{\mathrm{int}}
= \frac{2E-k}{3}.
\end{equation}

We can also compare $E$ and $F$. Since each interior face $f \in F_{\mathrm{int}}$
has at least 7 edges and there are $F-1$ interior faces, we obtain 
\begin{equation}\label{totaledgelow}
7(F-1) \le\sum_{f \in F_{\mathrm{int}}}\#\{\mathrm{edges}( f)\}\end{equation}
In order to bound the right hand side of (\ref{totaledgelow})  we rewrite
\begin{align}
 \sum_{f \in F_{\mathrm{int}}}\#\{\mathrm{edges}( f)\} = \sum_{e \in E_{\mathrm{int}}}
\#\{f \in F_{\mathrm{int}}| e \in \partial f\},   
\end{align}
where $e \in \partial f$ means that $e$ is an edge of the face $f$. For each
edge $e \in E_{\mathrm{int}}$ the corresponding term in the sum  is either
$0$, $1$ or $2$.  Let
$E_{\mathrm{int}}^1$ denote the number of edges that are adjacent to a face
$f \in F_{\mathrm{int}}$ as well as the exterior face of $T$. Equivalently,
this is the
number of terms in the sum which are $1$.  Then
 \begin{align}
\sum_{e \in E_{\mathrm{int}}}
\#\{f \in F_{\mathrm{int}}| e \in \partial f\} \leq 1\cdot E^1_{\mathrm{int}}+2\cdot(E-k-E^1_{\mathrm{int}})
=2E-2k-E^1_{\mathrm{int}}.
 \end{align}
We show that  $E^1_{\mathrm{int}} \ge k$, which means
\begin{align} \label{upperface}
  \sum_{f \in F_{\mathrm{int}}}\#\{\mathrm{edges}( f)\}\le 2E-3k.
\end{align}
To this end,  recall that the vertices  $p_1, \dots, p_k$ are ordered cyclically.
Therefore, between any two vertices $p_j$ and
$p_{j+1}$ there exists a path of edges
$e_1^j, \dots, e_{l_j}^j$ that connect $p_k$ and $p_{k+1}$, and are also
edges of the exterior face of $T$. The sequence of edges
\begin{equation}
e_1^1,\dots, e_{l_1}^1, \dots, e_1^k,\dots,  e_{l_k}^k \label{circuit}
\end{equation} forms a nonintersecting circuit of length at least $k$, with
each edge
 in the circuit bordering exactly one face in $F_{\mathrm{int}}$. This shows
$E^1_{\mathrm{int}} \ge k $ and proves (\ref{upperface}).

The graph  $T$ has an Euler characteristic of  $2 =  V-E+F$. We use this
fact, along with  (\ref{vertest}) and (\ref{upperface}), to obtain
\begin{align}
F &= 2-k+E-V_{\mathrm{int}} \nonumber\\
 &=2+\frac{E}3-\frac{2k}3\nonumber\\
&\ge2+\frac{7(F-1)+3k}{6}-\frac{2k}3 \nonumber\\
&\Rightarrow F\le k-5 \label{contra}.
\end{align}
By the hypothesis of the theorem, $k \le 5$, and thus $F \le 0$. This is
a contradiction since $T$ always contains an exterior face.  Therefore, $T$
contains no interior faces, which means that it must be a tree. 
\end{proof}
\begin{proof}[Proof of Corollary~\ref{topcor}]
We have established that the graph $T$ is a tree embedded in the disk $D$,
with $k$ labeled vertices $\{p_j\}_{j=1}^k$ on the boundary, and trivalent
vertices in the interior of $D$. For $k = 2,3,4,5$, let $S_k$ denote this
set and $C_k=|S_k|$  denote its size. By splitting the tree into descendants
of a degree one vertex, a direct recursive calculation shows that 
\begin{equation}\label{catrec}
C_k = \sum_{i+j=k-1} C_{i+1}C_{j+1}.
\end{equation}
 The trees of relevance to us are shown in Figure~\ref{crittypes}. \end{proof}

The description of trees
above  is closely related the Catalan numbers (see~\cite{Stanley} for a variety
of examples). Specifically, using the recurrence (\ref{catrec}) for $k \ge
2$, $C_k$ is the $k-2$nd Catalan number.    
\begin{figure}
\centering
\includegraphics[width=.9\linewidth]{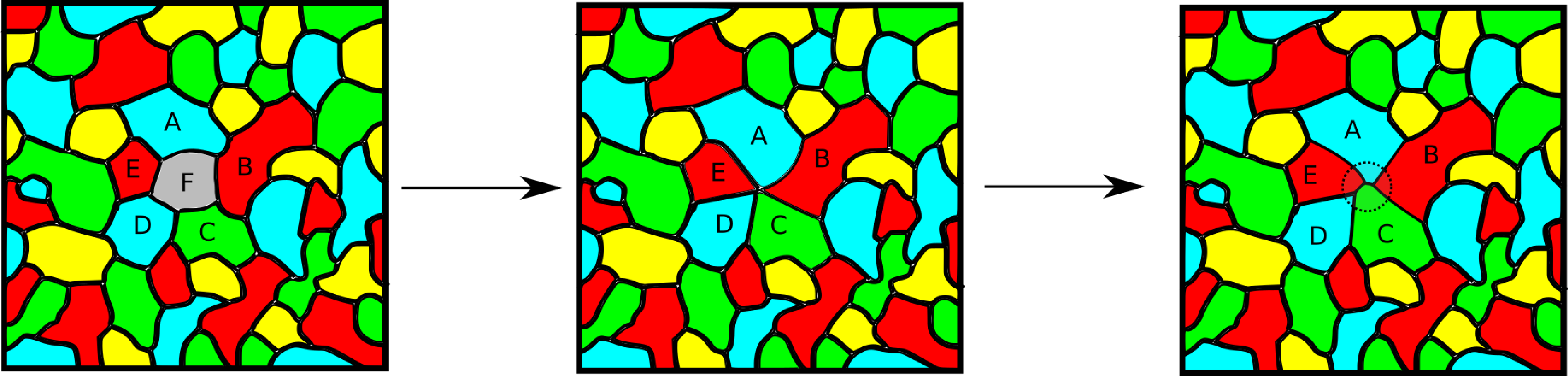}
\caption{\textbf{Trivalent map evolution. }Left: A network before trivalent
map evolution is applied to a five sided face, labeled $F$. Center: The annihilation
step, in which $F$ is contracted to a point. Right: The creation step, in
which a tree is glued at the degree five vertex, recovering the graph's trivalency.
} \label{fig:topchange}
\end{figure}

\subsection{Trivalent map evolution}
Theorem~\ref{topthm} allows us to formulate a Markov chain model for the
topological evolution during grain boundary coarsening on a compact surface
$S$. In order to completely describe a Markov chain, we need a countable
state space and a matrix of transition probabilites for jumps between states.
Our state space is $\mathfrak{T}(S) \cap \{\partial\}$, where $\partial$ denotes
a special, cemetery state. The neighbors of each state $M \in \mathfrak{T}(S)$
are characterized as follows:
\begin{enumerate}
\item {\bf Annihilation\/} : As in Section~\ref{subsec:attachment}, $\mathcal{O}(M)
\subset \mathfrak{M}(S)$, denotes the maps that may be obtained by contracting
an edge of $M$ or a $k$-sided face of $M$, $k \leq 5$, to obtain a map $N
\in \mathfrak{M}(S)$ with a distinguished vertex $v_*$.    
\item {\bf Creation\/}: A tree $T$ is glued to $N$ at $v_*$ to obtain a trivalent
map $M' \in \mathfrak{T}(S)$. The set of all such maps $M'$ is denoted $\mathcal{J}(M)$.
\end{enumerate}
Any assignment of transition probabilities $p_{M,M'}$ for each $M \in \mathfrak{T}(M)$,
$M' \in \mathcal{J}(M) \cap \{\partial\}$ determines a Markov chain. The cemetery
state is included to deal with the possibility that $\mathcal{J}(M)$ is empty.
In this case, $p_{M,\partial}=1$ and $p_{\partial,\partial}=1$. We call any
Markov chain of this form trivalent map evolution (TME).

A trivalent map $M \in \mathfrak{T}(S)$ must always have faces with fewer than
six sides when $S=\mathbb S^2$ or $S=\mathbb{T}^2$. We use Euler's formula
$V-E+F = \chi(S)$ along with the identities $3V = 2E $ and 
\begin{align}
E_{\mathrm{av}}\cdot F := \sum_{f \in F} \#\mathrm{edges}(f) = 2E
\end{align}
to see that the average number of edges in a face is given by\begin{equation}\label{avgsides}
E_{\mathrm{av}} = \frac{6}{\frac{\chi(S)}{2E}+1}.
\end{equation}
When $S=\mathbb S^2$, $\chi(S) = 2$ and $E_{\mathrm{av}}  < 6$, so that there
is always a face with fewer than six sides. When $S = \mathbb T^2$, $\chi(S)
= 0$ and we may always find a face with fewer
than six sides unless $M$ is a hexagonal tiling. For all other compact surfaces,
the genus $g \geq 2$, $\chi(S) = 2-2g \leq -2$, so that $ E_{\mathrm{av}}
> 6$. 
 
\subsection{Topological restrictions on kinetic models} \label{subsec:topstriction}
Kinetic models for grain boundary coarsening typically assume the following
{\bf grain coarsening rules\/} at topological transitions (see Figure \ref{distrules}):
\begin{itemize}
\item If $k = 2$, two neighboring grains will both lose two
sides.
\item If $k = 3$, three neighboring grains will each lose
one side.
\item If $k = 4$, two neighboring grains each lose
one side, and two neighboring grains retain their topology.
\item If $k = 5$, two neighbors each lose one sides,
one neighboring grain gains  one side and two neighboring grains retain
their topologies.
\item If an edge connecting vertices $p$ and $q$ vanishes, with $p$ and $q$
adjacent
to four distinct grains, then  two of these grains each lose one side
and the other two grains each gain one side.
\end{itemize}

\begin{figure}
\centering
\includegraphics[width=.85\linewidth]{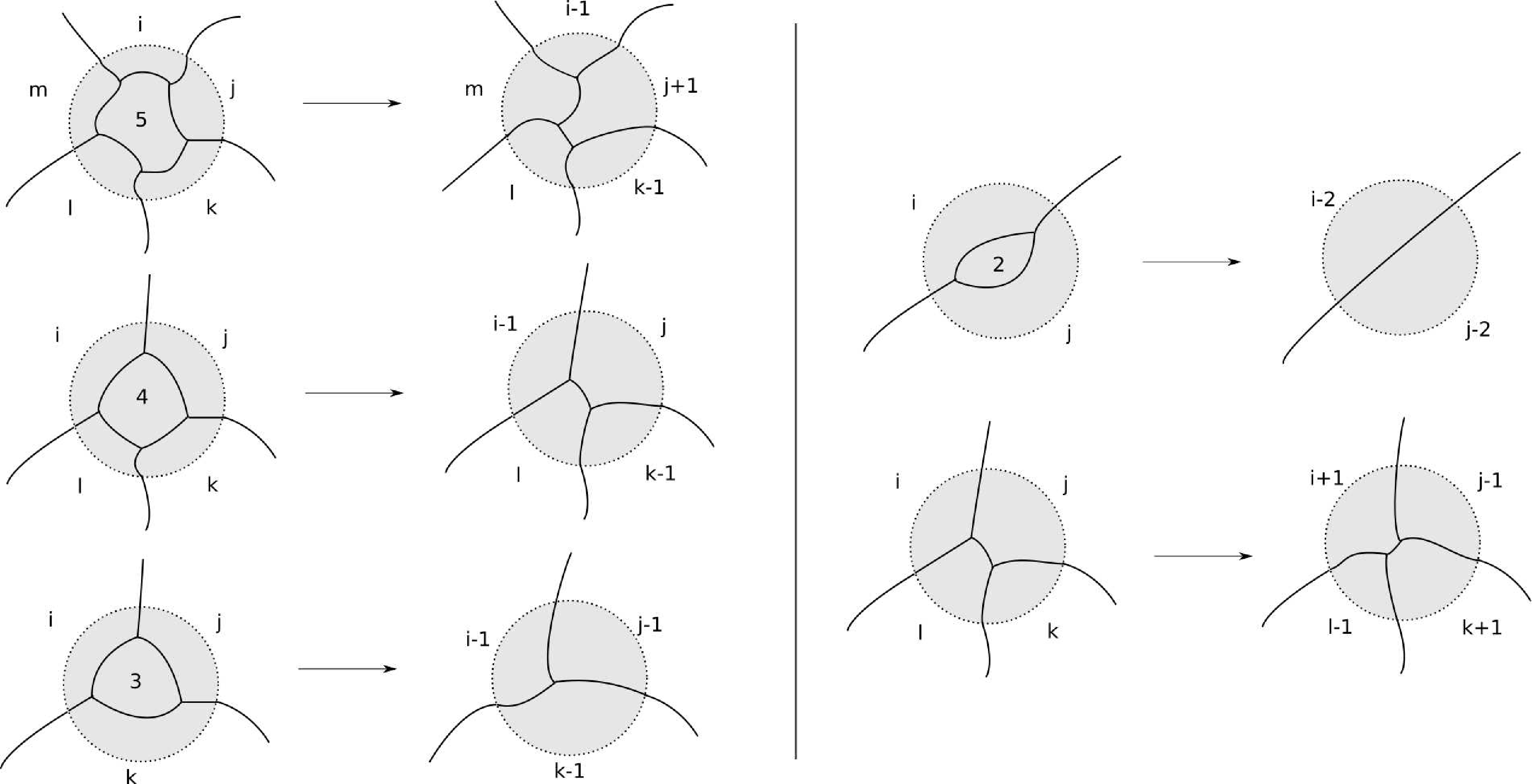}
\caption{\textbf{Changes  of topology before and after TME. }
A number in a grain refers to the number of its sides. When four and five
sided grains vanish, the tree that can be attached is not unique. However,
the net topological change is always the same.}
\label{distrules}
\end{figure}
 
However, trivalent map evolution may not follow these rules (see Fig. \ref{countersphere}).
An extra condition on the maps must be imposed so that TME is consistent
with kinetic theory. An additional hypothesis is required, which is that
all faces with  $k$ sides have  $k$  {\it distinct neighboring grains}, and
that an edge and its two vertices are adjacent to four distinct faces. Denote
this space as $\mathfrak N(S)$. For these  embeddings\footnote{The space
$\mathfrak N(S)$ is a proper
subspace  of closed 2-cell embeddings, in which the closure of each face
is homeomorphic to a closed disc.}, the following holds.

\begin{corollary} \label{cor:top}
Suppose $M \in \mathfrak N(S)$. For a face or edge selected in the annihilation
step of TME,  the topological change from $M$ to $M'$ follows the grain coarsening
rules.\end{corollary}
\proof

\ If $M \in \mathfrak N(S)$, then the distinguished vertex $v_*$ arising
from  the annihilation step of TME will be of degree $k$, and will furthermore
  border $k$ distinct faces.  From Corollary 1, we can directly verify that
the grain coarsening rules hold  for all of the possible trees that we may
glue in the creation step of TME.  Note that the creation step for TME with
four and five sided grains can result in different possible topologies, but
in all cases either two grains lose one side for a vanishing four-sided grain,
or two grains lose a side and one grain gains a side for a vanishing five-sided
grain. 
\qed

\begin{figure}
\centering
\includegraphics[width=\linewidth]{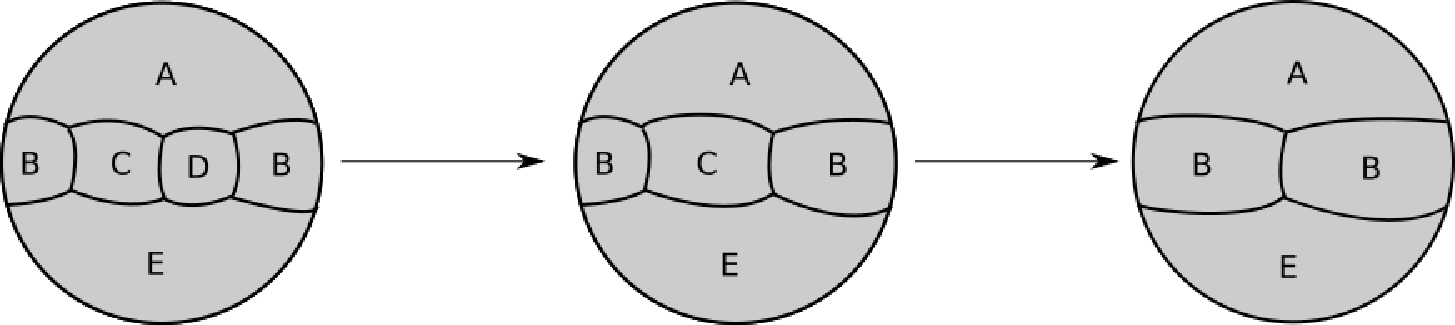} 
\caption{\textbf{Failure of grain coarsening rules in the sphere $\mathbb
S^2$.  }In all three networks shown, the face $B$ wraps around the back 
of the sphere. Left: A network of five grains in $\mathfrak N(\mathbb S^2)$.
Center: The network after face $D$ has been removed through TME. This network
is  in $\mathfrak M(\mathbb
S^2)$, but no longer in $\mathfrak N(\mathbb S^2)$, as the face $C$ has four
edges, but only three neighbors. Right: The network after  face $C$ has been
removed through TME. The grain coarsening rules in this case do not hold
(here,  $A,B,$ and $E$ each lose one edge).      } \label{countersphere}
\end{figure}

\section{ Topological kinetic equations } 
\label{sec:mspecgrain}

The major motivation for the PDMP model is  to create a particle system representation
for mean field models in grain boundary coarsening.      In this section
and Section \ref{sec:tunableparam}, we will   assign values to each of the
parameters defined in Section~\ref{sec:model}  to   correspond with key coarsening
properties. In this section we will assign weights and mutation matrices
related to the topological requirements of the grain coarsening rules enumerated
in Section \ref{sec:tops}.   This allows us to write  topological kinetic
equations, which  should be considered as the most general set of kinetic
equations for describing  topological networks which evolve by TME and respect
the grain coarsening rules. In Section \ref{sec:tunableparam}, we will consider
other   
parameters associated with edge deletion and first neighbor correlations.
The free parameter space is sufficiently flexible to incorporate the correlation
and edge-deletion assumptions of several of the models described in Section
5.

We now assign parameters for the PDMP model that we will consider fixed for
all of the simulations carried out in Section \ref{sec:results}.    As mentioned
in Section~\ref{sec:model}, particles correspond to individual cells of a
two-dimensional grain network.  A particle's species $s$ corresponds its
number of sides, and its position $x>0$ corresponds to its area. The most
immediate of these  is setting species velocity of $v_s = \frac \pi 3(s-6)$,
which is exactly the von Neumann-Mullins  relation in which we set material
constant in (\ref{eq:von-neumann}) to $c= \pi/3$. Since grains with more
than 10 sides are quite rare, we also cap the maximum number of sides at
$M = 15$.    Our models do not consider one-sided grains,
as they are relatively rare in actual metal networks (around .1\%, see \cite{fradkov1985experimental}).
To keep  the system closed, we
forbid two-sided grains to lose edges,  three-sided grains to lose two edges,
and $M-$sided grains to gain an edge at a critical events.

We also assign values for parameters related to topological transitions in
  networks which remain in $\mathfrak
N(\mathbb T^2)$ during their evolution. This includes  $K^{(l)}$, the number
of grains affected by deletion events, and the mutation matrices $R_{kj}^{(l)}$,
which specify transitions for each such grain. Values for   $K^{(l)}$ are

\begin{equation}
K^{(2)}=2, \quad K^{(3)}
= 3, \quad K^{(4)}= 2, \quad K^{(5)}= 3,\quad K^{(0)} = 4,
\end{equation}
  and  the mutation matrices are defined as\begin{align}
R_{kj}^{(2)}&= \begin{cases}k-2, & k \in \{4,\dots, M\},\ j\in \{1,2\}, 
\label{startrules} \\
0, & k \in \{2,3\} ,  \ \\
\end{cases}\\
R_{kj}^{(3)}&= \begin{cases}k-1, & k \in \{3,\dots ,M\},\ j\in \{1,2,3\},
\\
0, & k =2,  \ \\
\end{cases}\\
R_{kj}^{(4)}&= \begin{cases}k-1, & k \in \{3,\dots ,M\},\ j\in \{1,2\}, \\
0, & k =2,  \ \\
\end{cases}\\
R_{kj}^{(5)}&= \begin{cases}k-1, & k\in \{3, \dots,M-1\},\ j \in \{1,2\},
\\
k+1, & k\in \{3 ,\dots,M-1\},\ j=3,\\
0, & k\in \{2 ,M\},\\
\end{cases}\\
R_{kj}^{(0)}&= \begin{cases}k-1, & k\in \{3, \dots,M-1\},\ j \in \{1,2\},
\\
k+1, & k\in \{3 ,\dots,M-1\},\ j\in\{3,4\},\\
0, & k\in \{2 ,M\}.\\
\end{cases} \label{endrules}
\end{align}  

Recall the upper index $l= 2, \dots M_-$  in $R_{kj}^{(l)}$ refers to the
number of sides for the deleted grain, $k= 2,\dots, M$ refers to the number
of sides of a neighboring grain before undergoing topology reassignment,
and $j= 1\dots, K^{(l)}$ is the index of the $j$th reassignment. Refer to
Fig. \ref{distrules} for a pictorial
description of the distribution of sides before and after grain and edge
deletions.
Note that to keep the particle system closed, we disallow mutations that
create 1 and $M+1$ sided  grains.

With these defined parameters, we now give a more explicit form of the limiting
kinetic equations of (\ref{eq:kinetic})
for models of grain boundary coarsening. Assuming a continuous area density
 $f_n(a,t)$ for $n$-sided grains at time $t$,
we  can write a system of \textbf{topological kinetic equations}  for grains
with $n = 2, \dots, M$  sides as\begin{equation}\label{grainpde}
\partial_tf_{n}(a,t)+(n-6)\partial_af_{n}(a,t) = h^{n+}_{\mathrm{grain}}(f,t)-h_{\mathrm{grain}}^{n-}(f,t)+h^{n+}_{\mathrm{edge}}(f,t)-h_{\mathrm{edge}}^{n-}(f,t)
\end{equation}
for $a,t\ge 0$.

The four source terms of the right hand side of (\ref{grainpde}) describe,
in order, addition and deletion  of $n$-sided grains due to grain deletions,
and addition and deletion of $n$-sided grains due to edge deletions.
By directly substituting values of  $v_s$ and $R_{kj}^{(l)}$ into  (\ref{eq:kinetic}),
we may write,  for $n = 2, \dots, M$,
\begin{align} \label{hgrainbegin}
h_{\mathrm{grain}}^{n,+}(f,t)&= 8f_2(0,t)W_{n+2}^{(2)}(t)f_{n+2}(a,t)+9f_{3}(0,t)W_{n+1}^{(3)}(t)f_{n+1}(a,t)\\
&+4f_{4}(0,t)W_{n+1}^{(4)}(t)f_{n+1}(a,t)+
2f_{5}(0,t)W_{n+1}^{(5)}(t)f_{n+1}(a,t)\nonumber
\\&+f_{5}(0,t)W_{n-1}^{(5)}(t)f_{n-1}(a,t),
\nonumber\\
h_{\mathrm{grain}}^{n,-}(f,t) &= f_n(a,t)[8f_2(0,t)W_{n}^{(2)}(t)+9f_{3}(0,t)W_{n}^{(3)}(t)\\&+4f_{4}(0,t)W_{n}^{(4)}(t)+3f_{5}(0,t)W_{n}^{(5)}(t)],
\nonumber\\
h_{\mathrm{edge}}^{n,+}(f,t)&=2\beta (t)\gamma(t)[w_{n-1}^{(0)}f_{n-1}(a,t)+w_{n+1}^{(0)}f_{n+1}(a,t)]
,
\\
h_{\mathrm{edge}}^{n,-}(f,t)&= 4\beta(t) \gamma(t)w_n^{(0)}f_n(a,t),
 \label{hgrainend}
\end{align}
for tier weights $w_n^{(l)},w_n^{0}$ and species selection weight fractions
$W_n^{(l)}(t),\gamma(t)$ defined in (\ref{eq:numbers-not2}).

\section{Previous kinetic models } \label{sec:kinmodels}

In this section, we review several previous kinetic models which describe
area
densities $f_n(a,t)$ of $n$-sided grains at time $t$.
All
of the models assume that  a  neighbors are uncorrelated. Specifically,
if we consider a single
edge for some  $k$-sided grain, the probability that another grain with $n$
sides  is  its neighbor
 is  $n/(S-k)$, where $S$ is the total number  of sides
for all grains in the system. Three of the models mentioned here, those of
Beenakker, Marder, and Flyvbjerg, assume no edge deletion. 
  
\subsubsection{The model of Beenakker}The model posed by  Beenakker \cite{beenakker1986evolution}
begins by assuming grains are shaped like regular polygons, except having
edges of circular arcs meeting at $120^\circ$.  This assumption gives the
relation $P = \sqrt{a/\phi(n)}$ for an $n-$sided grain with area $a$ and
perimeter $P, $  where $\phi(n)$ has an explicit (albeit lengthy) form. For
distributions $f_n(a)$ for $n$-sided grains, an isotropic network in which
thermal effects are neglected has a  free energy given by the total perimeter

\begin{equation}
F = \int_0^\infty \sum_{n = 2}^\infty\   \sqrt{a/\phi(n)}f_n(a).
\end{equation}

To obtain an equilibrium state, Beenakker minimizes  $F$ over possible  densities
$\{f_n(a)\}_{n \ge 2}$ having with a fixed area $A$ and zero polyhedral defect,
meaning 
\begin{equation}
\int_0^\infty \sum_{n = 2}^\infty(n-6)f_n(a)da = 0. \label{polycon}
\end{equation}
It can be shown that  the minimizer for this variational problem gives a
unique, explicit  assignment $n_c(a)$ of sides for all grains with area $a$.
 Thus,  densities take the form 
\begin{equation}
f_n(a) = f(a)\mathbf 1_{n = n_c(a), }
\end{equation}

From here, Beenakker invokes the $n-6$ rule to  arrive at a simple
evolution of area densities $f(a,t)$ for time $t \ge 0$, given by   \begin{equation}
\partial_tf(a,t)  +c\partial_a (f(a,t)(n_c(a)-6)) = 0.\label{beenakin}
\end{equation} In a numerical integration of (\ref{beenakin}), Beenakker considers initial
conditions of mostly hexagonal cells with a small fraction of pentagon-heptagon
pairs. As expected, defects initially propagate in the system.  However,
a ``collapse" eventually occurs in which   the system rapidly returns to
having  near zero defect.   The network then repeats this pattern indefinitely,
oscillating between   ordered
and highly disordered states. 

\subsubsection{The model of Marder}
    Marder's model \cite{marder1987soap} is unique in that it sets a rule
for which neighbors lose an edge when a four-sided grain vanishes, choosing
the two smallest grains.  Similarly, when a five-sided grain vanishes, its
two smallest neighbors lose an edge,  and the largest neighbor gains an edge.
  The kinetic equations have the form\begin{align}
&\partial_t f_n(a,t) +c(n-6)\partial_af_n(a,t) = u_{n-1}(a,t)\frac{n-1}{S(t)}f_{n-1}(a,t)\\&-(u_n(a,t)+d_n(a,t))\frac{n}{S(t)}
f_n(a,t)+d_{n+1}(a,t)\frac{n+1}{S(t)}f_{n+1}(a,t). \label{marder}
\end{align}

Here $u_n(a,t)$ and $d_n(a,t)$ are the rates that $n$-sided grains of area
$a$ gain or lose a edge, respectively, and $S(t) $ is the sum of all sides
over all grains at time $t$.  Expressions for $u_n$ and $d_n$ are complicated,
and involve probabilities $p(a,t)$ that a grain selected at time $t$ has
area greater than $a$. A scaling analysis correctly shows that coarsening
should be linear, and numerical approximations for  (\ref{marder}) give statistics
for topological frequency and coarsening that were comparable to experimental
results on soap bubble networks. 

\subsubsection{The model of  Flyvbjerg}
  Kinetic equations for Flyvbjerg's model \cite{flyvbjerg1993model} are of
the form
\begin{equation}
\partial_t f_n(a,t)+c(n-6)\partial_af_n(a,t) =\sum_{m =n- 1}^{n+1}T_{n,m}(f)f_m(a,t),
\quad n = 0,1, \dots  .,\label{flyvbjerg}
\end{equation}
where
tridiagonal coupling is given by\begin{align}
&T_{n,m} =\begin{cases}  c_+(n-1) & m = n-1, \\
\bar A -(c_+-c_-)n  & m= n, \\
c_-(n+1) & m = n+1, \\
0 & \hbox{otherwise,} \\
\end{cases} 
\\
&c_+ = \frac \pi{18}(n-6)f_5(0,t), \quad c_- = \frac \pi{18}\sum_{k = 0}^5(k-6)^2f_k(0,t)+c_+,\\
&\bar A= \frac \pi 3 \sum_{k = 0}^5(6-k)f_k(0,t).
\end{align}
The transition rates $T_{n,m}$ for grains evolving from $m$ to $n$ sides
contain boundary terms $f_k(0,t)$ which make (\ref{flyvbjerg}) nonlinear.
 An assumption for determining  $T_{n,m}$ is that networks have zero first-neighbor
correlations.    This means that for any grain with $l$ sides, the probability
that a neighboring grain has $j$ sides  is independent of $l$, and proportional
to $j$. To keep $T_{n,m}$ in a simplified form, Flyvbjerg assumptions required
his
model to allow for 1 and 0 sided grains.  Simulations later showed these
topologies were essentially negligible. In general, the comparison between
Flyvbjerg's model and experimental
data is surprisingly accurate, with frequencies of topologies differing by
only a few percentage points.

\subsubsection{ The model of Fradkov} The Fradkov model, posed in 1988 as
a `gas approximation' to grain coarsening \cite{Fradkov1,Fradkov2}, has kinetic
equations of the form
\begin{equation}
\partial_t f_n(x,t)+c(n-6)\partial_af_n(a,t) =\Gamma(f(t))(Jf)_n(a,t). 
\end{equation}
The collision operators $(Jf)_n$, $n \ge 2$,  account for topological transitions,
taking the form
\begin{align}
(Jf)_2 &= 3(\beta_{\mathrm{RD}}+1)f_3-2\beta_{\mathrm{RD}}f_2 \label{jeqn1}\\
(Jf)_n &= (\beta_{\mathrm{RD}}+1)(n+1)f_{n+1}-(2\beta_{\mathrm{RD}}+1)nf_n+\beta_{\mathrm{RD}}(n-1)f_{n-1},
\quad
n>2. \label{jeqn2}
\end{align}
Like the model of Flyvbjerg, topological transitions assume networks have
zero correlation for first neighbors. Furthermore, collision operators also
include a removal driven deletion assumption in which  
the ratio between edge deletion and
grain deletion  is a fixed constant $\beta_{\mathrm{RD}}$.  
Thus, if $S_F(t)$ denotes the total number of edge deletions under the Fradkov
model, then 
\begin{equation}\label{fradform}
S_F(t) = \beta_{\mathrm{RD}}(N(0)-N(t)). 
\end{equation}The quantity $\Gamma(f(t))$ is then chosen  to conserve
total area, and invokes   nonlinearity in the system. To keep the $J$ tridiagonal,
topological effects from two-sided grains  were ignored.  The addition of
an edge resulting from a five sided grain is also ignored. Derivations for
both
(\ref{jeqn1})-(\ref{jeqn2}) and $\Gamma(f(t))$ can be found in \cite{Henseler1}.

\subsubsection{The BKLT model}

The model of Barmak, Kinderlehrer, Livshits, and Ta'asan \cite{barmak2006remarks},
or BKLT model, has  topological coupling  with the same general  form as
the right hand side of (\ref{flyvbjerg}).  In the absence of edge deletion,
the coupling terms $T_{n,m}^{\mathrm{BKLT}}(f)$ are determined through the
conservation of polyhedral defect (\ref{polycon}) to be
\begin{align}
T_{n,m}^{\mathrm{BKLT}} =\begin{cases} \frac{\psi(t)}{F_{m}(t)}a_m & m =
n,n+1, \\
0 & \mathrm{otherwise}. \\
\end{cases} 
\end{align} 
Here, $\psi(t)$ is the rate of boundary deletion given by
\begin{equation}
\psi(t) = \sum_{n = 3}^5 (n-6)^2 f_n(0,t),
\end{equation}
and $a_i$ are free constants which satisfy $\sum_{i \ge 3} a_i = 1$.  These
constants are free, and can be fit through comparisons with experiments.

In \cite{cohen2010stochastic}, Cohen gives a stochastic description of the
governing equations for the BKLT model in which a solution has an interpretation
of a conditional backwards expectations  of a sample tagged grain.  
We mention that in   \cite{cohen2010stochastic}, Cohen considers the  one-dimensional
minimal model
\begin{equation}
\partial_t f(x,t)-\partial_xf(x,t) = -\frac{f(0,t)f(x,t)}{\int_{\mathbb R_+}
f(x,t)}, \quad x,t>0 \label{onespec}
\end{equation}
as an explanatory example.  In \cite{klobusicky2017concentration}, two of
this paper's authors (JK and GM) considered  (\ref{onespec}) as a hydrodynamic
limit of a one-species stochastic particle system on the positive real line,
and proved    exponential convergence of empirical densities to this limit.

\section{ Parameter fitting with direct numerical simulations }\label{sec:tunableparam}

In order to complete the specification of our stochastic particle model,

it remains to provide expressions for  the  weights $w^{(l)}_k$, which will
represent topological frequencies of a grain's neighbors, and the interior
event intensity
$\beta$, which will represent the rate of edge deletion.   
For parameter fitting, the particle models will be compared with  direct
simulations of network
flow, with data obtained from a  level set method of Elsey, Esedoglu,
and Smereka \cite{Elsey2}. The numerical simulation  is performed with 667,438
initial
grains,  determined from a Voronoi tesselation of the
unit square with periodic boundary conditions.  Grains evolve
 for a total time of  $2.384\times 10^{-5}$, with
updates given over two hundred uniform time steps ($\Delta t = 1.192 \times
10^{-7}$).
The dataset   from this simulation consists of each grain's area and topology
(included
whether the grain vanished) over each time step. 

\subsection{First-neighbor correlations}\label{subsec:parweight}

From the species selection probabilities   (\ref{eq:flip-probs}) and (\ref{eq:flip-probs-int}),
the parameters $w_k^{(l)}$  allow us to weight the frequency of grains with
$k$ sides which border grains with 
$l$ sides. 
Similar to the models of Flyvbjerg and Fradkov, we may assume frequencies
are independent
of the topology of the vanishing grain, and are only proportional to $k$,
 meaning $w_k^{(l)} \propto
k.$  In this case, imposing again that grains have between 2 and $M$ sides,
we obtain \textbf{uncorrelated
weights}
\begin{align}\label{sideweights}
w_k^{(2)}&= \begin{cases}k, & k \in \{4,\dots, M\},   \ \\
0, & k \in \{2,3\},  \ \\
\end{cases} &
w_k^{(3)}&= \begin{cases}k, & k \in \{3,\dots, M\},   \ \\
0, & k =2,  \ \\
\end{cases}\\
w_k^{(4)}&= \begin{cases}k, & k \in \{3,\dots, M\},   \ \\
0, & k =2,  \ \\
\end{cases}&
w_k^{(5)}&= \begin{cases}k, & k \in \{3,\dots, M-1\},   \ \\
0, & k \in\{2,M\},  \ \\
\end{cases}\\
w_k^{(0)}&= \begin{cases}k, & k \in \{3,\dots, M-1\},   \ \\
0, & k \in\{2,M\}.  \ \\
\end{cases}
\end{align}

Experiments on grain networks have shown that nontrivial nearest neighbor
 correlations do in fact exist, with grains having fewer sides tending to
neighbor higher sided grains, and vice versa.  The most popular relation
  is Aboav's law~\cite{aboav1970arrangement}, an empirical observation
relating the average number of neighbors  $m_n$ for an $n$-sided cell as
\begin{equation}
nm_n = (6-a)n+b,
\end{equation}
where $a$ and $b$ are numerical constants. For our purposes, we are actually
able to incorporate more information, and consider  distributions $\{c_k^{(l)}\}_{k
= 1}^M$ which give the probability of a $l$-sided cell to have a neighbor
with $k$ sides, assuming the network to  be stabilized.  This distribution
is not provided in the dataset provided from the  level set used, so we instead
use  results  in   \cite{meng2015study} which simulated grain boundary coarsening
through a Monte Carlo-Potts model.  With these distributions, we may estimate
\textbf{correlated weights} $\tilde w_k^{(l)}$ from the level set simulation.
 This is done by using using topological frequencies $p_k$ simulated from
the level
set model which give the frequency of grains with $k$ sides.  This is done
at time $t = 5.364
\times 10^{-6}$, when about 80\% of grains have vanished, and frequencies
are essentially stable for the remainder of the simulation.  Neighbor distributions
for a $l$-sided grain in the PDMP model may then be written   in terms of
 $p_k$ and $\tilde w_k^{(l)}$, and then compared with  species selection
probabilities
to $c_k^{(l)}$ to obtain the linear system of equations in $\tilde w_k^{(l)}$,
given by \begin{align} 
\frac{p_k \tilde w_k^{(l)}}{\sum_{n = 1}^M p_n \tilde w_n^{(l)}} &= c_k^{(l)},
\quad k = 1, \dots, M, \\ \sum_{n = 1}^M \tilde w_n^{(l)} &= 1.
\end{align}
 Comparison between  the weights $ w_k^{(l)}$ and $\tilde w_k^{(l)}$ are
given in Figure \ref{fig:corrgraph}.

\begin{figure}
\centering
\includegraphics[width=.6\linewidth]{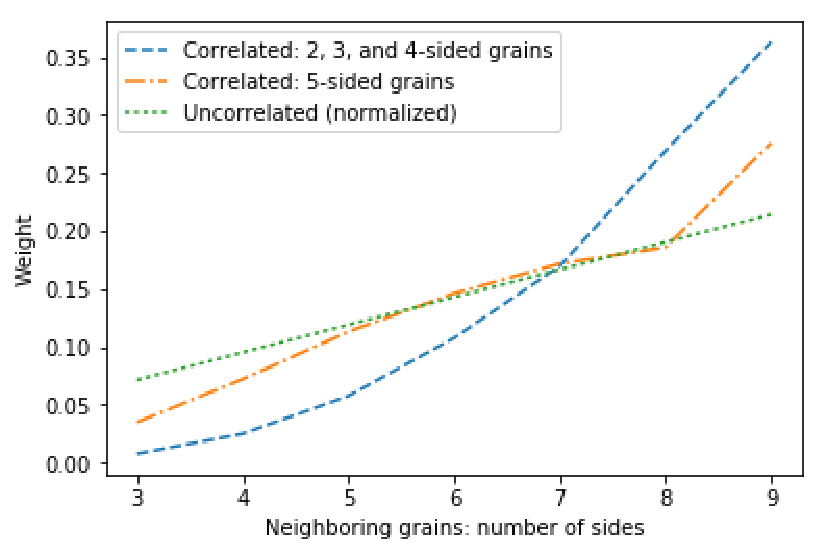}
\caption{\textbf{Correlated and uncorrelated weights}. For each graph, the
independent variable denotes the number of sides $k$ for neighboring grains,
and the  dependent variable is either the correlated weight $\tilde w_k^{(l)}$
or normalized uncorrelated weight $ w_k^{(l)}$.  Since instances for two
and three sided grains in \cite{meng2015study} were rare, frequencies for
four-sided grains were used to determine $\tilde w_k^{(l)}$ for $l = 2,3,4$.}
\label{fig:corrgraph}
\end{figure}

\subsection{Edge deletion}\label{subsec:parside}

\subsubsection{Population-driven vs.\  removal-driven edge deletion}
We now describe how to fit the rate function $\beta(t)$ for several varieties
of edge deletion models.  The most simplified model, of course, is to assume
no edge deletion at all.  We may incorporate the no deletion assumption into
the topological kinetic equations easily by setting $\beta(t) \equiv 0$,
 which we call the   \textbf{no edge deletion (ND)
model}.

 Edge deletion for the topological kinetic equations at any fixed time $t$
is governed by
a Poisson process with rate $\beta(t) N(t)$, where $N(t)$ is the total number
of grains at time $t$.  If we assume that $\beta(t) \equiv \beta_{\mathrm{PD}}
$ is constant, then we call the resulting particle system a \textbf{population-driven
edge deletion (PD) model.}   We can fit $\beta_{\mathrm{PD}}$ by   comparing
$\beta_{\mathrm{PD}} N(t)$
against edge deletion rates in experimental  data.

We may also consider a \textbf{removal-driven edge deletion  (RD)  model}
motivated by Fradkov's model \cite{Fradkov1,Fradkov2},
which imposes that the total
number of edge deletions is proportional to the total number of grain deletions.
To derive such an expression of $\beta(t)$, we begin with an an  explicit
estimate for the total number $N(t)$. 
This estimate assumes a  linear coarsening rate $\alpha>0$ so that
for a system with
total area $A$, the average grain  size $\langle A_t\rangle$ at time $t$
takes the form
\begin{equation}
\langle A_t\rangle =\frac{ A}{N(0)}+ \alpha t. \label{linearcoarsening}
\end{equation}
This assumption, although not rigorously shown,
is quite accurate in practice. Indeed, a linear regression on the coarsening
rate for the level set data gives a Pearson correlation coefficient $R>.999$.
  With this assumption,   the total number is  
\begin{equation}\label{lingrow}
N(t) = A/\langle A_t\rangle = \frac {AN(0)}{A+N(0)\alpha t}.
\end{equation}

If $S_{\mathrm{RD}}(t)$ denotes the total number of edge deletions under
the RD
model up to time $t$, then 
\begin{equation}\label{sfform}
S_\mathrm{RD}(t) = \beta_\mathrm{RD}(N(0)-N(t)). 
\end{equation}
The quantity $\beta_\mathrm{RD}$ may be found  by fitting against level set
data. On
the other hand, for the PD model, the total amount of edge deletion
is
\begin{equation}\label{smform}
S_{\mathrm{PD}}(t) =   \int_0^t N(s)\beta(s) ds.
\end{equation}
 We equate   $S_\mathrm{RD}(t) $ with $S_{\mathrm{PD}}(t) $ in (\ref{sfform})
and (\ref{smform}), so that 
\begin{equation} \label{fraddiff}
\int_0^tN(s)\beta(s)ds =\beta_\mathrm{RD}(N(0)-N(t)). 
\end{equation} 
Using (\ref{lingrow}),  we may easily solve
for $\beta(t)$ in (\ref{fraddiff}) to obtain
the autonomous form
\begin{equation}
\beta(t) =\alpha \beta_\mathrm{RD} N(t).
\end{equation}
To summarize, in implementing the $M$-species model, we obtain edge deletion
similar to RD model assumptions by setting the rate of the Poisson clock
to $\beta(t)N(t) = \alpha\beta_\mathrm{RD}N(t)^2$. 

Fitted values for
$\alpha,  \beta$, and $\beta_\mathrm{RD}$ are given in Table \ref{table:flip}.
The
coarsening rate $\alpha$ was found through linear regression. The constant
edge deletion rate $  \beta_{\mathrm{PD}}$ was obtained by minimizing the
$L^1$ distance over
time between the rate of  edge deletion and $  \beta_{\mathrm{PD}} F(t)$.
 Similarly, the
Fradkov parameter $\beta_\mathrm{RD}$ was obtained by minimizing $L^1$ distance
over
time of the fraction of grains removed over edge deletions.
See Figure \ref{fig:fitting} for graphs illustrating edge deletion behavior
over time. Note that there appears to be a larger number  edge
deletions at the beginning of the simulation, possibly due to a burn-in
period which adjusts to initial conditions of a Voronoi tessellation.

\subsubsection{Estimating total edge deletion}
We note here that  tracking systems for grain networks in most level set
methods, including
the one used in this study, only track individual faces, and not edges. While
we do not have precise data on edge deletion, we will present a method
 for estimating independent edge deletion using the level set dataset. 
This uses the available data of topologies and areas for each grain at each
time step $t_k = k\Delta t$. The  calculations for estimating total edge
deletions  are approximate for two
reasons. First, we assume that each grain  changes its number of sides at
most once during
a time step.  Second, we ignore possible edge deletions occurring immediately
before
grain deletion,  and thus assume that 
a grain changes its number of sides  from either a vanishing neighboring
grain, or an
edge deletion from a  neighboring grain of typical size.  

To estimate total edge deletion,  we record grains $(a_1^k, s_1^k), \dots,
(a_g^k, s_g^k)$ which
vanish
in the time interval  $[t_k, t_{k+1})$.  Under our assumptions,   the  total
number of remaining
grains which change their number of sides  as a result of grain deletion
is then

\begin{equation}
\Delta E_k^f = \sum_{i = 1}^g (3\cdot \mathbf 1_{s_i ^k= 3}+2\cdot \mathbf
1_{s_i ^k= 4}+ 3\cdot \mathbf 1_{s_i^k = 5}).
\end{equation}
Let $\Delta E_k$ denote the total number of grains affected by both grain
deletion and edge deletion. Since each edge deletion affects four neighboring
grains, the total
number of
topological changes not due to grain deletions is then estimated
as
\begin{equation}
\Delta S_k = \frac{\Delta E_k- \Delta E_k^f}4.
\end{equation}

\section{Numerical results}
\label{sec:results}

In this section, we will simulate six varieties of PDMP models for coarsening,
 using either correlated or
uncorrelated weights, and edge deletion rates assumptions that follow either
the RD, PD, and ND models shown  in Table \ref{table:flip}. 
Each of these models are also compared with a level set method. For each
model, 200,000 grains are sampled for initial conditions. Initial areas and
topologies for the particle system model were  selected
through sampling with replacement from the initial empirical distribution
of the level set model.  After this selection, a small number of grains were
then modified to produce a distribution with exactly 0 polyhedral defect,
and like the level
set simulation,  evolve for a total time of   $t = 2.384\times 10^{-5}$.

\begin{table}
\centering
\begin{tabular}{|c||c|c|c|}
\hline
Model & ND &    PD &  RD \\
\hline
Switching & $\beta(t) =  0$ &    $\beta(t) =  \beta_{\mathrm{PD}}$ &   $\beta(t)
= \alpha\beta_\mathrm{RD}N(t)$ \\\hline
Fitted value &N/A & $\beta_{\mathrm{PD}} = 75072.74$ & $\alpha = 1.27,\beta_\mathrm{RD}=
2.02 $ \\\hline
\end{tabular}
\caption{\textbf{Parameter fitting for edge deletion models}. For the RD
model, the parameter $\alpha$ corresponds to the coarsening rate, given the
linear coarsening assumption in (\ref{linearcoarsening}).}\label{table:flip}

\end{table}

\begin{figure}
\subfloat{\includegraphics[width=.49\linewidth]{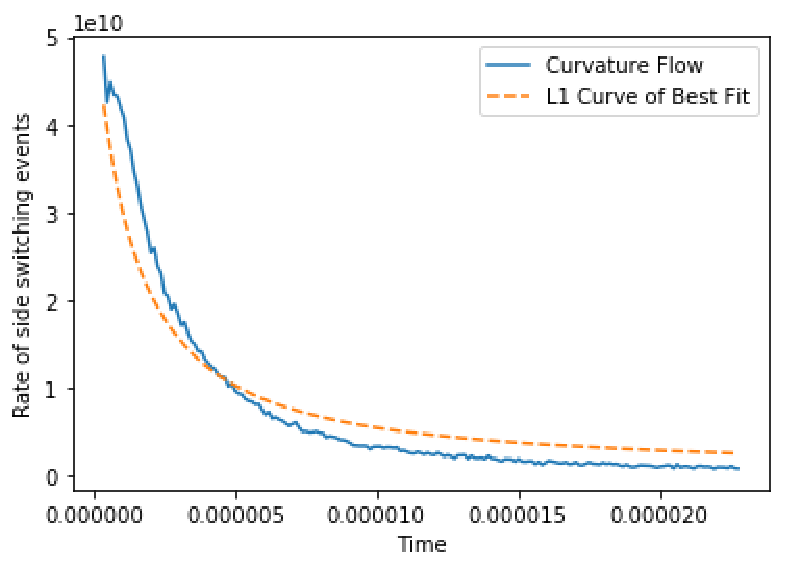}}
\subfloat{\includegraphics[width=.49\textwidth]{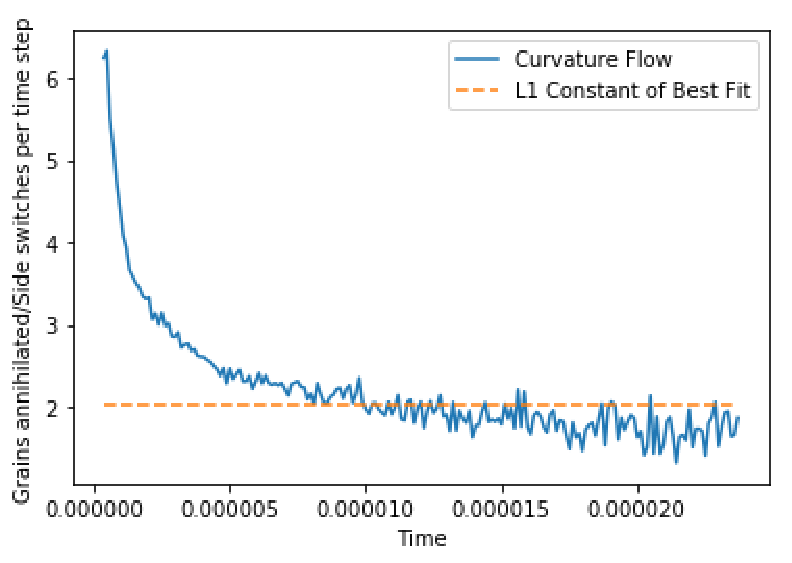}}
\caption{\textbf{Parameter fitting for PD  and RD models.} Left: Fitted curve
for the rate of edge deletion  $\beta_{\mathrm{PD}}N(t).$ Right: Fitted constant
of the  fraction of edge deletions over grain
deletions $\beta_{\mathrm{RD}}$. }  \label{fig:fitting}
\end{figure}

Comparisons of grain statistics between the level set method and the three
particle models are given in Figures \ref{fig:tops}-\ref{fig:7areacorr}
and Table \ref{table:metrics}.
In Table \ref{table:metrics}, we use the total variation metric, which for
two discrete probability vectors $p = (p_1, \dots, p_n)$ and $q = (q_1, \dots,
q_n)$ is given by the distance
\begin{equation}
d_{TV}(p,q) = \frac 12 \sum_{i = 1}^n|p_i-q_i|.
\end{equation}
Distances between area distributions for grains with 5, 6, and 7 sides are
measured with the Kolmogorov-Smirnov metric, which for two cumulative distribution
functions $F,G:[0,\infty)\rightarrow [0,1]$ is given by the distance
\begin{equation}
d_{KS}(F,G) = \sup_{x \in \mathbb R^+} |F(x)-G(x)|.
\end{equation}

\subsection{Topological frequencies of grains }
Figures \ref{fig:tops} and \ref{fig:corrtops} show   topological frequencies
of grains  with snapshots occurring at times $10k\cdot \Delta t_k$ for $k
= 1,\dots, 20$.  Figure \ref{fig:topstimes} shows topologies at $t = 5.364
\times 10^{-6}$, which corresponds to the time when approximately 20\% of
grains remain from the level set method, and also at all simulations' end
time of $t = 2.384\times 10^{-5}$ at which 31,887 grains remain from the
level set method ($\sim 4.8\%$ of initial grains). 

For the level set model, topological frequencies stabilize quickly, with
a minor trend for six sided grains to become less frequent, and five sided
grains more frequent. Both the ND and RD models stabilize quickly, regardless
of whether correlated weights are considered.  Topological frequencies under
the PD model with uncorrelated weights tend to become more uniform with time.
 Adding correlated weights, however, appears to reduce variance.  
 
It has already been observed that higher rates of edge deletion are associated
with more uniform distributions of grain topologies \cite{Fradkov1}.
Thus, it is not surprising that the deletion-free ND model tends to concentrate
near its mean of six sides more than the other models. For uncorrelated weights
at $t = 5.364
\times 10^{-6}$, the ND and PD models differ from the level set method by
a few percentage points, whereas the RD model is substantially more uniform.
However, as frequencies under the PD model become more diffuse in time, as
opposed to other models, by $t =2.384\times 10^{-5}$ differences between
the PD and level set models are magnified. 

The addition of correlated weights appears to have two effects on topological
frequencies.  First, as noted before, diffusion is slowed under the PD model.
 Second, adding correlation for weights reduces variance.  For the RD model,
this reduces accuracy (see Table \ref{table:metrics}) against the level set
model compared to the uncorrelated model, which already a smaller variance
than the level set model. Adding correlation to the PD and RD model, however,
increases accuracy, with the most drastic improvement occurring at $t =2.384\times
10^{-5}$.          

\subsection{Coarsening}

In Figure \ref{fig:fitting}, average grain area versus time is plotted for
the level set and particle system model. For all models, coarsening rates
appear to be linear, with almost no transition period from adjusting to initial
conditions.  The PD and RD models have similar coarsening weights, while
the ND model coarsens significantly at a slower rate.  Adding correlated
weights has the effect of slowing coarsening, with  all particle models having
slower rates than the level set model. 

\subsection{Relative area distributions}
 
Figures \ref{fig:totarea}-\ref{fig:7areacorr} provide snapshots of relative
area densities for remaining grains  at time $t =
5.364
\times 10^{-6}$.
Relative area distributions were also considered in \cite{barmak2006remarks},
but we also include relative areas of grains with 5, 6, and 7 sides. The
Kolmogorov-Smirnov (KS) distances comparing particle models against the level
set model are given in Table \ref{table:metrics}.  For the level set method,
area densities for 5, 6, and 7-sided grains have modes at positive values
and tend to zero as areas approach zero.  In contrast, densities for 5-sided
grains for all particle models appear to be strictly decreasing.  Positive
modes appear for 6-sided grains, but densities for particle systems do not
tend to zero as grain area approaches zero.   Despite the similar shape of
the level set and the ND model under both correlated and uncorrelated weights,
 the KS distance for 7-sided grains is   larger than those of 5 and 6-sided
grains.  This is likely due to the sensitivity of the KS distance for distributions
which are concentrated at a single value. We also note that expect for 5-sided
 grains under the PD and RD models, adding correlations increases the KS
distance.

\section{Conclusions}   
\label{sec:conclusions}
 In this paper, we have developed a framework to  study the coarsening of
two dimensional isotropic grain boundary networks.  The framework combines
PDMP based particle systems, their kinetic limits, and a set of mutation
rules based on topological restrictions. We show in  Appendix \ref{sec:wp}
that the limiting  kinetic equations for  particle densities are well-posed
in an appropriate Banach space. The  parameters of the particle  system can
be varied to produce several models that describe grain boundary coarsening.
All such models rely on mutation matrices $R_{kj}^{(l)}$ which respect the
topological changes during a grain or edge deletion.  The remaining parameters,
  the selection weights $w_j^{(l)}$ and rates $\beta(t)$ of interior events,
may be  chosen to satisfy  other modeling assumptions related to first-neighbor
correlations and edge deletion frequencies. 

The results obtained from considering six separate particle models reveal
that several types of qualitative behavior are dependent on the rate of edge
deletion and first-neighbor correlations.  As expected, increased edge deletion
rates appear to increase the variance of the empirical distributions of one-point
statistics of grain topologies.  For the PD model, this smoothing is continuous
in time, whereas in the RD and ND models, stabilization occurs almost immediately.
 With uncorrelated weights, the PD model remains  similar to the level set
model for some time, but eventually diverges as topological frequencies continue
to smooth out.  The addition of correlated weights appears to retard the
variance of topologies for all models, and as a result, topological frequencies
for the PD model perform substantially better when compared against the level
set model. However, the improvement in statistics is not uniform across all
of the metrics we consider. The uncorrelated PD and RD models, for instance,
coarsen at a rate that is slower than the level set model. The addition of
weights  causes coarsening to slow, which in effect further decreases the
accuracy of coarsening rates for the PD and RD models. Thus, it would be
precipitate at this time to claim a specific particle model is superior.

To conclude, we find that the  set of particle models offered in this paper
are useful in several ways.  Perhaps most importantly, they contain  advantages
 of both the  kinetic models developed in the past three decades and the
 accurate but computationally expensive direct simulation of grain networks.
In particular, our model allows for rigorous examination of limiting densities.
On the other hand, simulation of our model is relatively easy to implement
compared to level-set methods. Implementation requires no discretization
of differential equations, since advection of species occurs at  constant
rates, and mutation times are effectively handled with  Poisson processes.

\begin{figure}
\centering
\subfloat{\includegraphics[width=.45\linewidth]{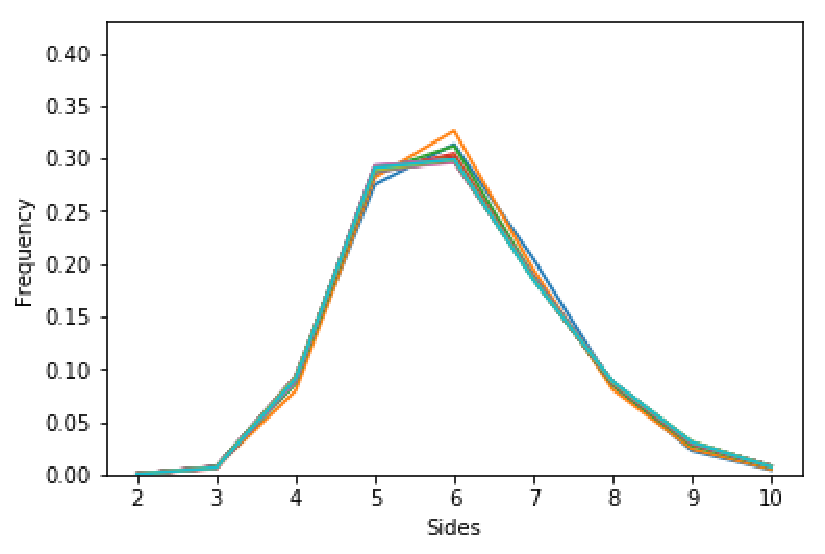}}
\subfloat{\includegraphics[width=.45\linewidth]{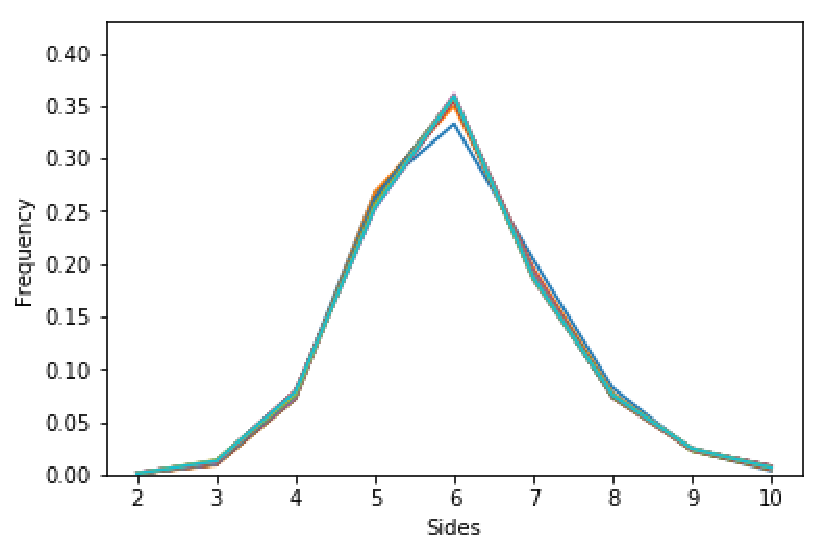}}

\subfloat{\includegraphics[width=.45\linewidth]{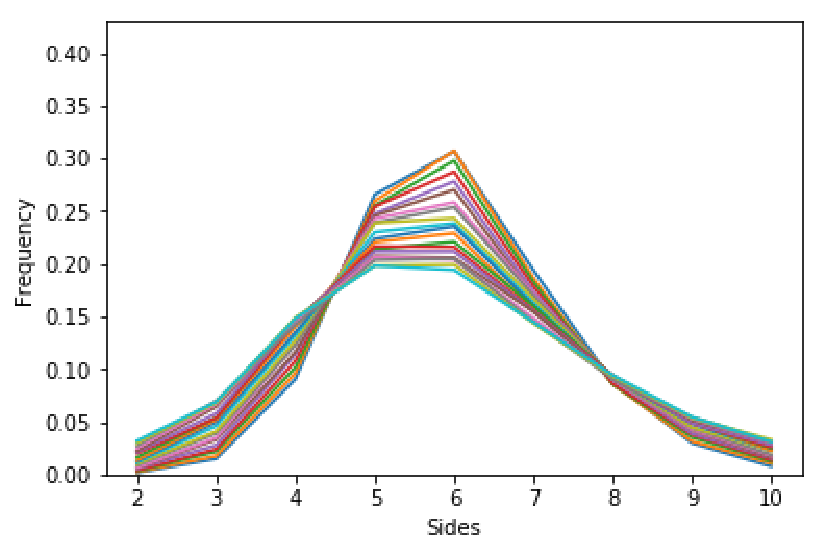}}
\subfloat{\includegraphics[width=.45\linewidth]{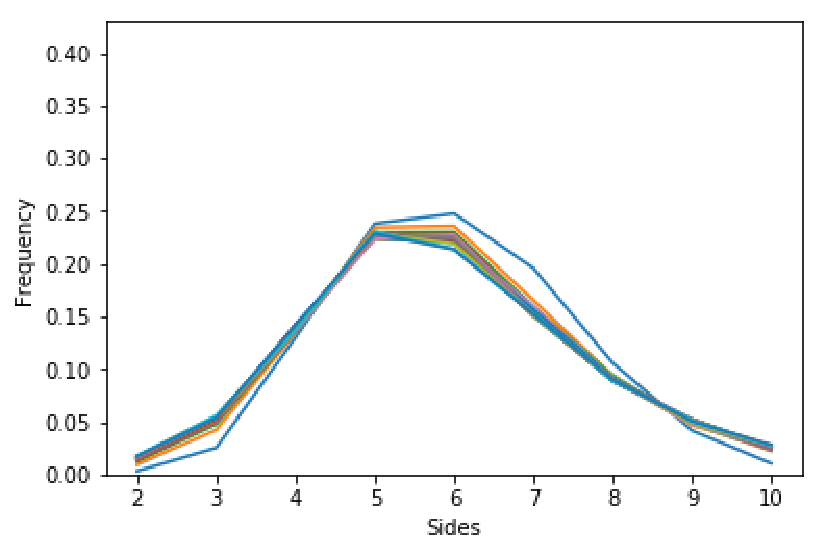}}
\caption{\textbf{The probability that a  grain has $s$ sides.} In each
figure, graphs
correspond to topological frequencies of grains at times $ 10k\cdot \Delta
t$, with
$\Delta t = 1.192\times 10^{-7}$, and $k = 1, \dots, 20.$ Connecting lines
between sides serve as a visual aid. Top left: Level set model. Top right:
ND model. Bottom left: PD model. Bottom right: RD model.}\label{fig:tops}
\end{figure}

\begin{figure}
\centering
\subfloat{\includegraphics[width=.45\linewidth]{elseytops.eps}}
\subfloat{\includegraphics[width=.45\linewidth]{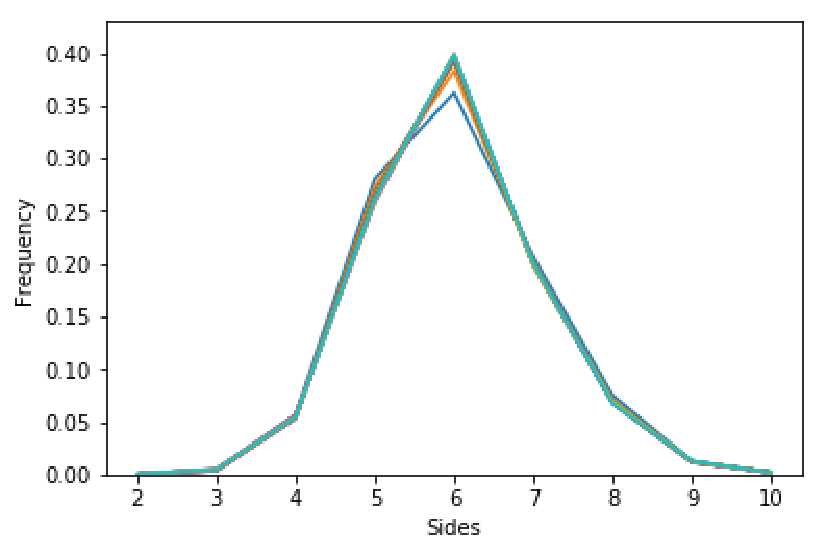}}

\subfloat{\includegraphics[width=.45\linewidth]{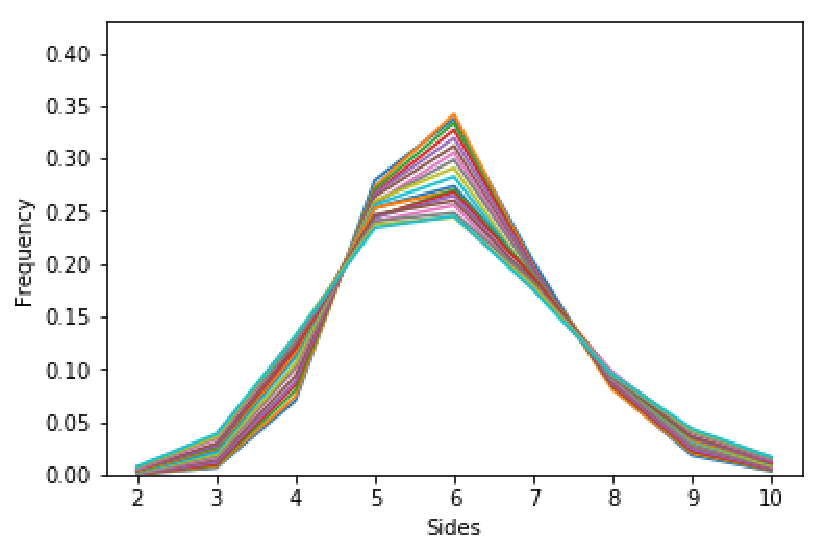}}
\subfloat{\includegraphics[width=.45\linewidth]{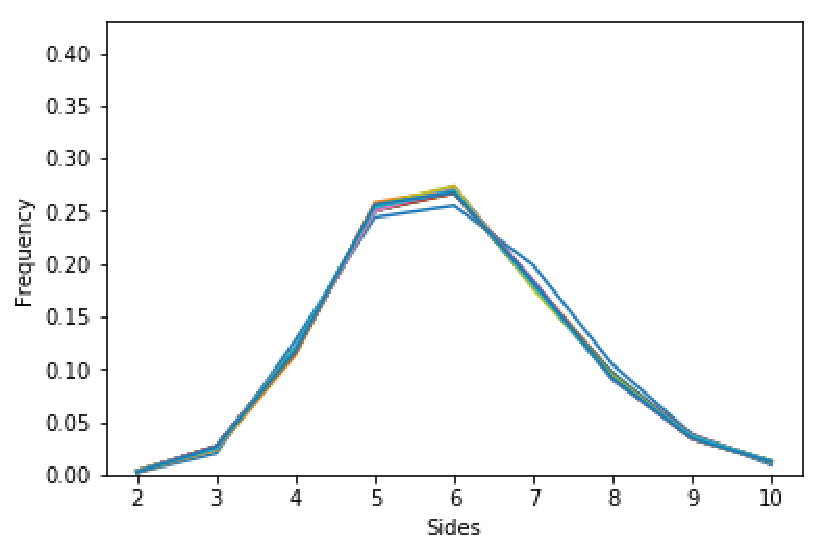}}
\caption{\textbf{The probability that a  grain has $s$ sides when the weights
are correlated.}
In each figure, graphs
correspond to topological frequencies of grains at times $ 10k\cdot \Delta
t$, with
$\Delta t = 1.192\times 10^{-7}$, and $k = 1, \dots, 20.$ Connecting lines
between sides serve as a visual aid. Top left: Level set model. Top right:
ND model. Bottom left: PD model. Bottom right: RD model.}\label{fig:corrtops}
\end{figure}

\begin{figure}
\begin{center}\subfloat{\includegraphics[width=.5\linewidth]{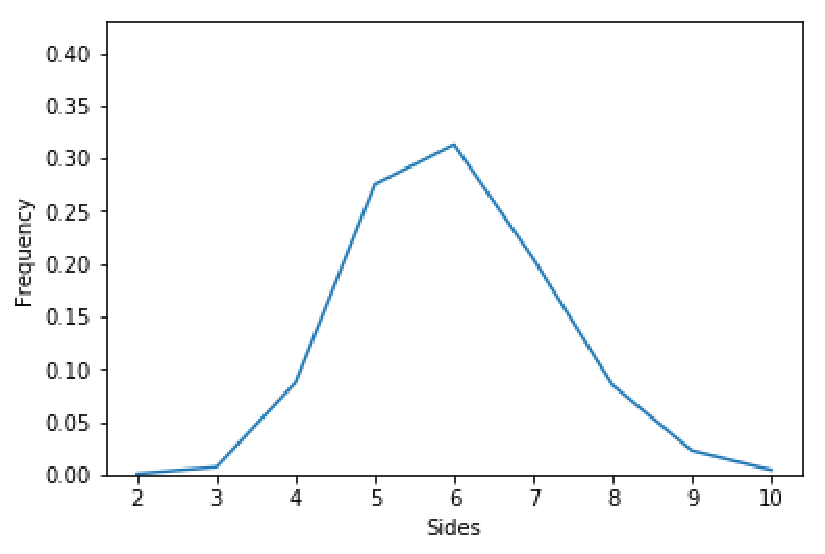}}\end{center}

\subfloat{\includegraphics[width=.49\linewidth]{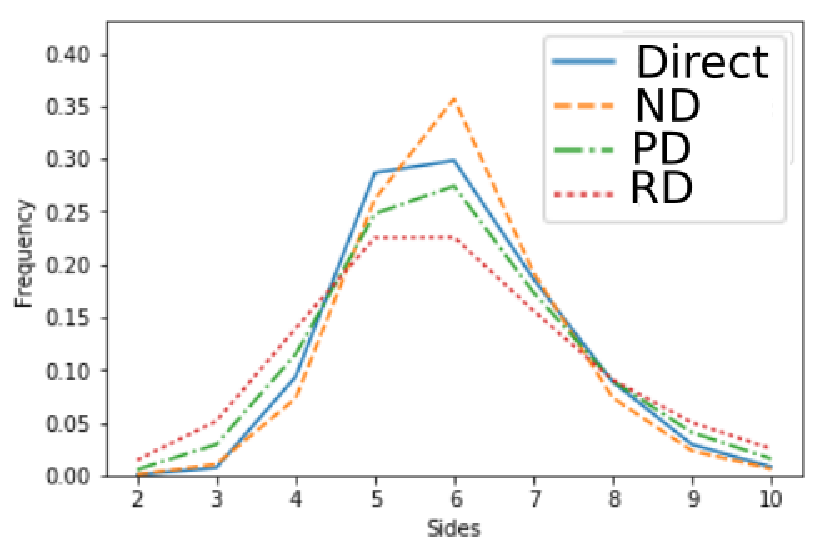}}
\subfloat{\includegraphics[width=.49\linewidth]{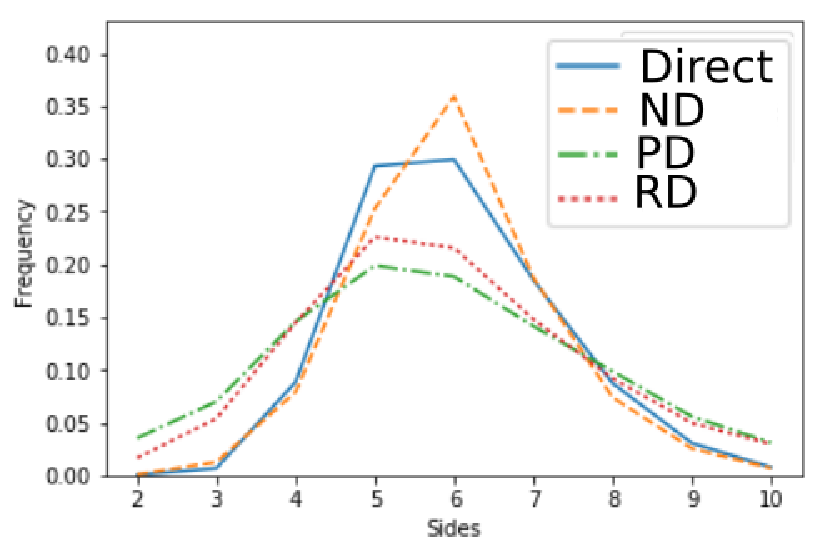}}

\subfloat{\includegraphics[width=.49\linewidth]{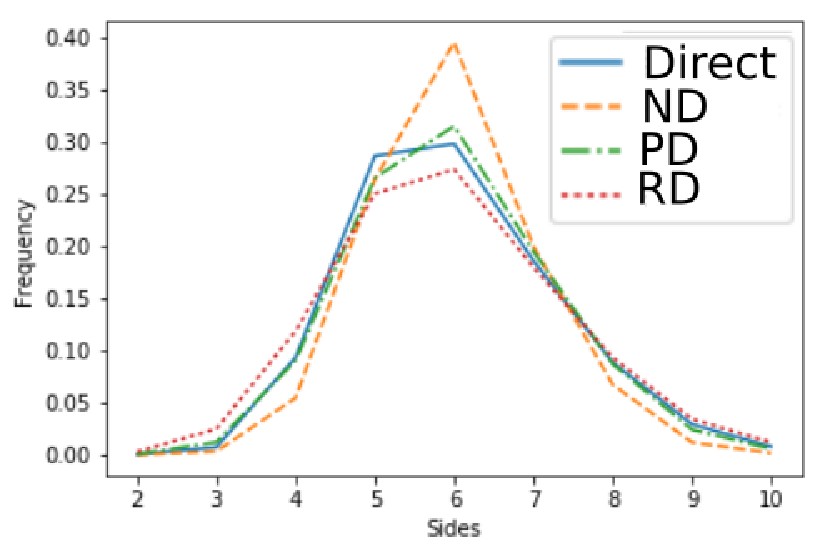}}
\subfloat{\includegraphics[width=.49\linewidth]{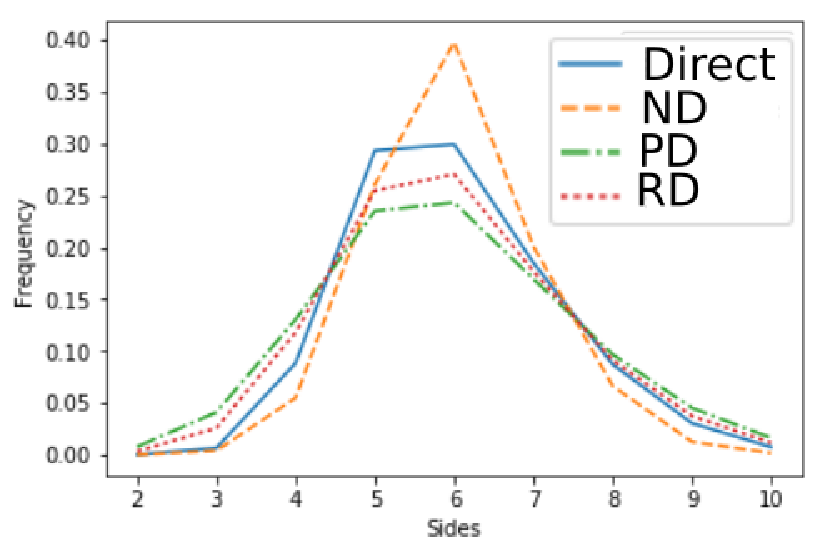}}
\caption{\textbf{Snapshots of grain topologies.} Empirical distributions
as in Figures \ref{fig:tops} and  \ref{fig:corrtops}. Connecting lines are
for visual
aid. Top:\ Initial conditions for all models.  Second row: Snapshots for
models with uncorrelated weights at times $t = 5.364 \times 10^{-6}$ (left)
and $t = 2.384\times 10^{-5}$ (right). Third row:
Snapshots for models with correlated weights at times $t = 5.364 \times 10^{-6}$
(left)
and $t = 2.384\times 10^{-5}$ (right).} \label{fig:topstimes}
\end{figure}

\begin{figure}
\centering
\subfloat{\includegraphics[width=.49\linewidth]{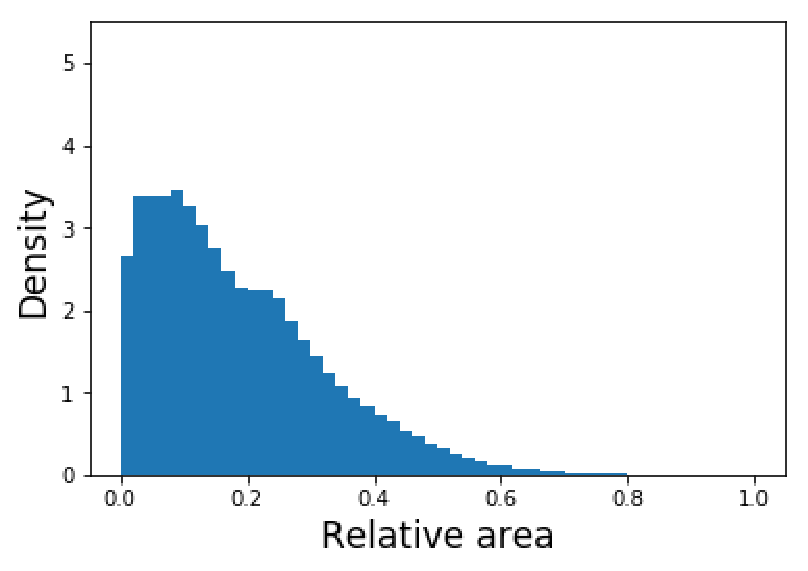}}
\subfloat{\includegraphics[width=.49\linewidth]{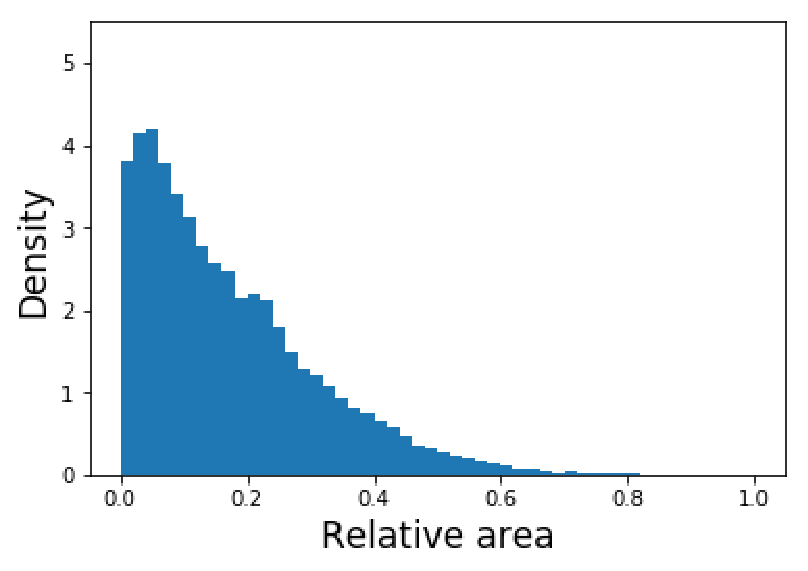}}

\subfloat{\includegraphics[width=.49\linewidth]{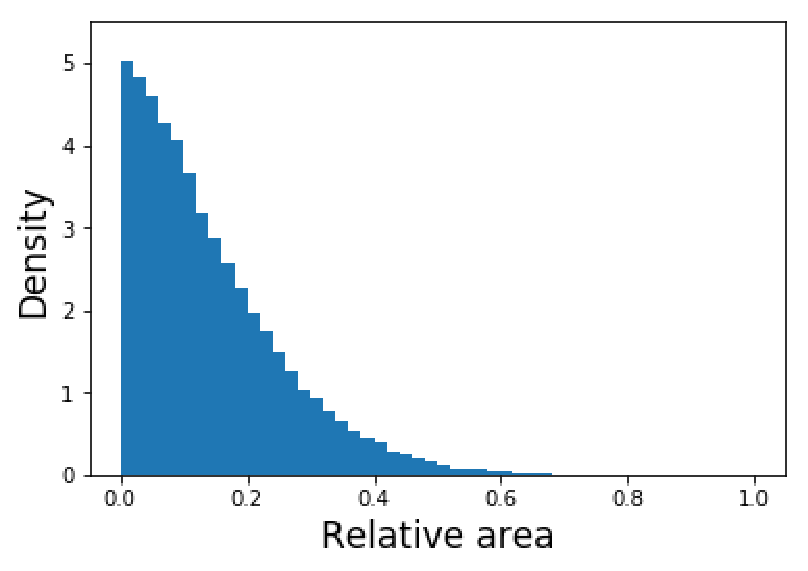}}
\subfloat{\includegraphics[width=.49\linewidth]{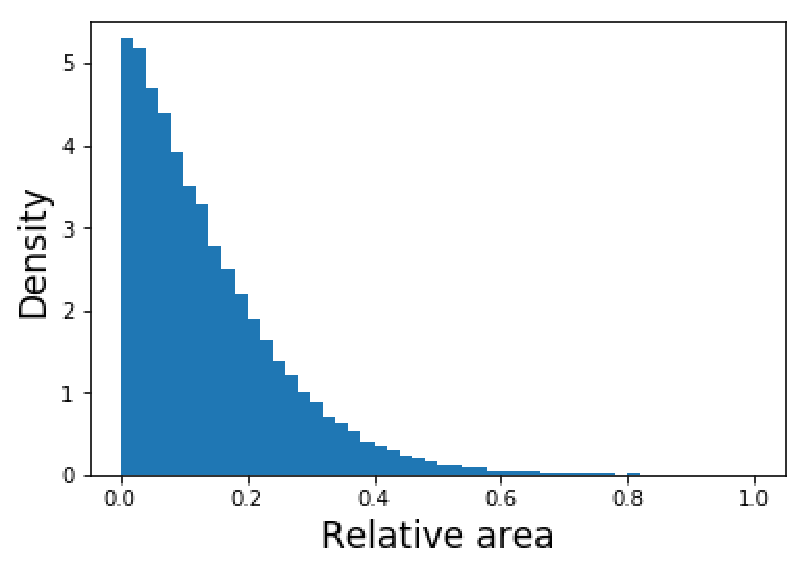}}
\caption{\textbf{Relative area densities for models with uncorrelated weights
of grains at $t = 5.364\times 10^{-6}$. 
}Top left: Level set model. Top right:
ND model. Bottom left: PD model. Bottom right: RD model.}\label{fig:totarea}
\end{figure}

\begin{figure}
\centering
\subfloat{\includegraphics[width=.49\linewidth]{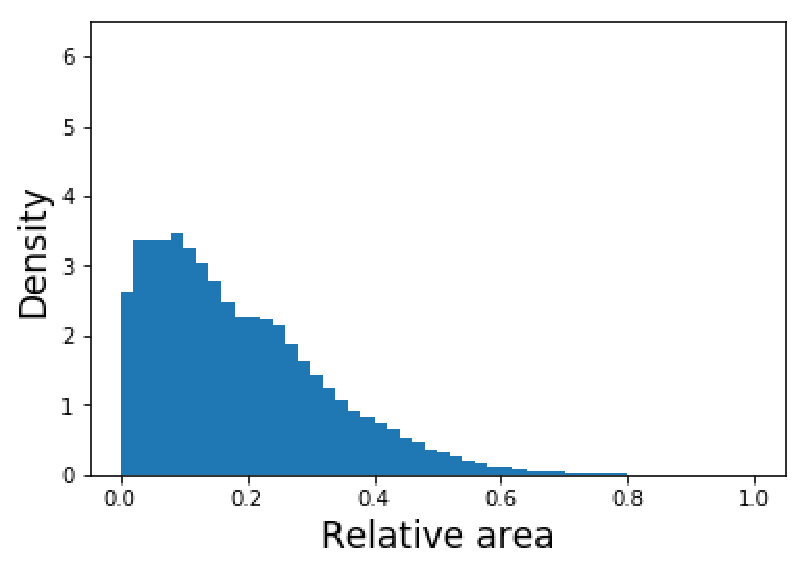}}
\subfloat{\includegraphics[width=.49\linewidth]{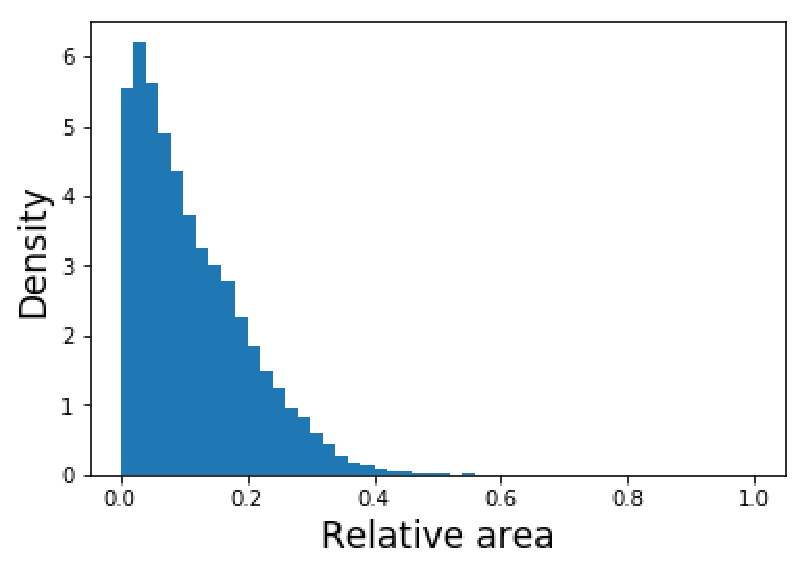}}

\subfloat{\includegraphics[width=.49\linewidth]{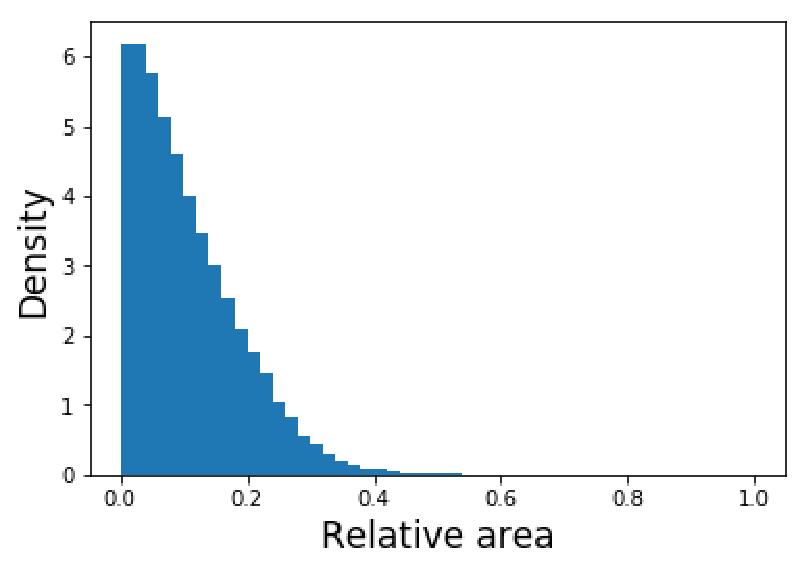}}
\subfloat{\includegraphics[width=.49\linewidth]{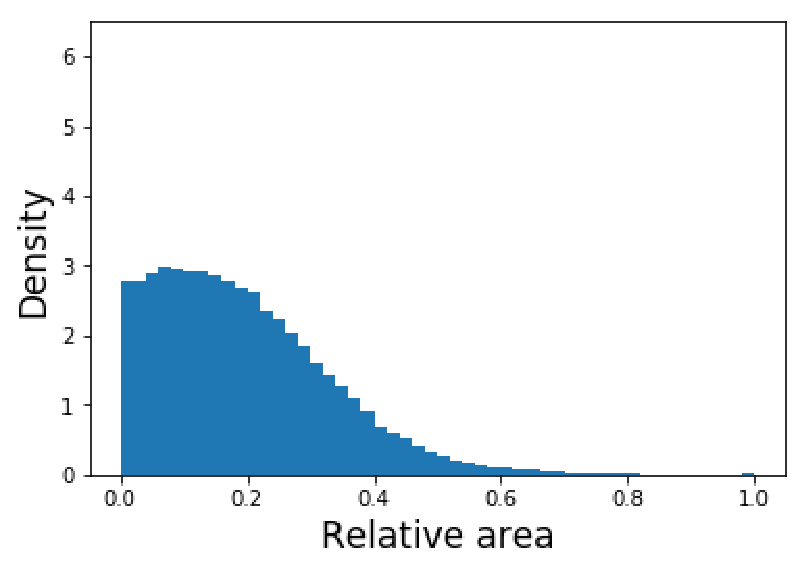}}
\caption{\textbf{Relative area densities for models with correlated weights
of grains at $t = 5.364\times 10^{-6}$.
}Top left: Level set model. Top right:
ND model. Bottom left: PD model. Bottom right: RD model.}\label{fig:totareacorr}
\end{figure}

\begin{figure}
\centering
\subfloat{\includegraphics[width=.49\linewidth]{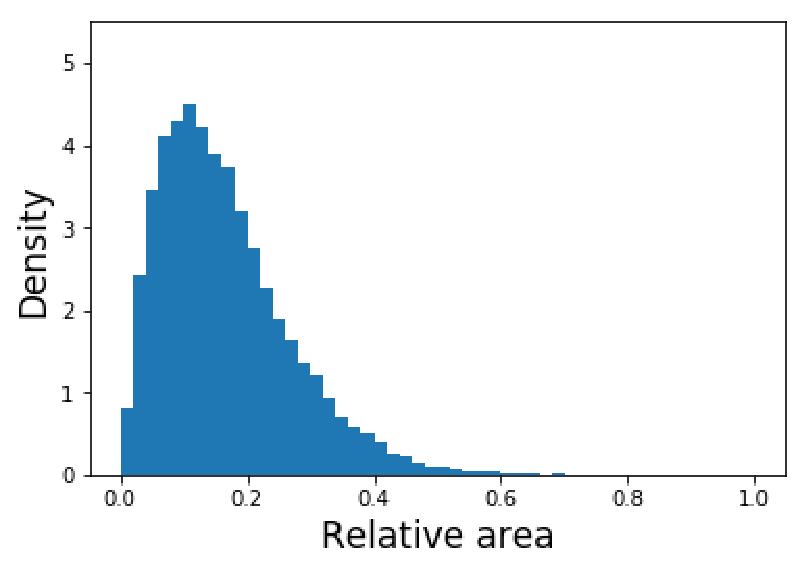}}
\subfloat{\includegraphics[width=.49\linewidth]{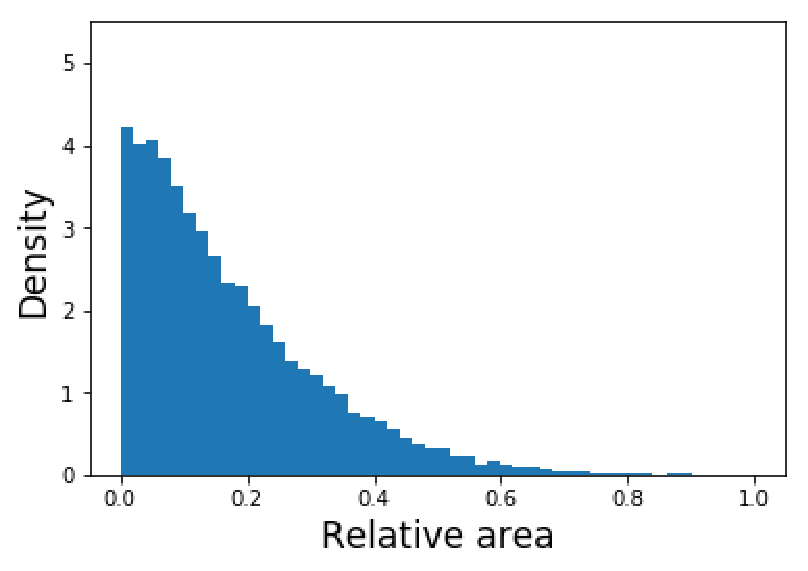}}

\subfloat{\includegraphics[width=.49\linewidth]{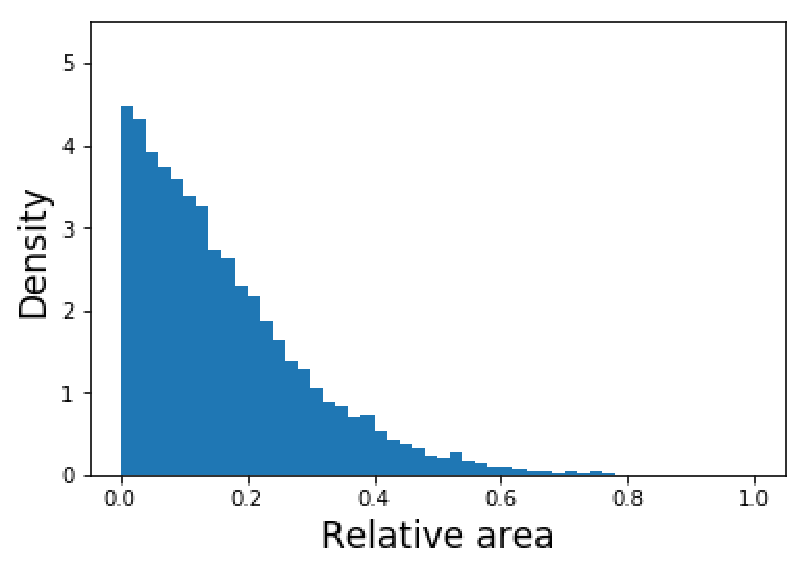}}
\subfloat{\includegraphics[width=.49\linewidth]{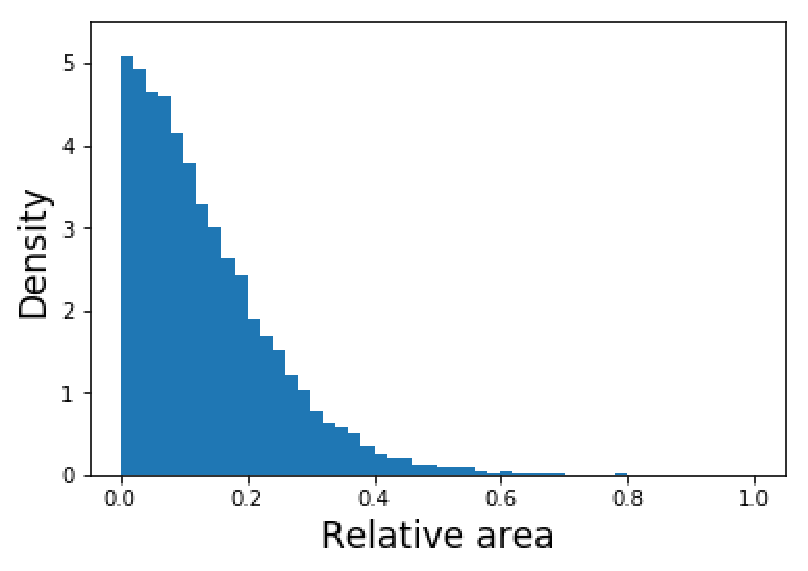}}
\caption{\textbf{Relative area densities for models with uncorrelated weights
of 5-sided
 grains at $t = 5.364\times 10^{-6}$.
}Top left: Level set model. Top right:
ND model. Bottom left: PD model. Bottom right: RD model.}\label{fig:5area}
\end{figure}

\begin{figure}
\centering
\subfloat{\includegraphics[width=.49\linewidth]{elsey5.eps}}
\subfloat{\includegraphics[width=.49\linewidth]{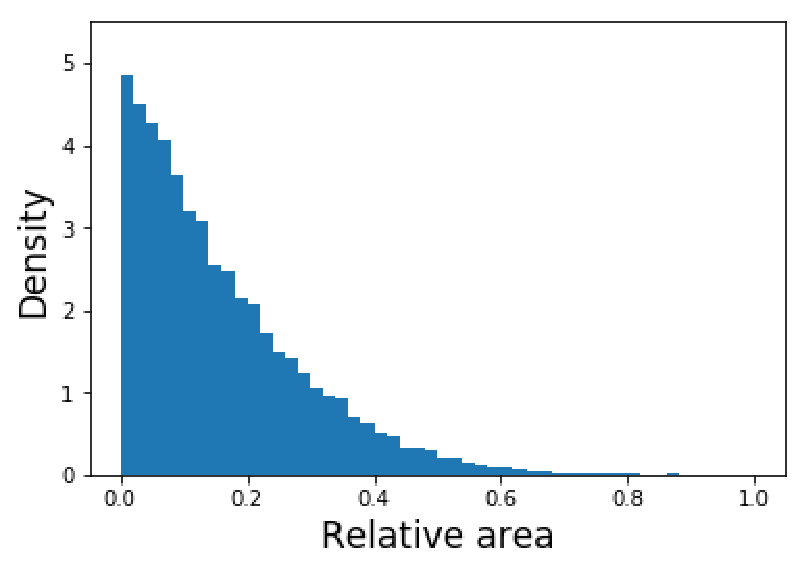}}

\subfloat{\includegraphics[width=.49\linewidth]{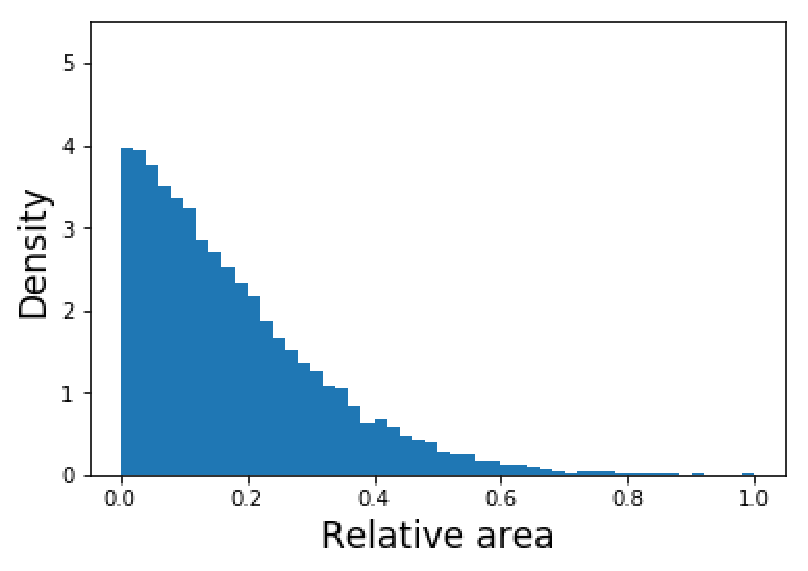}}
\subfloat{\includegraphics[width=.49\linewidth]{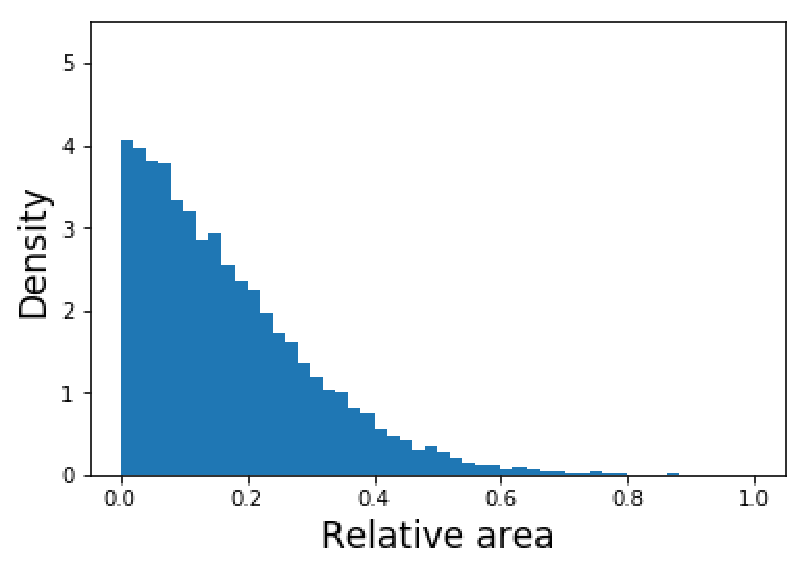}}
\caption{\textbf{Relative area densities for models with correlated weights
of 5-sided
grains at $t = 5.364\times 10^{-6}$.} Top left: Level set model. Top right:
ND model. Bottom left: PD model. Bottom right: RD model.}\label{fig:5areacorr}
\end{figure}

\begin{figure}
\centering
\subfloat{\includegraphics[width=.49\linewidth]{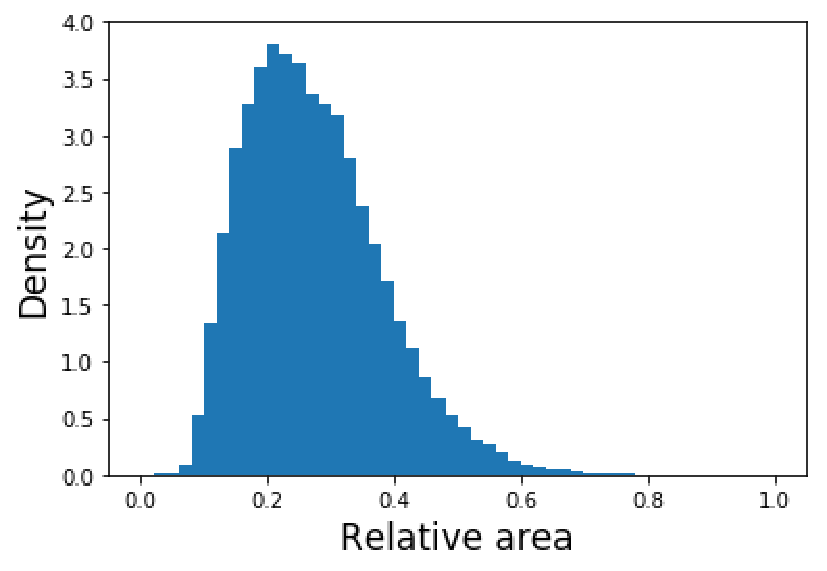}}
\subfloat{\includegraphics[width=.49\linewidth]{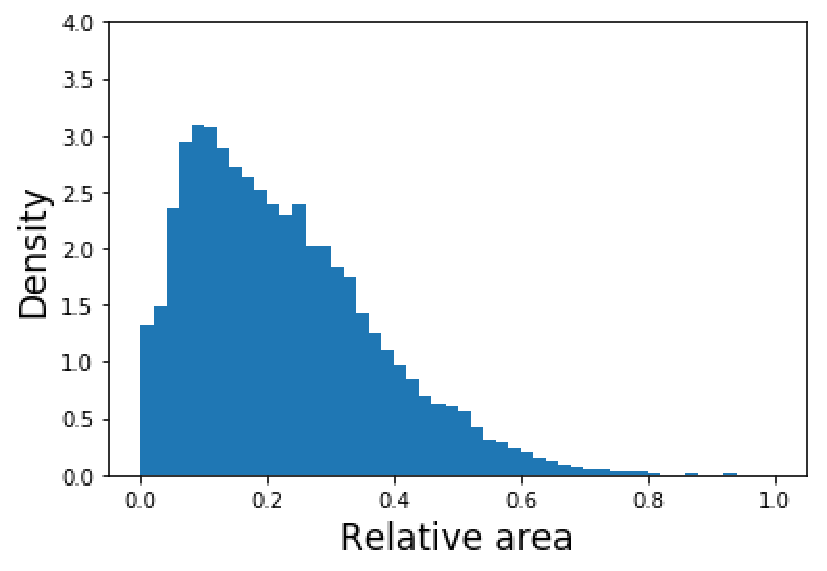}}

\subfloat{\includegraphics[width=.49\linewidth]{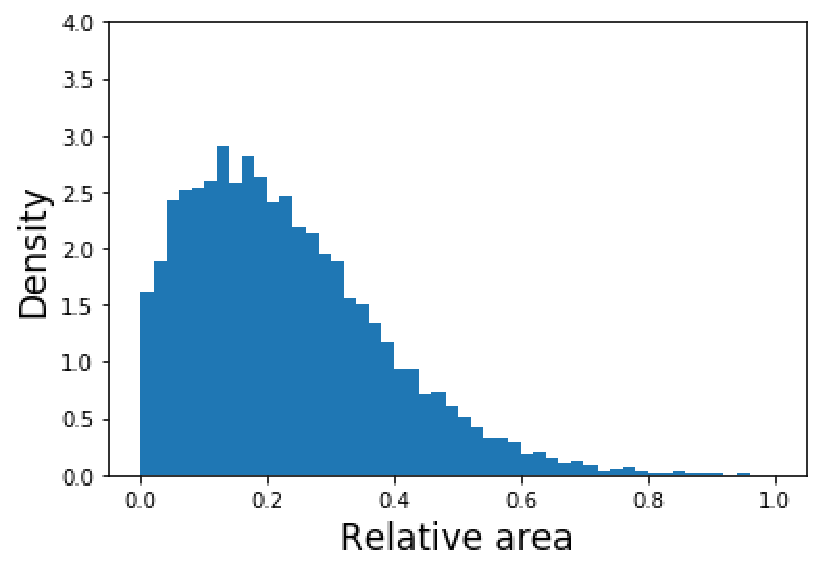}}
\subfloat{\includegraphics[width=.49\linewidth]{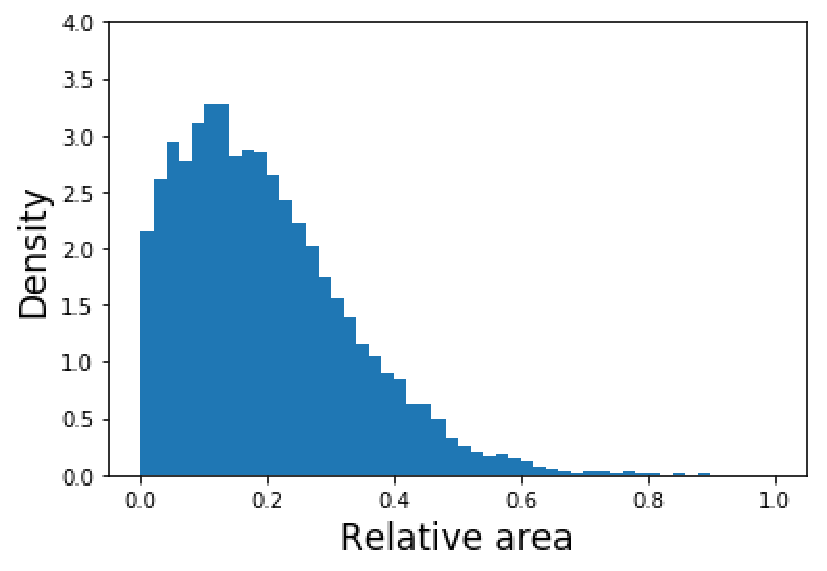}}
\caption{\textbf{Relative area densities for models with uncorrelated weights
of 6-sided
grains at $t = 5.364\times 10^{-6}$.
} Top left: Level set model. Top right:
ND model. Bottom left: PD model. Bottom right: RD model.}\label{fig:6area}
\end{figure}

\begin{figure}
\centering
\subfloat{\includegraphics[width=.49\linewidth]{elsey6.eps}}
\subfloat{\includegraphics[width=.49\linewidth]{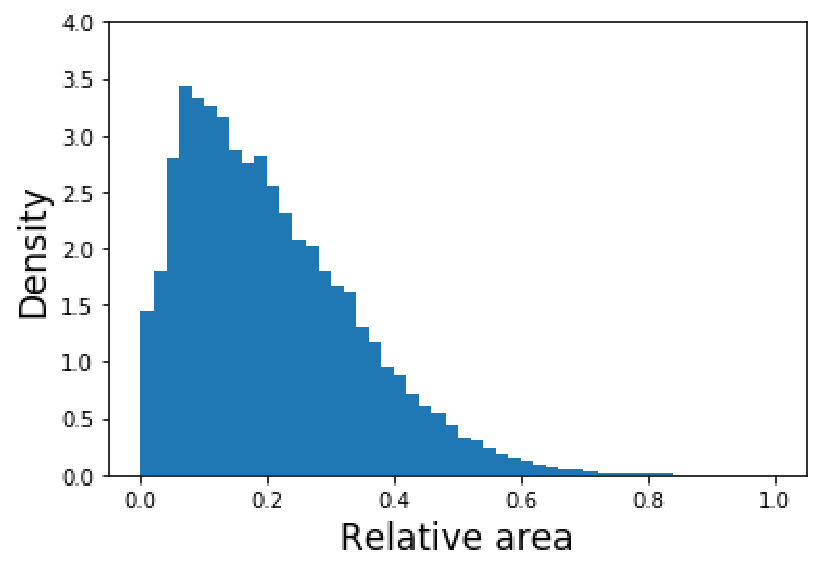}}

\subfloat{\includegraphics[width=.49\linewidth]{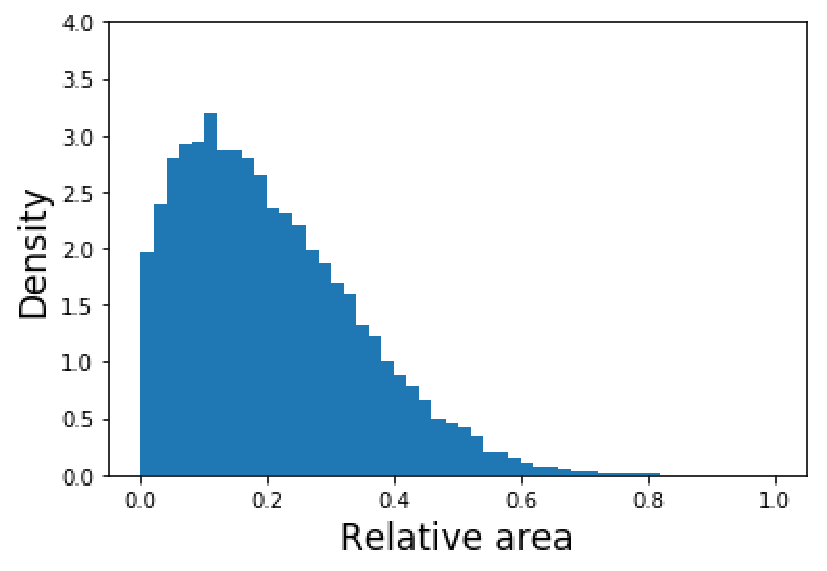}}
\subfloat{\includegraphics[width=.49\linewidth]{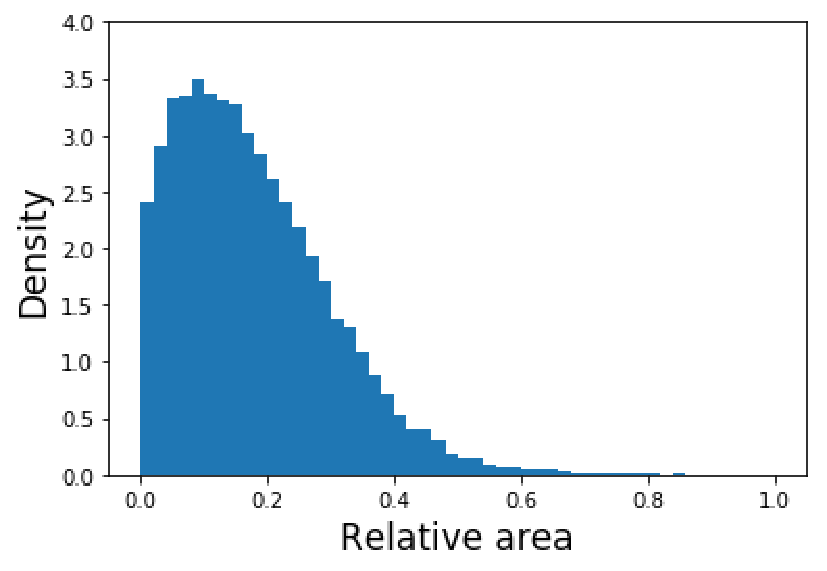}}
\caption{\textbf{Relative area densities for models with correlated weights
of 6-sided
grains at $t = 5.364\times 10^{-6}$.} Top left: Level set model. Top right:
ND model. Bottom left: PD model. Bottom right: RD model.}\label{fig:6areacorr}
\end{figure}

\begin{figure}
\centering
\subfloat{\includegraphics[width=.49\linewidth]{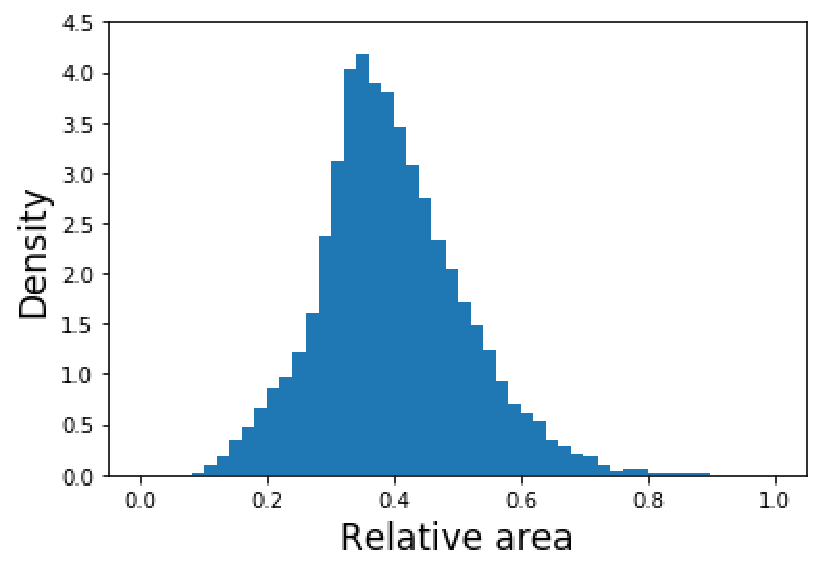}}
\subfloat{\includegraphics[width=.49\linewidth]{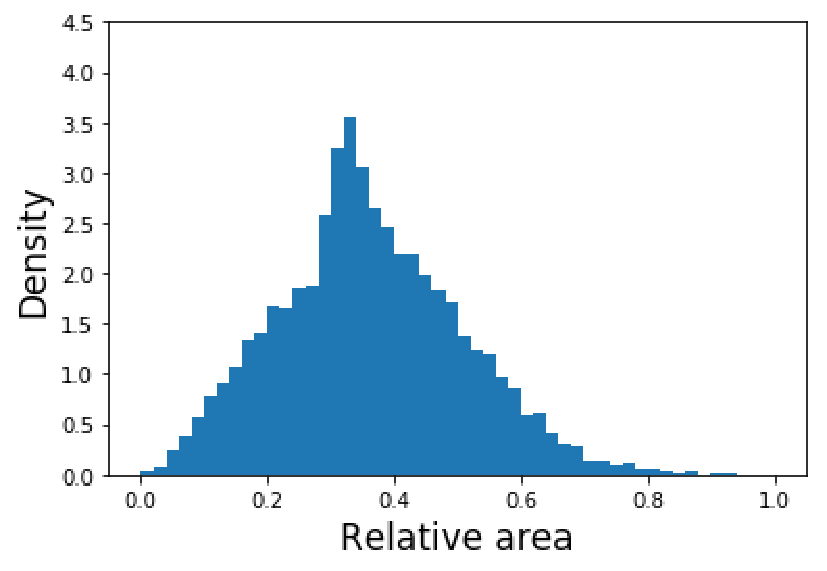}}

\subfloat{\includegraphics[width=.49\linewidth]{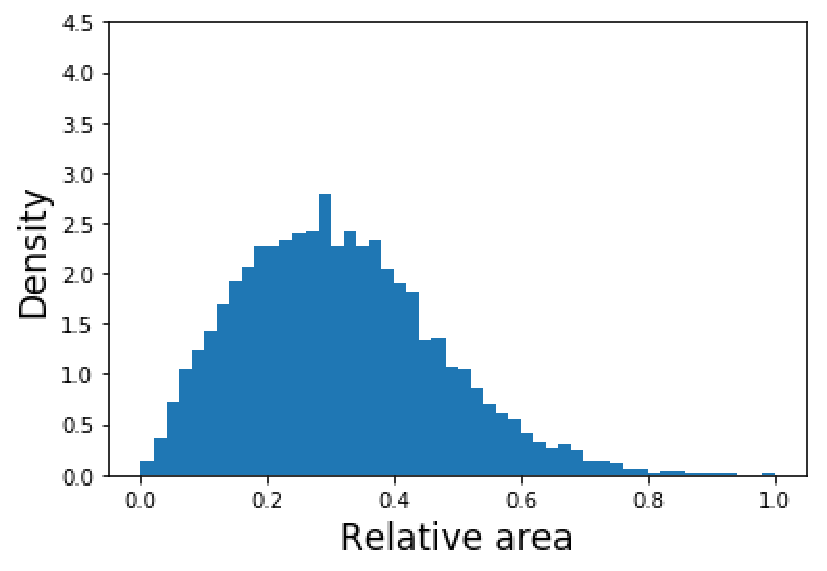}}
\subfloat{\includegraphics[width=.49\linewidth]{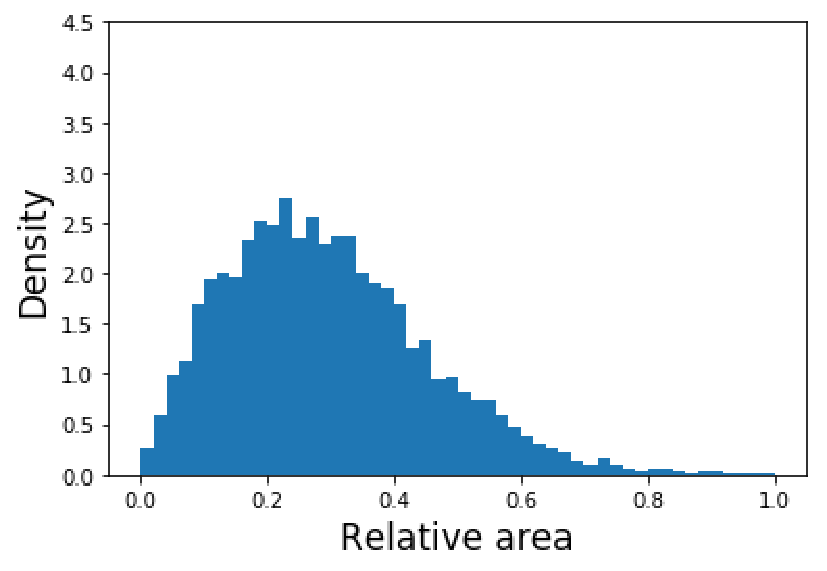}}
\caption{\textbf{Relative area densities for models with uncorrelated weights
of 7-sided
grains at $t = 5.364\times 10^{-6}$.
} Top left: Level set model. Top right:
ND model. Bottom left: PD model. Bottom right: RD model.}\label{fig:7area}
\end{figure}

\begin{figure}
\centering
\subfloat{\includegraphics[width=.49\linewidth]{elsey7.eps}}
\subfloat{\includegraphics[width=.49\linewidth]{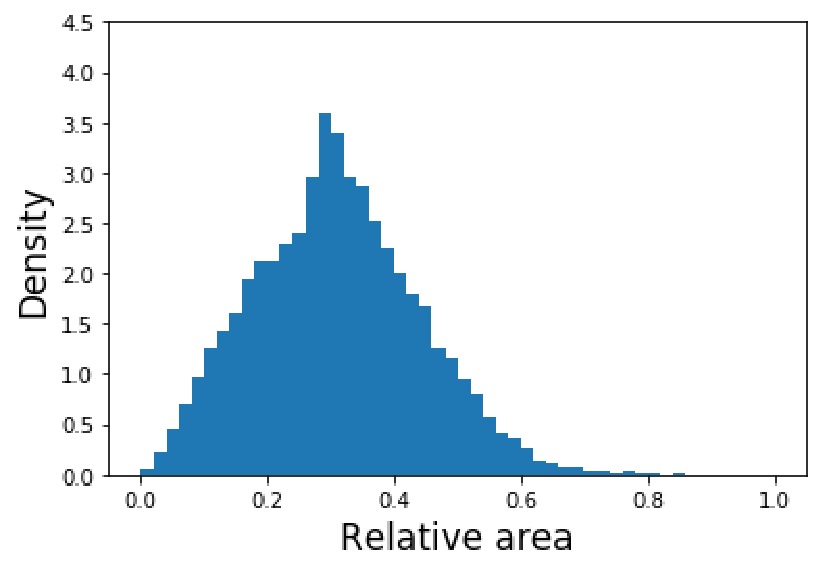}}

\subfloat{\includegraphics[width=.49\linewidth]{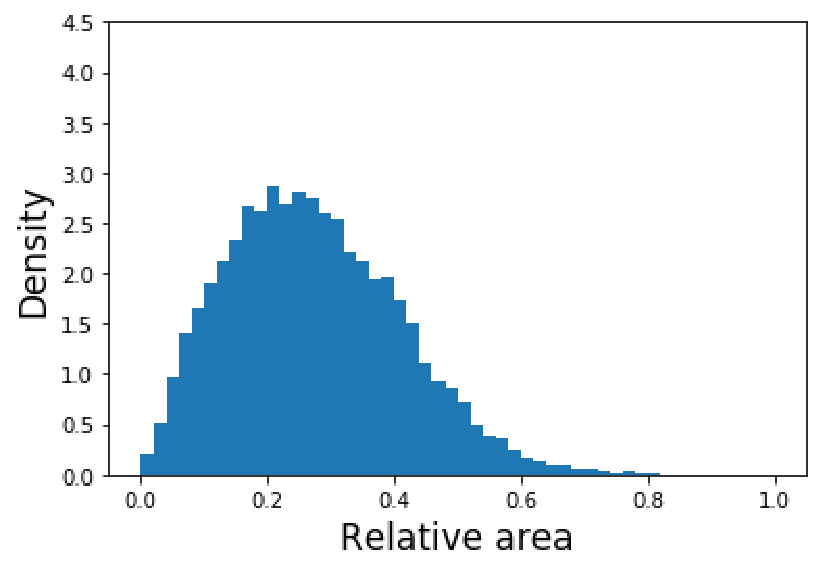}}
\subfloat{\includegraphics[width=.49\linewidth]{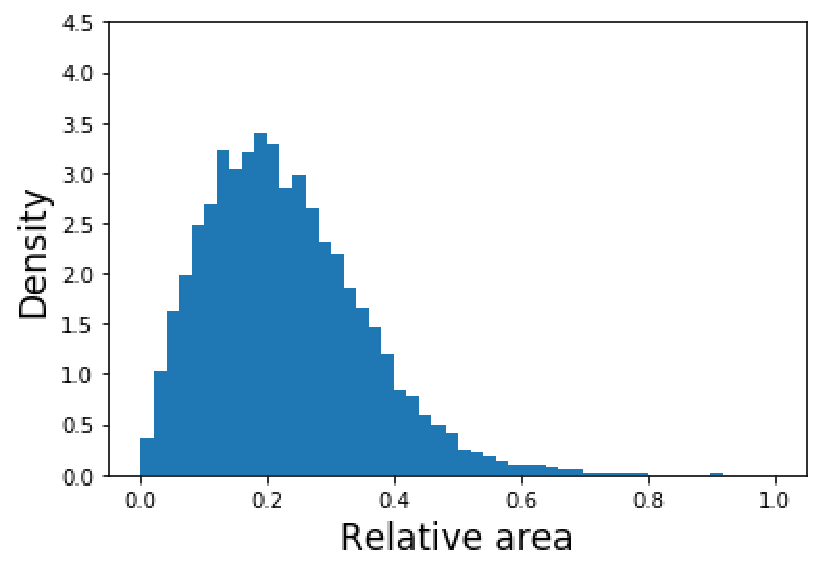}}
\caption{\textbf{Relative area densities for models with correlated weights
of 7-sided
grains at $t = 5.364\times 10^{-6}$.} Top left: Level set model. Top right:
ND model. Bottom left: PD model. Bottom right: RD model.}\label{fig:7areacorr}
\end{figure}

\begin{figure}
\centering
\subfloat{\includegraphics[width=.49\linewidth]{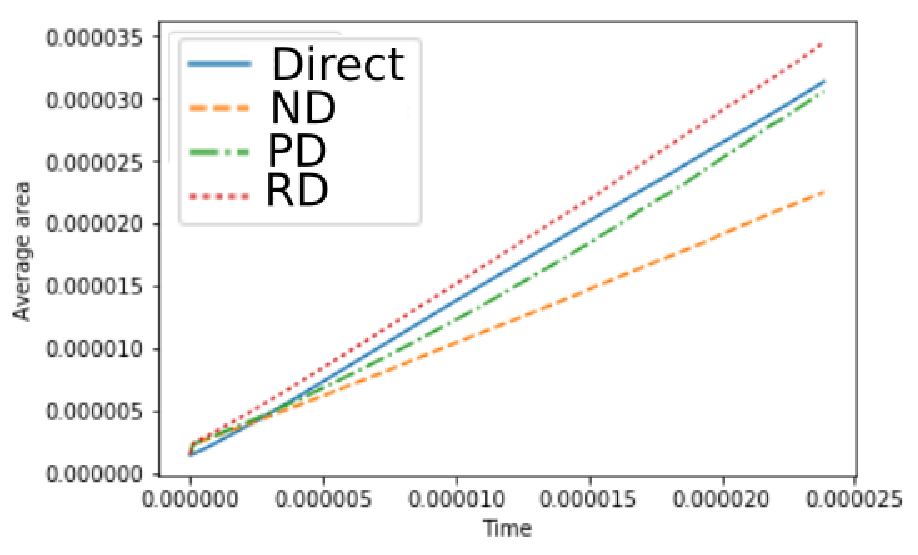}}
\subfloat{\includegraphics[width=.49\linewidth]{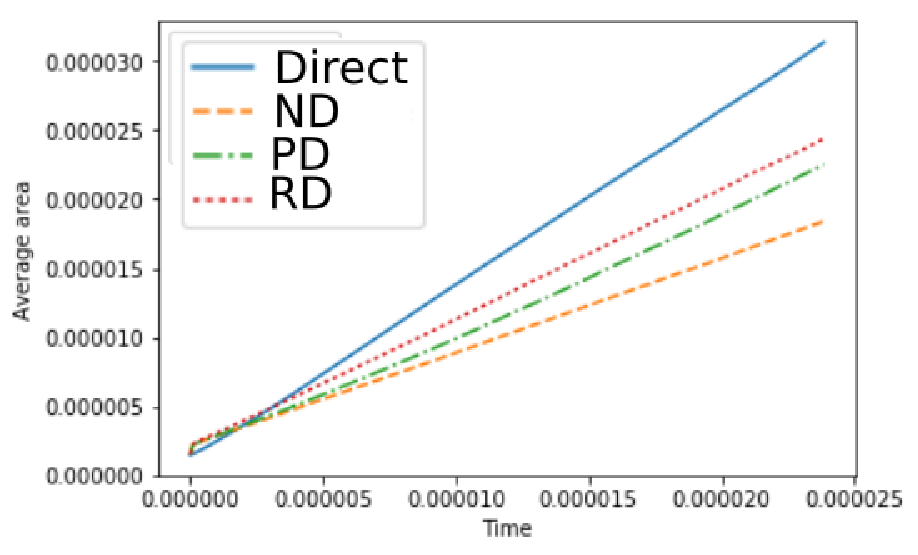}}
\caption{\textbf{Growth of average grain area.} Left: Uncorrelated weights.
Right: Correlated weights. } \label{fig:coarsening}
\end{figure}

\begin{table}
\centering
\begin{tabular}{|c||c|c|c|}\hline
\multicolumn{4}{ |c| }{\textbf{Uncorrelated Weights}}  \\\hline
 & ND & PD & RD  \\\hline
KS distance: 5 sides & .113 & .120 & .170 \\\hline
KS distance:\ 6 sides & .262 & .247 & .322 \\\hline
KS distance: 7 sides & .159 & .327 & .388  \\\hline
KS distance: all grains & .064 & .147 & .156  \\\hline
Total Variation ($t =5.364\times 10^{-6}$)  & .068 & .072 & .154  \\\hline
Total Variation ($t =2.384\times 10^{-5}$)  & .069 & .232 & .177\\\hline
\multicolumn{4}{ |c| }{\textbf{Correlated Weights}}  \\\hline
 & ND & PD & RD  \\\hline
KS distance: 5 sides & .141 & .100 & .104 \\\hline
KS distance:\ 6 sides & .304 & .301 & .369 \\\hline
KS distance: 7 sides & .316 & .430 & .571  \\\hline
KS distance: all grains & .235 & .262 & .226  \\\hline
Total Variation ($t =5.364\times 10^{-6}$)  & .111 & .032 & .062  \\\hline
Total Variation ($t =2.384\times 10^{-5}$)  & .114 & .124 & .072\\\hline
\end{tabular}
\caption{\textbf{Comparison of particle models to level set model}. Rows
with the
Kolmogorov-Smirnov metric measure distances between 5, 6, and 7 sided grain
area distributions
of particles systems and the direct numerical  simulation. Rows with the
total variation
metric measure distances between statistical topologies of particle systems
and the
direct numerical simulation.  }\label{table:metrics}
\end{table}

\section{Acknowledgements}
This work was supported by NSF grants  DMS 1714187 (JK), DMS 1344962 (GM),
and DMS 1515400 and DMS 1812609 (RLP).
GM acknowledges partial support from the Simons Foundation and the Charles
Simonyi Fund at IAS. 
RLP was partially supported by the Simons Foundation  
and by the Center for Nonlinear Analysis under NSF PIRE Grant no.\ OISE-0967140.

\section{Conflicts of Interest}
The authors declare that they have no conflict of interest.

\appendix
\section{Description of the $M$-species system as a PDMP} 
\label{app:pdmpexplain}                 
We briefly review the basics of PDMPs, following Davis~\cite{Davis}, and
then explain how the $M$-species stochastic particle process of Section~\ref{subsec:model}
fits into this framework. 

\subsection{Background: General theory of PDMPs} 
\label{subsec:bg}

We consider a countable set  $\mathcal{S}$ with elements denoted $\mathbf{s}$, a map $\mathbf{d}:
\mathcal{S} \rightarrow \mathbb{N}$, and open sets for each $\mathbf{s}$ of the form $M_{\mathbf{s}}
\subset
\mathbb{R}^{\mathbf{d}(\mathbf{s})}$. The state space is the disjoint union
\begin{equation} E =\coprod_{\mathbf{s} \in \mathcal{S}} M_{\mathbf{s}} = \left\{(\mathbf{s},\mathbf{x}):\mathbf{s}
\in \mathcal{S}, \mathbf{x} \in M_{\mathbf{s}} \right\}. \end{equation} 
The space $E$ has a natural topology.   Let
$\iota_\mathbf{s}:M_{\mathbf{s}} \rightarrow E$ be the canonical injection defined by
$\iota_\mathbf{s} (\mathbf{x}) = (\mathbf{s},\mathbf{x})$. A set $A \subset E$ is open if
for every $\mathbf{s}$, $\iota_\mathbf{s}^{-1}(A)$ is open in $M_{\mathbf{s}}$. The collection of
all open sets may be used to define the set $\mathcal E$ of Borel subsets
of $E$. This makes $(E,\mathcal E)$ a Borel space.  

A PDMP is an $E$-valued generalized jump process $X(t) = (\mathbf{s}(t),\mathbf{x}(t))$,
$t\geq 0$, that is prescribed by:  
\begin{enumerate} 
\item Sufficiently smooth vector fields $\mathbf{v}_{\mathbf{s}}: M_{\mathbf{s}} \to \mathbb{R}^{\mathbf{d}(\mathbf{s})}$, $\mathbf{s} \in
\mathcal{S}$.
\item A measurable function $\lambda:E \rightarrow
\mathbb R^+$. 
\item A transition measure $Q: \mathcal E \times (E \cup \Gamma^*)
\rightarrow [0,1]$. Here $\Gamma^*$ denotes the exit boundary defined in
equations \ref{eq:exit1}--\ref{eq:exit2} below. 
\end{enumerate}
Points in $M_{\mathbf{s}}$ travel according to flows defined by the vector fields $\mathbf{v}_{\mathbf{s}}$
until either a Poisson clock with intensity $\lambda(\mathbf{s},\mathbf{x})$ rings or
the point $\mathbf{x}(t)$ hits the exit boundary $\Gamma^*$. When such a {\em critical
event\/} occurs the point $X(t)$ jumps to a random new position whose law
is given by $Q$.

Each vector field $\mathbf{v}_{\mathbf{s}}$ may be viewed as a first-order  differential  operator
on $M_{\mathbf{s}}$.
We assume they define a flow $\varphi_\mathbf{s}(t,\mathbf{x})$ such that
\begin{equation}
\label{eq:flow-def}      
\frac{\partial}{\partial t} h \left(\mathbf{s},\varphi_\mathbf{s} (t,\mathbf{x}) \right)
= \mathbf{v}_{\mathbf{s}}\left(h \left( \varphi_\mathbf{s} (t,\mathbf{x})\right)\right), \quad \varphi_\mathbf{s}
(0,\mathbf{x}) = \mathbf{x},  
\end{equation}  
for all sufficiently smooth test functions $h$ and for  $t$ in a maximal
interval of existence. The flow terminates only when $\mathbf{x}(t)$ hits 
\begin{equation} 
\label{eq:exit2}
\partial^*M_{\mathbf{s}} =\left \{\mathbf{y} \in \partial M_{\mathbf{s}}:
\varphi_\mathbf{s} (t^-, \mathbf{x}) = \mathbf{y} \quad \hbox{ for some } (t, \mathbf{x}) \in
\mathbb{R}_+ \times M_{\mathbf{s}}\right \}. \end{equation} 
The {\em exit boundary\/} is the disjoint collection 
\begin{equation} 
\label{eq:exit1}
\Gamma^* = \coprod_{\mathbf s \in \mathcal{S}} \partial^{*}M_{\mathbf{s}} = \left\{(\mathbf{s},\mathbf{x}):\mathbf
s \in \mathcal{S}, \mathbf{x} \in \partial^*M_{\mathbf{s}} \right\},
\end{equation}
At a given state $(\mathbf{s},\mathbf{x}) \in E$ we define the first exit time 
\begin{equation} 
\label{eq:pdmp-first-exit}      
        t_\mathbf{s}^*(\mathbf x) = \sup \{t >0:
\varphi_\mathbf{s} (t, \mathbf{x}) \in M_{\mathbf{s}}\},
\end{equation} 
and the survivor function
\begin{equation} 
\label{eq:pdmp-survivor}          
\mathcal{F}_{\left(\mathbf{s}, \mathbf{x}\right)}
(t) = \begin{cases}
\exp\left(-\int_0^t \lambda \left( \mathbf{s},\varphi_\mathbf{s} \left(\tau,\mathbf{x}\right)
 \right)  \,d\tau \right), &
t<t_\mathbf{s}^*(\mathbf{x}), \\ 0, & t \ge t^*(\mathbf{x}). \end{cases} 
\end{equation}

The stochastic process $(X(t))_{t \ge 0}$ with initial condition $X(0) =
(\mathbf{s}_0,\mathbf{x}_0)$ is defined as follows. Choose a random time $T_0$ such that
$\mathbb{P}[T_0>t] =  \mathcal{F}_{\left(\mathbf{s}, \mathbf{x}_0\right)}(t)$ and  an $E$-valued random variable $(\mathbf{s}_1,\mathbf{x}_1)$
with
law $Q(\cdot \,; \varphi_{\mathbf{s}_0}(T_0,\mathbf{x}_0))$ that is independent of $T_0$.
The trajectory of $X(t)$ for $t \le T_0$ is then 
\begin{equation} X(t) = \begin{cases} 
        \left(\mathbf{s}_0,\varphi_{\mathbf{s}_0}(t,\mathbf{x}_0)\right), &
t<T_0, \\ (\mathbf{s}_1,\mathbf{x}_1), & t = T_0. \end{cases} 
\end{equation} 
At $t=T_0$, we repeat this process,  replacing the jump time $T_0$  in the
algorithm above with $T_1-T_0$ and the state $(\mathbf{s}_0,\mathbf{x}_0)$ with $(\mathbf{s}_1,\mathbf{x}_1)$.
Iterating this process, jump by jump,  yields a cadlag process  $X(t)$, $t
\in [0, \infty)$. 

Under modest assumptions, it can be shown that $X(t)_{t \geq 0}$ is a strong
Markov process~\cite[\S 3]{Davis}. We only require that $Q\left(A;(\mathbf{s},\mathbf{x})\right)$
is a
measurable function of $(\mathbf{s},\mathbf{x})$ for each Borel set $A \in \mathcal E$
and a probability measure on $(E, \mathcal E)$ for each $(\mathbf{s},\mathbf{x}) \in
E \cup \Gamma^*$.  The rate function $\lambda:E \rightarrow \mathbb R^+$
must be measurable  with a little integrability: specifically,  for each
state $(\mathbf{s}, \mathbf{x}) \in E$ we require the existence of $\varepsilon > 0$ such
that the
function $\tau \rightarrow \lambda(\mathbf{s}, \varphi_\mathbf{s}(\tau, \mathbf{x}))$ is
summable for $\tau \in [0,\varepsilon)$.  These conditions are easily verified
in our model.

\subsection{The $M$-species model as a PDMP }
\label{subsec:Mspec}
We now show the $M$-species model defined in Section~\ref{subsec:model} is
a PDMP. Define the countable set of {\em species indices\/}
\begin{equation} 
\label{eq:speci-index-def}      
\mathcal{S}= \bigcup_{m \in \mathbb N}\{1, \dots, M\}^{m}.
\end{equation} 
It is convenient to introduce notation that makes explicit the distinction
between the number of particles in a state $(\mathbf{s},\mathbf{x})$ and the fixed parameter
$N$ that is the normalizing factor in the empirical measure \ref{eq:empirical}.
We denote the number of particles in the state $(\mathbf{s},\mathbf{x})$ by $|\mathbf{s}|$ and
write
\begin{equation}
        \label{eq:new-state}    
        (\mathbf{s},\mathbf{x}) = \left(s_1, \ldots, s_{|\mathbf{s}|}; x_1, \ldots, x_{|\mathbf{s}|})\right),
\end{equation}  
and the associated empirical measures is
\begin{equation}
        \label{eq:new-empirical}
     \mu^N_\sigma(\mathbf{s},\mathbf{x}) = \frac{1}{N}\sum_{i=1}^{|\mathbf{s}|} \delta_{x_i}
\mathbf{1}_{s_i=\sigma}, \quad \sigma =1, \ldots, M.  
\end{equation} 
For the $M$-species process, $N(t) = |\mathbf{s}|(t)$, and equations \ref{eq:state},\ref{eq:empirical}
and equations \ref{eq:new-state}--\ref{eq:new-empirical}are consistent.

Similarly, each open set $M_{\mathbf{s}} = \mathbb{R}_+^{|\mathbf{s}|}$ and 
\begin{equation}
E =\coprod_{\mathbf{s} \in \mathcal{S} }\mathbb{R}_+^{|\mathbf{s}|} = \left\{(\mathbf{s},\mathbf{x}): \mathbf{s} \in
\mathcal S, \, \mathbf{x} \in \mathbb{R}_+^{|\mathbf{s}|}\right\}.  
\end{equation}

The velocity fields $\mathbf{v}_{\mathbf{s}}$ on $E$ are obtained from the velocity fields $v_s$,
$s=1,\ldots,M$  of the $M$-species model, 
\begin{equation}
        \label{eq:vel}
\mathbf{v}_{\mathbf{s}} = \sum_{i = 1}^{|\mathbf{s}|} v_{s_i}(x_i)\frac{\partial}{\partial x_i},
\end{equation}
and the exit boundary is
\begin{equation}
\Gamma^* = \{(\mathbf{s},\mathbf{x})\in E| \;\hbox{there exists } \; (s_i,x_i)\; \hbox{such
that}   \;x_i = 0,\, s_i \in S_-\}. 
\end{equation}

In order to define the transition kernel $Q$,  we first describe the finite
set of `neighbors' $E^\partial_{\mathbf{s},\mathbf{x}}$ for each state $(\mathbf{s},\mathbf{x}) \in
\Gamma^*$. Each point $(\mathbf{s},\mathbf{x})$ has a finite number, $p$, of particles
with size zero. Let us label these particles with indices $i=k_1$,$k_2$,$\ldots$,
$k_p$, ordered such that the species $s_{k_1} \leq s_{k_2}\leq \ldots s_{k_p}$.
Let us begin by discussing the case when $p=1$ (this is the most important
case, since boundary events happen at distinct times with probability $1$).
When $p=1$, the set $E^\partial_{\mathbf{s},\mathbf{x}}$ may be decomposed into $M_-$
subsets, corresponding to boundary events at $M_-$ species. More precisely,
a boundary event occurs at species $l$, if the size $x_{j_1}=0$ and the associated
species $s_{j_1}=l$. According to the rules of Section~\ref{subsec:model},
at such a boundary event, $K^{(l)}$ random variables $(S_j,X_j)$ are chosen,
and mutated as in equation \ref{eq:mutation}. 
Each such mutation gives rise to a neighbor $(\mathbf{r},\mathbf{y})$ of $(\mathbf{s},\mathbf{x})$.
Since the $X_j$ are a random collection of $K^{(l)}$ points of $\mathbf{x}$, we may
write  
$X_j = x_{i_j}$, for indices $i_1,\ldots, i_{K^{(l)}}$. Then $(\mathbf{r},\mathbf{y})$ is
obtained from $(\mathbf{s},\mathbf{x})$ in two `sub-steps':
\begin{enumerate}
\item[(i)] {\em Pure mutation:\/} $\mathbf{x}$ is unchanged. The coordinates of
$\mathbf{s}$ are changed as follows: $s_{i_j} \mapsto R^{(l)}_j$, $j=1, \ldots,
K^{(l)}$.  Call this intermediate state $\hat{\mathbf{s}}$.
\item[(ii)] {\em Removal of zero size:\/} $\mathbf{x}$ is changed to $\mathbf{y}$ by deleting
the particle $x_{j_1}$ with size zero.  
\end{enumerate}
The probability $p^\partial(\cdot; \mathbf{s},\mathbf{x})$ of each  transition $(\mathbf{s},\mathbf{x})
\mapsto (\mathbf{r},\mathbf{y}) \in E^\partial_{\mathbf{s},\mathbf{x}}$ is given by the rules of Section~\ref{subsec:model}.
Finally, observe that these rules extend naturally to degenerate boundary
points, where $0=x_{k_1} = x_{k_2} = \ldots x_{k_p}$. In this case, according
to the rules of Section~\ref{subsec:model}, we order the points $x_{j_1},\ldots,
x_{j_p}$ so that the species $s_{j_1} < s_{j_2} < s_{j_p}$, and mutate and
remove particles $p$ times in sequence as above. 

Similarly, given an interior point $(\mathbf{s},\mathbf{x}) \in E$ we can use the mutation
matrix $R^{(0)}$ and the weights $w^{(0)}$ to define a set of interior
points $E^{(0)}_{\mathbf{s},\mathbf{x}}$ that $(\mathbf{s},\mathbf{x})$ jumps to along with the corresponding
probabilities  $p^{(0)}(\cdot;\mathbf{s},\mathbf{x})$. In this case, the transition involves
only a mutation and no removal of zero sizes.

In summary,  the transition kernel is given by
\begin{equation}  
\label{eq:transition}   
Q(A; \mathbf{s},\mathbf{x}) = \begin{cases}
\int_A  p^\partial(\mathbf{r},\mathbf{y}; \mathbf{s},\mathbf{x}) \mathbf{1}_{E^\partial _{\mathbf{s},\mathbf{x}}}(\mathbf{r},\mathbf{y})
 \,d(\mathbf{r},\mathbf{y}), & (\mathbf{s},\mathbf{x}) \in \Gamma^*, \\
\int_Ap^{(0)} (\mathbf{r},\mathbf{y}; \mathbf{s},\mathbf{x})  \mathbf{1}_{E^{(0)}_{\mathbf{s},\mathbf{x}}}(\mathbf{r},\mathbf{y}) \,
d(\mathbf{r},\mathbf{y}) , & (\mathbf{s},\mathbf{x}) \in E.
\end{cases}
\end{equation}
Since each particle carries an independent Poisson-$\beta$ clock $\beta$,
the first time $T$ that a clock rings follows the distribution $T  \sim \min_{1\le
i\le |\mathbf{s}|} \mathrm{Poisson}(\beta) = \mathrm{Poisson}(|\mathbf{s}|\beta)$.
Thus 
\begin{equation}
\label{eq:rates}
\lambda(\mathbf{s},\mathbf{x}) = \beta |\mathbf{s}|.
\end{equation} 
This completes the description of the $M$-species model as a PDMP.

\subsection{Conservation of total area and zero polyhedral defect}
A benefit of using a finite particle system for grain boundary coarsening
is   the  conservation of
area and zero polyhedral defect.  Using the notation presented above, we
may write area of polyhedral defect in an $N$ particle system as function
$A, P : E \rightarrow
[0,\infty)$ given by
\begin{equation}\label{finiteareapoly}
 A^N[(\mathbf{s},\mathbf{x})] =
\sum_{i = 1}^{|\mathbf s|} x_i \quad P^N[(\mathbf{s},\mathbf{x})] = \sum_{i = 1}^{|\mathbf
s|} (s_i-6). \quad
\end{equation} 
Here, we used the identity function  $id:x \mapsto x$.  Zero polyhedral defect
for a trivalent planar network means
that a grain
has, on average, six sides, which follows from (\ref{avgsides}) for networks
evolving on a torus.

\begin{theorem}  \label{thm:finitecons}
For the PDMP model with fixed parameters from Sect. \ref{sec:mspecgrain},
suppose we have initial polyhedral defect $P^N(0) =
0$ and total area $A^N(0) = A$, for all times $t$ where the process is well-defined
(i) $P^N(t) = 0$ and (ii) $A^N(t) = A.$
\end{theorem}

\begin{proof} 

We consider a well-defined path $(\mathbf s(t), \mathbf x(t))$ for $t \in
[0,T]$. In all realizations,  this path will have a finite set of  jump times
$\tau_1\le \dots\le \tau_n$. 

To show conservation of 
zero polyhedral defect,  suppose for a state $(\mathbf{s},\mathbf{x})$ that $(\mathbf
r, \mathbf y) \in E^\partial _{(\mathbf{s},\mathbf{x})}$. We will directly show that defect
does not change over jumps, or that  
\begin{equation}\label{pdiff}
P^N[(\mathbf{s},\mathbf{x})]=  P^N[(\mathbf
r, \mathbf y)]
\end{equation}
 in the case of a three-sided grain vanishing (other critical events have
similar proofs). Under a reindexing, we may write our state as 
\begin{equation}
(\mathbf{s},\mathbf{x}) = ((3, s_2, \dots,s_{|\mathbf s|}), (0, x_2, \dots, x_{|\mathbf
s|})).
\end{equation}
We may assume, without loss of generality, From (\ref{startrules})-(\ref{endrules})
three particles with indices 2,3, and 4 lose an edge from The mutated state
then takes the form
\begin{equation}
(\mathbf
r, \mathbf y)  = (s_2-1, s_3-1, s_4-1, \dots, s_{|\mathbf s|}), (x_2, \dots,
x_{|\mathbf
s|})),
\end{equation}
from which (\ref{pdiff}) follows immediately. This implies that $\Delta P^N[(\mathbf{s}(\tau_i),\mathbf{x}(\tau_i))]
= 0$.  Since $\mathbf s$ does not change between any jumps, if $P^N(\mathbf{s}(0),\mathbf{x}(0))
=0$, then $P^N(\mathbf{s}(t),\mathbf{x}(t))= 0$ for all times $t$ in which the PDMP is
well defined.

To show conservation of total area, again assume zero initial polyhedral
defect.  Then it is immediate that $\Delta A^N[(\mathbf{s}(\tau_i),\mathbf{x}(\tau_i)
] = 0$, and for $t \in (\tau_i, \tau_{i+1})$ for $i = 1, \dots, n-1$,  
 \begin{equation}
 \frac{\partial A^N}{\partial t} =\sum_{i = 1}^{|\mathbf s|}  \frac{\partial
A^N}{\partial x_i}  \frac{\partial x_i}{\partial t}  = \sum_{i = 1}^{|\mathbf
s|} (s_i(t)-6) = P^N[(\mathbf{s}(t),\mathbf{x}(t))] = 0.  
 \end{equation} 
 
 \end{proof}

\section{Proof of well-posedness}
\label{sec:wp}
\begin{figure}
\includegraphics[width=\textwidth]{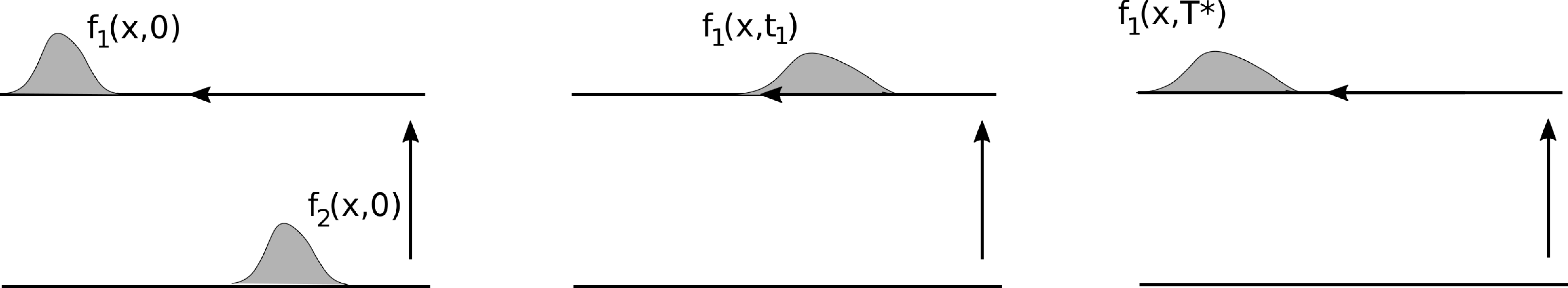}
\caption{\textbf{Limiting equations for a two species example with} $T^*<\infty$.
The PDMP is characterized by $v_1 = -1, v_2 = 0, K^{(1)} = 1, R_{21}^{(1)}
= 1,$ and weights $w_2^{(1)} = 1, w_1^{(1)} = 0$. Left: Initial densities
on the two species with disjoint supports and  $F_1 (0)= F_2(0) = 1/2.$ Center:
As the initial density of species 1 is transported to the origin, species
2 mutates to species 1. At some time $t_1$,  all of species 2 have mutated,
so that $f_2(x,t_1) = 0$. Right: the density is transported until  it reaches
the origin at time $T^*$, at which point mutation probabilities are undefined.}
\label{fig:counterpdmp}  
\end{figure}

Theorem~\ref{thm:wp} is proved in the following lemmas. The structure of
the kinetic equations is a little more transparent when the flux is rewritten
as a matrix vector product. Let $f=(f_1, \ldots, f_M)$ and ${\bm \jmath} =(j_1,
\ldots, j_M)$. We may then write 
\begin{equation}
\label{eq:flux-matrix}
{\bm \jmath} = \left( \sum_{l=1}^l A^{(l)} \dot{L}_l  + \beta \gamma(t) \, A^{(0)}
\right) f, 
\end{equation}
where the matrices $A^{(l)}$ and  $A^{(0)}$ have off-diagonal terms given
by
\begin{equation}
\label{eq:defA}
A^{(l)}_{s,\sigma} = J^{(l)}_{s,\sigma} W^{(l)}_s, \quad A^{(0)}_{s,\sigma}
= J^{(0)}_{s,\sigma} w^{(0)}_s, \quad \sigma \neq s, 
\end{equation}
and diagonal terms given by
\begin{equation}
\label{eq:defB}
A^{(l)}_{\sigma,\sigma} = - K^{(l)} W^{(l)}_\sigma, \quad A^{(0)}_{\sigma,\sigma}
= - K^{(0)}_{\sigma,\sigma} w^{(0)}_\sigma. 
\end{equation}
 
We first show that the flux ${\bm \jmath}$, defined in \ref{eq:flux-matrix}, 
is a locally Lipschitz map. This allows us to obtain local existence of positive
mild solutions by Picard's method. We then extend the solutions to a maximal
interval of existence by utilizing a more careful estimate of the flux. 

Let  $B_r(f_0) \subset X$ denote the ball of radius $r>0$ centered at $f_0
\in X$. As in \ref{eq:def-F} we denote
\[ F_0 = \sum_{\sigma=1}^M \int_0^\infty f_{0,\sigma}(x). \]

We adopt the following convention in the proof. The letter $C$ denotes a
universal, positive, finite constant depending only on the parameters of
the model such as the number of species $M$, the constant velocities $v_\sigma$,
the number of mutations $K^{(l)}$ and $K^{(0)}$, the  mutation matrices $R^{(l)}$
and $R^{(0)}$, the weights $w^{(l)}$ and $w^{(0)}$. It does not depend on
$f_0$. 
\begin{lemma}[Uniform bounds]
\label{le:un-bound}     
Assume $f_0 \in X$ is positive and non-zero. There exists $r>0$, depending
only on $f_0$, such that for each $f \in B_r(f_0)$.
\begin{equation}
\label{eq:flux-bound1}
\| {\bm \jmath} (f) \| \leq C \left( \beta + \frac{\|f_0 \|}{F_0}  \right) \|f\|.
\end{equation}
\end{lemma}
\begin{proof}
Recall that the flux ${\bm \jmath}(f)$ is defined by equations \ref{eq:flux-matrix}--\ref{eq:defA}.
We will estimate each term in this expression in turn.

We first estimate $\dot{L}$. We find from  \ref{eq:fluxdef} that for every
$l \in S_-$
\begin{equation}
\label{eq:boundL}
|\dot{L}_l| \leq |v_l| |f_l(0)| \leq \left(\max_\sigma |v_\sigma| \right)
\|f_l \|_{L^\infty} \leq C \|f_0\|.
\end{equation}

In order to estimate the weights $W^{(l)}_k$ defined by \ref{eq:numbers-not2},
we first establish a lower bound on the denominator $\sum_{n=1}^M w^{(l)}_n
F_n$ for each $f \in B_r(f_0)$. Let 
\[ \underline{w} = \min_{\sigma,l} \{w^{(l)}_\sigma: F_{0,\sigma}>0\}, \qquad
\overline{w} = \max_{\sigma,l} w^{(l)}_\sigma. \]
We then have 
\begin{eqnarray}
\nonumber     
\lefteqn{ \sum_{n=1}^M w^{(l)}_n F_n = \sum_{n=1}^M w^{(l)}_n (F_n  - F_{0,n}+
F_{0,n})  \geq \underline{w} F_0  - \sum_{n=1}^M w^{(l)}_n |F_n -F_{0,n}|
}\\
\label{eq:lowerW}
&&
\geq \underline{w} F_0  - \sum_{n=1}^M w^{(l)}_n \|f_n -f_{0,n}\|_{L^1} \geq
\underline{w} F_0 - \overline{w} \|f-f_0\| \geq \frac{1}{2}\underline{w}
F_0,
\end{eqnarray}
provided the radius $r$ satisfies
\begin{equation}
\label{eq:boundr}
 r <  \frac{\underline{w}}{2\overline{w}}F_0.
 \end{equation}
We assume that $r$ is chosen as above. It then follows from \ref{eq:defA}
and \ref{eq:defB} that each entry in the matrix $A^{(l)}$ is bounded above
by
\begin{equation}
\label{eq:A-op0}
|A_{\sigma k}| \leq \frac{C}{F_0}.
\end{equation}
Thus, the operator norm $\|A^{(l)} \|_o$ of the matrix $A^{(l)}$  
satisfies the estimate\footnote{ $\|A\|_o := \sup_{|v|=1}|Av|$, with $v \in
\mathbb{R}^M$, $|v|^2 = \sum_{n=1}^M v_n^2$. Since $M$ is finite any norm may be
chosen.} 
\begin{equation}
\label{eq:A-op}
\| A^{(l)}\|_o \leq \frac{C}{F_0}.
\end{equation}
We combine \ref{eq:A-op} with \ref{eq:boundL} to see that the flux due
to boundary events is bounded by
\begin{equation}
\label{eq:j-est}
\left\| \left( \sum_{l=1}^{M_-} A^{(l)} \dot{L}_l \right) f \right\| \leq
C\frac{\|f_0\|\|f\|}{F_0}.
\end{equation}

The estimates for the interior events are simpler. We use the definition
of  $\gamma$ in \ref{eq:numbers-not2} and the lower bound \ref{eq:lowerW}
to obtain the estimate
\begin{equation}
\label{eq:gamma-up}
0 \leq \gamma \leq C\frac{F}{F_0} \leq C, \quad f \in B_r(f_0).
\end{equation}
It follows from the  definition of $A^{(0)}$ in \ref{eq:defA}--\ref{eq:defB}
that $\|A^{(0)}\|_o \leq C$. Thus, the flux from interior events is bounded
by
\begin{equation}
\label{eq:int-flux}
\| \beta \gamma  A^{(0)} f \| \leq C \beta \|f\|.
\end{equation} 
We combine estimates \ref{eq:j-est} and \ref{eq:int-flux} to complete the
proof.
\end{proof}
\begin{lemma}[Lipschitz estimate] 
\label{le:loc-lip}      
Let $f_0$ and $r$ be as in Lemma~\ref{le:un-bound}. Then for every $f,g \in
B_r(f_0)$       
\begin{equation}
\label{eq:flux-bound2}
\| {\bm \jmath}(f)  -{\bm \jmath} (g)\| \leq C \left( \beta + \frac{\|f_0\|}{F_0}\right)
\left( 1+ \frac{\|f_0\|}{F_0} \right) \|f-g\|.
\end{equation}
\end{lemma}     
\begin{proof}
We use the expression \ref{eq:flux-matrix} to obtain the inequality
\begin{eqnarray}
\lefteqn{\| {\bm \jmath}(f) -{\bm \jmath}(g) \|} \\
\nonumber
&&  \leq \sum_{l=1}^{M_-} \| A^{(l)}(f)\dot{L}_l(f) f - A^{(l)}(g)\dot{L}_l(g)
g \| + \beta \| \gamma(f) A^{(0)} f - \gamma(g) A^{(0)}  g\|. 
\end{eqnarray}
Let $l$ be fixed. It is clear that 
\begin{equation}
\label{eq:L-diff}
| \dot{L}_l(f) - \dot{L}_l(g)| = |v_l| |f(0)-g(0)| \leq C \|f-g\|. 
\end{equation}
For each $k$, the difference $\left|W^{(l)}_k(f)-W^{(l)}_k(g)\right|$ is estimated
as follows. Let $G_n = \int_0^\infty g_n(x) \, dx$ and $G = \sum_{n=1}^M
G_n$. Then
\begin{eqnarray}
\label{eq:W-diff}
\left| \frac{w^{(l)}_k}{\sum_{n=1}^M w^{(l)}_n F_n} - \frac{w^{(l)}_k}{\sum_{n=1}^M
w^{(l)}_n G_n}\right| && =   \frac{w^{(l)}_k \left| \sum_{n=1}^M w^{(l)}_n (F_n
-G_n)\right|} {\left|\sum_{n=1}^M w^{(l)}_n F_n\right| \left| \sum_{n=1}^M
w^{(l)}_n G_n\right|} \\
\nonumber
&& \leq \frac{C}{F_0^2}  \sum_{n=1}^M \| f_n -g_n \|_{L^1} \leq  \frac{C}{F_0^2}
\|f-g\|,
\end{eqnarray} 
using \ref{eq:lowerW}. It then follows from \ref{eq:defA} and \ref{eq:defB}
that each term in the matrix $A^{(l)}(f)-A^{(l)}(g)$ satisfies an estimate
as above, so that 
\begin{equation}
\label{eq:op-diff}
\| A^{(l)}(f) - A^{(l)}(g) \|_o \leq \frac{C}{F_0^2} \|f-g\|.
\end{equation}
Finally, we use the estimates \ref{eq:boundL}, \ref{eq:A-op}, \ref{eq:L-diff}
and \ref{eq:W-diff}  to obtain the Lipschitz bound:
\begin{eqnarray}
\nonumber     
\lefteqn{\| A^{(l)}(f)\dot{L}_l(f) f - A^{(l)}(g)\dot{L}_l(g) g \| \leq \|
A^{(l)}(f) - A^{(l)}(g) \|_o |\dot{L}_l(f)| \|f\|} \\
&& \nonumber
 + \|A^{(l)}(g)\|_0 |\dot{L}_l(f)  - \dot{L}_l(g)| \|f\| + \|A^{(l)}(g)\|_o|L_l(g)|
\|f-g \|\\
\nonumber
&& 
\leq C \left( \frac{\|f_0\|^2}{F_0^2} + \frac{\|f_0\|}{F_0} \right) \|f-g\|.
\end{eqnarray} 
A calculation similar to \ref{eq:gamma-up}  yields the estimate 
\begin{equation}
\label{eq:int-flux-lip1}
\|\gamma(f)-\gamma(g)\|  \leq \frac{C}{F_0}|F-G| \leq   \frac{C}{F_0}\|f-g\|.
\end{equation}
Thus, we find (also using the fact that $A^{(0)}$ is a constant)
\begin{equation}
\beta \| \gamma(f)A^{(0)} f -  \gamma(g)A^{(0)} g \| \leq C\beta \left(
1+ \frac{\|f_0\|}{F_0} \right) \|f-g\|.
\end{equation}

\end{proof}     
\begin{lemma}[Local existence]
Assume $f_0 \in X$ is positive and non-zero. There exists a time $T_0>0$
and a map $f \in C([0,T];X)$ such that $f$ is the unique mild solution to
\ref{eq:kinetic} on the time interval $[0,T]$ that satisfies the initial
condition $f(0)=f_0$.  

Further, $f(t)$ is positive for each $t \in [0,T]$.
\end{lemma}
\begin{proof}
Let $r(f_0)$ be chosen as in Lemma~\ref{le:un-bound}. It then follows from
Lemma~\ref{le:loc-lip} that the flux ${\bm \jmath}(f)$ is locally Lipschitz. The
existence of a unique mild solution now follows by a standard application
of the contraction mapping theorem. 

The fact that the solution preserves positivity is seen as follows. We note
that the loss term in \ref{eq:jminus}, may be rewritten as $j_\sigma^-(x,t)
= \alpha_\sigma(t) f_\sigma(x,t)$ where
\begin{equation}
\label{eq:def-alpha}
\alpha_\sigma(t) = \sum_{l = 1}^{M_-} \dot{L_l} K^{(l)} W_\sigma^{(l)}(t) +
 \beta \gamma(t) K^{(0)}  w^{(0)}_\sigma.
\end{equation}
We now rewrite the kinetic equation \ref{eq:kinetic} in the form 
\begin{equation}
\label{eq:pres-pos}
\partial_t f_\sigma + v_\sigma \partial_x f_\sigma + \alpha_\sigma(t) f_\sigma
= j_\sigma^+,
\end{equation}
and observe that integration along characteristics yields 
\begin{equation}
\label{eq:uq2}
f_\sigma(x,t) = e^{-\int_0^t \alpha(s) \, ds} f_\sigma(x-v_\sigma t, 0) +
\int_0^t e^{-\int_\tau^t \alpha(s) \, ds} j_\sigma^+ \left(x-v_\sigma(t-\tau),\tau\right)
\, d\tau,
\end{equation}
which clearly preserves positivity.
\end{proof}

\begin{lemma}[Maximal existence]
Let $f \in C([0,T];X)$ be a positive, mild solution. Then 
\begin{equation}
\label{eq:cons-num}     
F(t) + \sum_{l=1}^{M_-} L_l(t) = F(0), \quad L_l(t) := |v_l| \int_0^t f_l(0,s)
\, ds, \quad t \in [0,T].
\end{equation}
There also exists a universal constant $C >0 $ such that 
\begin{equation}
\label{eq:infinity-bound}       
\|f (t) \|_{L^\infty} \leq \|f(0)\|_{L^\infty} \exp\left( C\int_0^t\Phi(\tau)d\tau\right)
, \quad t \in [0,T],
\end{equation}
where $\Phi(t) = t+\max_{l \le M_-}\sum W_k^{(l)}(t)$. 
\end{lemma}             
Equation \ref{eq:cons-num} expresses conservation of the total number density
of the system. The  bound \ref{eq:infinity-bound}       degenerates if and
only $F(t) \to 0$, i.e. if and only if $F_p(t) \to 0$ for each $p=1, \ldots,
M$, as $t$ approaches a critical time, say $T_*$. It is well-known that continuous
mild solutions on an interval $[0,T]$ can be uniquely continued onto a maximal
interval of existence $[0,T_*)$, such that $\lim_{t \to T_*} \|f(t)\|_X =
+\infty$. Thus, the above estimates suffice to complete the proof of Theorem~\ref{thm:wp}.

\begin{proof}
{\em 1. \/} The conservation of number for the kinetic equations is a consequence
of the switching rules for the particle system. We use the identity \ref{eq:Jequality},
and the definition of the fluxes in equations \ref{eq:jplus} and \ref{eq:jminus}
to obtain the identity
\begin{equation}
\label{eq:j-identity}
\sum_{\sigma=1}^M j_\sigma =0.
\end{equation}
It follows from \ref{eq:kinetic} and \ref{eq:j-identity} that 
\begin{equation}
\sum_{\sigma=1}^M \partial_t f_\sigma  + v_\sigma \partial_x f_\sigma =0.
\end{equation}
We integrate over $x \in [0,\infty)$ to obtain the identity\begin{equation}
\label{eq:dotF}
\frac{dF}{dt} = \sum_{\sigma=1}^M v_\sigma f_\sigma(0,t) = -\sum_{l=1}^{M_-}
\dot{L}_l.
\end{equation}
The integral form of this identity is \ref{eq:cons-num}.

{\em 2.\/} In order to prove \ref{eq:infinity-bound} we  combine equations
\ref{eq:defA} and \ref{eq:defB}  to obtain the pointwise estimate 
\begin{equation}
\label{eq:pointwise}
\| A^{(l)}(t)  \|_o \leq C\sum_{k = 1}^MW^{(l)}_k(t) , \quad t \in [0,T].
\end{equation}
Consequently, the flux due to boundary events satisfies the $L^\infty$ estimate
\begin{align}
\label{eq:flux-infinity}
\left\| \sum_{l=1}^{M_-} A^{(l)} \dot{L}_l f(t)\right\|_\infty &\leq \sum_{l=1}^{M_-}
\| A^{(l)}(t)  \|_o \dot{L}_l \|f(t)\|_\infty\\
 &\leq  C\left(\max_{l \le M_-}\sum_{k = 1}^{M} W_k^{(l)}(t) \right) \|f(t)\|_\infty
.
\end{align}

The flux due to interior events is controlled in a similar manner. As in
\ref{eq:gamma-up} we find $\gamma(t) \leq C$, $t \in [0,T]$. Since $\|A^{(0)}
\|_o \leq C$, we find
\begin{equation}
\label{eq:intestimate}
\|\beta \gamma(t) A^{(0)} f (t)\|_\infty \leq C \beta \| f(t) \|_\infty.
\end{equation}
We combine \ref{eq:flux-infinity} and \ref{eq:intestimate} to obtain
\begin{equation}
\label{eq:flux-fin}
\| {\bm \jmath}(t) \|_\infty \leq C \Phi(t)\|f(t)\|_\infty . 
\end{equation} 

We now substitute these $L^\infty$ estimates in the solution formula \ref{eq:uq1}
 to obtain
\begin{eqnarray}
\nonumber
\lefteqn{\| f(t) \|_\infty = \sum_{\sigma=1}^M \|f_\sigma(\cdot,t) \|_\infty
\leq \sum_{\sigma=1}^M \|f_\sigma(\cdot,0)\|_\infty + \int_0^t \|j_\sigma(\cdot,
\tau) \|_\infty \, d\tau} \\
\label{eq:pre-gron} &&
\stackrel{\ref{eq:flux-fin}}{\leq}  \|f(0)\|_\infty + C \int_0^t \Phi(\tau)
 \|f(\tau)\|_\infty\, d\tau.
\end{eqnarray}
An application of Gronwall's lemma yields \ref{eq:infinity-bound}. \end{proof}

From here, it is straightforward to obtain several  regularity properties
for mild solutions.
\begin{theorem} \label{thm: reg} The following hold for mild solutions:
\begin{enumerate}
\item  Let $f(x,t; f_0)$ denote  the unique mild solution with initial condition
$f_0$ for $t <T_*(f_0)$. Then,    the map defined by  $(f_0,t) \mapsto f(x,t;f_0)$
is continuous in both  time and space  at all points where the map is well-defined.
\item The maximum time of existence $T^*$ for mild solutions depends continuously
on initial
conditions.  
\item Let $f_0 \in X \cap C^1([0,\infty)^M)$ with $f'_0 \in X $. Then the
mild solution $f(x,t;f_0)$ is differentiable in both space and time for $t
\in [0,T^*)$, and satisfies
(\ref{eq:kinetic}).   
\end{enumerate}

\end{theorem}

\begin{proof}
We have already shown continuity in $t$ from a contraction mapping theorem.
We now show  continuous dependence on initial conditions. To show continuity
in space, let $f_0^1, f_0^2 \in X$ with mild solutions $f_1(x,t)$ and $f_2(x,t)$.
   For $f_0^2 \in X$ sufficiently close to $f_0^1$, we may use  (\ref{eq:uq1})
and  (\ref{eq:flux-bound2}) to show there exists $t^*>0$ such that for  

$y(t) =  \|f^1(t) - f^2(t)\|$, 
\begin{equation}
y(t) \le  y(0) +C(f_0)\int_0^t y(s)ds, \quad t \in [0,t^*).
\end{equation}
An application of Gronwall's inequality implies continuous dependence  up
to time $t^*$. By a standard compactness argument, this  may be extended
to any $[0,T]$  with  $T<T_*(f_0)$. This shows part (1), from which   (2)
follows immediately. 

To prove (3), we show differentiability in space,  using the (without loss
of generality) right difference quotients \begin{equation}
 y_h(t) =  \|D_h^+(f(x,t))\|_\infty =\left\| \frac {f(x+ h,t)-f(x)}{h}\right\|_\infty
\end{equation}
for $h>0$. Similar to Part (1), through (\ref{eq:uq1}) and (\ref{eq:flux-bound2})
we may show that for  
 $t \in [0,T]$ with $T<T^*$, there exists a finite constant $C(T,f_0) <\infty$
\begin{equation}
y_h(t) \le C(T, f_0)\|D_h f_0(x)\|_\infty.
\end{equation} 
Since $f'_0 \in L^\infty$,  we may take  $h\rightarrow 0$ to obtain $\|\partial_x
f(x,t)\|_\infty < \infty.$ The argument for time derivatives is identical.

\end{proof} 
 
 For the purposes of  demonstrating global existence for PDMPs related to
grain coarsening in Section \ref{sec:grainprops}, we mention that characteristic
speeds are bounded by $\bar v = \max_{s \le M} v_s$.   
   
 \begin{theorem}[Finite speed of propagation] \label{thm:fsp} If $f_0$ has
support contained in $[0,L]$, then for $0<t<T_*$,  the support $f(x,t)$ is
contained in $[0, L+\bar vt]$.  
 \end{theorem}
 
 \begin{proof}
The argument is similar to Theorem \ref{thm: reg}, but now using    (\ref{eq:uq1})
and Gronwall's inequality applied to 
\begin{equation}
g_\sigma(t) = \sup_{x>L+(t-T)\bar v} f_\sigma(x,t), \quad \sigma = 1, \dots,
M.
\end{equation} 
 \end{proof}

\subsection{Properties of kinetic equations for grain coarsening} \label{sec:grainprops}

\subsubsection{Conserved quantities}

Conservation of area and zero polyhedral defect for kinetic equations, defined
similarly from (\ref{finiteareapoly}), are conserved for the kinetic equations
when considering initial data which is compactly supported.
\begin{theorem}
Let $f_0 \in X \cap C_c([0,\infty))$. Then
for $t \in [0,T^*)$,
 
 \begin{equation}\label{conserve}
P(t):=\sum_{n
= 2}^M(n-6)F_n(t),\quad  A(t)
:= \sum_{n = 2}^M\int_0^\infty a\cdot f_n(a,t) da,
\end{equation}
if $P(0) = 0$ and $A(0) = A$, then $P(t) = 0$ and  $A(t) = A$ for all $t>0$.
\end{theorem}

\begin{proof}

 Suppose
$f_0^j \in X$ are differentiable with compact support for $j\in\mathbb N$,
and $f_0^j\rightarrow f_0$ in $X$ as $j\to\infty$. By continuous dependence
of parameters, for all  
$t>0$, solutions $f^j(a,t) \rightarrow f(a,t)$ in $X$. Thus it is sufficient
to show (\ref{conserve}) for classical solutions. To do so, we  integrate
(\ref{grainpde}) and sum over
all species
to obtain
\begin{equation}\label{sumgrainpde}
\sum_{n = 2}^M (n-6)F_n(t) = -\sum_{n = 2}^{M_-} (n-6)^2f_n(0,t)+\sum_{n
= 2}^M (n-6)
\int_0^\infty j_n(a,t)da. 
\end{equation}
 The left hand side of  (\ref{sumgrainpde}) is the polyhedral defect, and
a  straightforward computation using (\ref{hgrainbegin})-(\ref{hgrainend})
shows that
\begin{equation}
\sum_{n = 2}^M (n-6)
\int_0^\infty j_n(a,t)da = \sum_{n = 2}^{M_-} (n-6)^2f_n(0,t), 
\end{equation}
which shows the conservation of polyhedral defect. 

To show the conservation of area, we find through  an integration by parts
of (\ref{grainpde}) with  (\ref{eq:j-identity}) that
\begin{align}
\frac{dA(t)}{dt} &= \sum_{n = 2}^M (n-6)\int_0^\infty  a \cdot\partial_x
f_n(a,t)da+
\int_0^\infty a\sum_{n=2}^M j_n(a,t)da\\
&= P(t).
\end{align}
For initial conditions with zero polyhedral defect, conservation of area
then follows.

\end{proof}

The conservation of total area is sufficient to show global existence under
a wide choice of weights.  

\begin{theorem}[Global existence]
 Suppose $M \neq 10$, and   $w_k^{(l)}>0$, for $(k,l) \in \{2, \dots M\}
\times\{2\times M_-\}\backslash\{\cup_{i  \in \{0,2,3,4,5\}} (2,i), (3,2),
(M,5), (M,0)\}.$ For nonzero initial conditions with zero polyhedral defect,
 the maximum interval of existence for mild solutions is infinite. \end{theorem}

\begin{remark}
We require $M\neq10$ due to the impossibility of edge deletion in the pathological
case of initial conditions  $F_2(0) = F_{10} (0) = 1/2$, and $ F_\sigma =
0$ for $ \sigma \notin \{2,10\}$.
\end{remark}

\begin{proof}
Suppose for the sake of contradiction, that a finite  maximum interval of
existence  $T_*<\infty$.  From Theorem \ref{thm:wp},  
$\sum_{\sigma=2}^M w^{(l)}_\sigma F_\sigma(T_*^-) = 0$ for  some $l  = 1,
\dots,
M^-$. We can check directly that from the conditions on $w_k^{(l)}$, using
zero polyhedral defect, that  the stronger condition of $F_\sigma(T_*^-)
= 0$ must hold for all $\sigma = 2, \dots, M$. Since $f(x,t) \in C_c^1(\mathbb{R}_+)$
from Theorem \ref{thm:fsp},  there exists $L>0$ such that $f(x,t) = 0$ for
$x>L$ and $0<t<T_*$.  This implies, however,
that
\begin{equation}
A(T_*^-) =\sum _{\sigma = 2}^M \int_0^\infty a\cdot f_\sigma(a,T_*^-)da 
\le L\sum _{\sigma = 2}^M F_\sigma(T_*^-) = 0,
\end{equation}
a contradiction to the conservation of total area.
\end{proof}

\bibliographystyle{spmpsci}
\bibliography{kmp}

\end{document}